\documentclass[preprint,12pt]{elsarticle}
\journal{ }
\date{}

\usepackage[version=4]{mhchem}

\usepackage{graphicx, amsmath, amsthm,amsfonts, amssymb}
\usepackage{hyperref}
\usepackage{subfigure}
\usepackage{algorithm,algorithmic}

\usepackage{mathtools,bm}
\usepackage{multirow,booktabs,makecell}
\usepackage{tikz,url}
\usetikzlibrary{shapes,arrows}
\tikzstyle{process} = [rectangle,minimum width=2cm,minimum height=1cm,text centered,text width =4cm,draw=black]

\newcommand{\bmxi}{{\xi}}

\newcommand{\bmb}{\bm{\beta}}
\newcommand{\bmr}{\bm{\rho}}
\newcommand{\bms}{\bm{s}}
\newcommand{\bmp}{\bm{\Phi}}

\newcommand{\bmm}{\bm{m}}
\newcommand{\bmu}{\bm{u}}
\newcommand{\bma}{\bm{\alpha}}

\newcommand{\Th}{\mathcal{T}_h}

\DeclareMathOperator*{\argmin}{argmin}

\newtheorem{theorem}{Theorem}[section]
\newtheorem{proposition}{Proposition}[section]

\newtheorem{remark}[theorem]{Remark}

\numberwithin{equation}{section}

\begin{document}
\begin{frontmatter}

\title{High order spatial discretization for variational time implicit schemes: Wasserstein gradient flows and reaction-diffusion systems}
\author[1]{Guosheng Fu\fnref{fn1}}
\ead{gfu@nd.edu}


\author[2]{Stanley Osher\fnref{fn2}}
\ead{sjo@math.ucla.edu}

\author[3]{Wuchen Li\fnref{fn3}}
\ead{wuchen@mailbox.sc.edu}

\cortext[cor]{Corresponding author}
\fntext[fn1]{G. Fu's work is supported in part by NSF DMS-2134168.}
\fntext[fn2]{S. Osher's work is supported in part by AFOSR MURI FP 9550-18-1-502, and ONR
grants:  N00014-20-1-2093, and N00014-20-1-2787.}

\fntext[fn3]{W. Li's work is supported by AFOSR MURI FP 9550-18-1-502, AFOSR YIP award 2023, and NSF RTG: 2038080.}

\affiliation[1]{organization={Department of Applied and Computational Mathematics and Statistics, University of Notre Dame},
                city={Notre Dame},
                postcode={IN 46556},
                country={USA}}
\affiliation[2]{organization={Department of Mathematics, University of California, Los Angeles},
                city={Los Angeles},
                postcode={CA 90095},
                country={USA}}     \affiliation[3]{organization={Department of Mathematics, University of South Carolina},
                city={Columbia},
                postcode={SC 29208},
                country={USA}}          

\begin{abstract}
We design and compute first-order implicit-in-time variational schemes with high-order spatial discretization for initial value gradient flows in generalized optimal transport metric spaces. 
We first review some examples of gradient flows in generalized optimal transport spaces from the Onsager principle. We then use a one-step time relaxation optimization problem for 
time-implicit schemes, namely generalized Jordan-Kinderlehrer-Otto schemes. Their minimizing systems satisfy implicit-in-time schemes for initial value gradient flows with first-order time accuracy. We adopt the first-order optimization scheme ALG2 (Augmented Lagrangian method) 
and high-order finite element methods in spatial discretization
to compute the one-step optimization problem. This allows us to derive the implicit-in-time update of initial value gradient flows iteratively. We remark that the iteration in ALG2 has a simple-to-implement point-wise update based on optimal transport and Onsager's activation functions. The proposed method is unconditionally stable for convex cases. 
Numerical examples are presented to demonstrate the effectiveness of the methods in two-dimensional PDEs, including Wasserstein gradient flows, Fisher--Kolmogorov-Petrovskii-Piskunov equation, and two and four species reversible reaction-diffusion systems. 
\end{abstract}

\end{frontmatter}


\section{Introduction}
\label{sec:intro}
Dissipative dynamics (gradient flows) are essential models in thermodynamics, chemistry, materials science, biological swarming, robotics path panning, and social sciences \cite{doi2011onsager,peletier2014variational}. Nowadays, they also find vast applications in designing machine learning optimization algorithms and Markov-Chain-Monte-Carlo sampling algorithms \cite{Amari, ChenLi2020_optimala, weinan2020machine, GaoLiLiu,Garbuno-InigoHoffmannLiStuart2020_interactinga, LiMontufar2018_naturalb,LinLiOsherMontufar2021_wassersteina,WangLi2022_accelerated}. In physics, dissipative dynamics describe that the systems have maximum efficiency, in which dynamics follow from the direction in which the (negative) entropy/Lyapunov functional dissipates most rapidly. It turns out that the dissipative dynamics are gradient flows in suitable metric spaces. Fast, efficient, and accurate dissipative dynamics simulations are one of the central problems in computational fluid dynamics. 

A particular type of gradient flow has been widely studied in optimal transport, where the metric is known as the Wasserstein-2 metric \cite{2005_gradienta,villani2008optimal}. Typical examples include gradient drift Fokker-Planck equations, porous media equations, aggregation-diffusion equations, etc. One property of simulating gradient flows is that one can design a proximal method for computing a variational implicit time algorithm. This algorithm is first proposed by Jordan-Kinderlehrer-Otto (JKO scheme) to compute Wasserstein gradient flows \cite{JKO}.  Moreover, general gradient flows have been widely studied. They follow the Onsager principle to design optimal transport-type metric spaces \cite{GaoLiLiu, LiLeeOsher22, Mielke11}. Similarly, one can develop variational proximal methods to compute and simulate gradient flow dynamics. 

This paper designs high-order spatial discretization in simulating gradient flow dynamics using variational proximal schemes in generalized optimal transport metric spaces. We formally illustrate the main computational framework. Consider a reaction-diffusion type equation:  
\begin{equation}\label{GD}
  \partial_t \rho=\nabla\cdot(V_1(\rho)\nabla \frac{\delta}{\delta \rho}\mathcal{E}(\rho))-V_2(\rho)\frac{\delta}{\delta\rho}\mathcal{E}(\rho). 
\end{equation}
where $\rho\colon \Omega\times \mathbb{R}_+\rightarrow \mathbb{R}_+$ is a scalar density function, $\Omega\subset \mathbb{R}^d$, $d=1, 2$, is a spatial domain with periodic or Neumann boundary conditions, $V_1$, $V_2\colon \Omega\times \mathbb{R}_+\rightarrow\mathbb{R}_+$ are positive mobility functions (Onsager activation functions), and $\mathcal{E}(\rho)\in\mathbb{R}$ is a Lyapunov functional (energy). 
We design a variational implicit time scheme, the linearized JKO scheme \cite{CancesGallouetTodeschi2020_variational,LiLuWang20}, to update equation \eqref{GD} as below:
\begin{subequations}\label{G-JKO}
\begin{align}
\label{G-JKO1}
\rho^{n}=\arg\min_{\rho} \underset{(\rho, m)}{\inf} \quad \frac{1}{2\Delta t}\int_\Omega 
\Big[\frac{|m|^2}{V_1(\rho)}+\frac{|s|^2}{V_2(\rho)}\Big] dx+
 \,\mathcal{E}(\rho),
\end{align}
where $\Delta t\geq 0$ is a stepsize and the minimization is over all functions $\rho\colon \Omega \rightarrow \mathbb{R}_+$, $m\colon \Omega \rightarrow\mathbb{R}^d$, and $s\colon \Omega \rightarrow\mathbb{R}$, subject to the constraint
\begin{align}
\label{Gs1}
\rho-\rho^{n-1}+\nabla\cdot m = s, \text{ on }\Omega.
\end{align}
\end{subequations}
We use time rescaling of $(m,s)$ in the constraint \eqref{Gs1}. We then compute variational problem \eqref{G-JKO} iteratively to find the sequence $\rho^n$, $n=1,2,\cdots$. This sequence forms an implicit update for gradient flow dynamic \eqref{GD}, which is first-order in time:
\begin{equation*}
    \frac{\rho^n-\rho^{n-1}}{\Delta t}=\nabla\cdot (V_1(\rho^n)\nabla\frac{\delta}{\delta\rho}\mathcal{E}(\rho^n)) -V_2(\rho^n)\frac{\delta}{\delta\rho}\mathcal{E}(\rho^n)+\mathcal{O}(\Delta t).
\end{equation*}
When $V_1$, $V_2$ is concave in term of $\rho$, and $\mathcal{E}$ is a convex functional, then the proposed method is unconditionally stable, meaning that we can take large time steps. 

Our framework also works for reversible reaction-diffusion systems with detailed balance \cite{Mielke11,Glitzky13,LiuWangWang21a}. We illustrate the main idea for a simple 2-component reversible reaction-diffusion system:
Let  $X_1, X_2$ be two species with a single reversible 
reaction
$
X_1
\xrightleftharpoons[k_-]{k_+}
X_2,
$
with $k_-, k_+>0$.
Let $\rho_1$ and $\rho_2$ be the respective densities of 
$X_1$ and $X_2$.
This leads to the following PDE system
\cite{Mielke11,Peletier10}:
\begin{align*}
\partial_t\rho_1 -\gamma_1
\Delta \rho_1
=&\; -(k_+\rho_1-k_-\rho_2),\\
\partial_t\rho_2 - \gamma_2\Delta\rho_2
=&\; (k_+\rho_1-k_-\rho_2),
\end{align*}
with positive diffusion rates $\gamma_1,\gamma_2>0$.
By introducing the following mobility functions, 
\begin{align*}
V_{1,1}(\rho_1) = \gamma_1\rho_1,\quad
V_{1,2}(\rho_2) = \gamma_2\rho_2,\quad
V_{2}(\rho_1, \rho_2) = 
\frac{k_+\rho_1-k_-\rho_2}{\log(k_+\rho_1)-\log(k_-\rho_2)}
\end{align*}
and the energies
\[
\mathcal{E}_1(\rho_1)
=\int_{\Omega}\rho_1(\log(k_+\rho_1)-1)\,dx,\quad
\mathcal{E}_2(\rho_2)
=\int_{\Omega}\rho_2(\log(k_-\rho_2)-1)\,dx,
\]
the above PDE system can be recast into the following system version of the form \eqref{GD}:
\begin{subequations}
\label{rGD}    
\begin{align}
\partial_t\rho_1 =&\; \nabla\cdot\left(V_{1,1}(\rho_1)\nabla\frac{\delta\mathcal{E}_1}{\delta \rho}(\rho_1)\right)
 -V_2(\rho_1, \rho_2)
  \left(
 \frac{\delta\mathcal{E}_1}{\delta \rho}(\rho_1)-\frac{\delta\mathcal{E}_2}{\delta \rho}(\rho_2)\right),\\
 \partial_t\rho_2 =&\; \nabla\cdot\left(V_{1,2}(\rho_2)\nabla\frac{\delta\mathcal{E}_2}{\delta \rho}(\rho_2)\right)
 +V_2(\rho_1, \rho_2)
  \left(
 \frac{\delta\mathcal{E}_1}{\delta \rho}(\rho_1)-\frac{\delta\mathcal{E}_2}{\delta \rho}(\rho_2)\right),
\end{align}
\end{subequations}
which can then be discretized using a similar variational time implicit scheme as \eqref{G-JKO}.
Here the system is called a {\it strongly reversible reaction-diffusion system} when $k_+=k_->0$, 
is called a {\it reversible reaction-diffusion system with detailed balance} when we allow the two positive reaction rates to be different $k_+\not=k_->0$, and is called an {\it irreversible reaction-diffusion system} when the backward reaction rate is zero $k_-=0$; see more detailed in \cite{Glitzky13,LiuWangWang21a,Mielke11}. Our framework does not directly work for irreversible reaction-diffusion systems as they do not satisfy an energy dissipation law and can not be formulated back to the form \eqref{rGD}.
However, we can approximate an irreversible reaction-diffusion system using a reversible one by using a very small backward reaction rate (see, e.g., \cite{Liang22GS}), and then solve the reversible system using our formulation. 

This paper adopts the augmented Lagrangian (ALG2) optimization method with  high-order spatial finite element discretizations to solve the variational problem \eqref{G-JKO}. Using finite element spatial discretization, we also develop a point-wise update in the optimization step of computing variational problem \eqref{G-JKO}. In this sense, we obtain a high-order spatial discretization scheme in finding the ground state, which is the minimizer of functional $\mathcal{E}$. In this iterative procedure, assuming that the optimization step finds a global minimizer, the Lyapunov functional $\mathcal{E}$ is guaranteed to decay for any large time step sizes. 

Computational optimal transport and mean field control/games have been widely investigated in \cite{achdou2012iterative,benamou1999numerical,benamou2014augmented,LiuMaimaitiyiming2023_dynamica, papadakis2014optimal,peyre2019computational,ShenXu2021_unconditionallya}. For example, generalized JKO schemes of Wasserstein gradient flows with first-order time accuracy have been studied in \cite{CancesGallouetTodeschi2020_variational,LiWang22,GallouetMonsaingeon2017_jko,LiLuWang20,LiuWangWang21a}. Semi-discretizations of JKO-type schemes have been used in \cite{ChengShen2020_global}. The Lagrangian  type JKO schemes have been investigated in \cite{Carrillo2018,Carrillo2016,liu2020lagrangian}. 
It is also worth mentioning that there are methods for high-order time discretizations of gradient flows \cite{GongZhaoWang2020_arbitrarily}. 
Meanwhile, generalized optimal transport metric spaces have recently been introduced in \cite{CARRILLO20101273,ChowHuangLiZhou2012_fokkerplanckc,Erbar2016_gradient,Maas2011_gradientb,Mielke11}. 
Study of 
conservative and dissipative operators in non-equilibrium thermodynamics \cite{OnsagerMachlup1953_fluctuations, OttingerGrmela1997_dynamics, ZhuHongYangYong2015_conservationdissipation} is an active research area. 
However, there are limited JKO-type computational results for reaction-diffusion systems. We specifically mention the recently introduced variational operator splitting schemes \cite{LiuWangWang21a,LiuWangWang22,LiuWangWang22x} for reversible reaction diffusion systems using the energetic variational framework \cite{Giga2017, liu2009introduction}. We note that generalized JKO schemes are examples of mean field control (MFC) problems \cite{benamou2014augmented,lasry2007mean}, which design optimal control/optimization problems for general initial value evolutionary equations not limited to gradient flows. Computation and modeling studies of MFCs have been conducted in controlling reaction-diffusion equations \cite{LiLeeOsher22} and conservation laws \cite{li2021controlling, li2022controlling} with applications in pandemics modeling \cite{lee2022mean,lee2021controlling}. Compared to the above approaches, we apply high-order spatial schemes in computing generalized JKO schemes towards initial value gradient flows. We adopt the first-order optimization method, the augmented Lagrangian method (ALG2), to implement the variational time implicit schemes for two and four species-reversible reaction-diffusion systems.  

This paper is organized as follows. We review some concepts of gradient flows, time implicit schemes, and their first-order optimization methods ALG2 in section \ref{sec:formulation}. 
Several examples of dynamics, including Wasserstein gradient flows, Fisher--Kolmogorov-Petrovskii-Piskunov (KPP) equation, and reversible reaction-diffusion systems, are presented in section \ref{sec:model}. We then present a high-order finite element method and derive all implementation details of the optimization algorithm ALG2 in section \ref{sec:alg}.
Numerical examples are presented for two-dimensional Wasserstein gradient flows of linear, interaction, and potential energies, Fisher-KPP equation, and reversible two and four-species reaction-diffusion systems in section \ref{sec:num}.  

\section{Optimal transport type gradient flows, generalized time implicit schemes, and first-order optimization methods }
\label{sec:formulation}
This section reviews generalized gradient flows and their variational implicit schemes in metric spaces. We also discuss a one-step time discretization relaxation of variational implicit schemes for generalized gradient flows. Entropy dissipation properties of variational implicit schemes are introduced. 
We then present the augmented Lagrangian method (ALG2) as the optimization solver to compute the variational implicit schemes. 

\subsection{Optimal transport type gradient flows}
In this subsection, we formally review generalized optimal transport gradient flows \cite{Carrillo2016,Erbar2016_gradient,Mielke11}. This is known as the Onsager gradient flow \cite{doi2011onsager}. We next discuss a class of variational schemes to compute implicit-in-time solutions of gradient flows. 

\subsubsection{Gradient flows and entropy dissipations}
Consider an initial value equation
\begin{equation}\label{dissipative}
\begin{split}
&\partial_t \rho(x,t)=\nabla\cdot(V_1(\rho(x,t))\nabla\frac{\delta}{\delta \rho}\mathcal{E}(\rho)(x,t))-V_2(\rho(x,t))\frac{\delta}{\delta\rho}\mathcal{E}(\rho)(x,t), \quad t\in [0, \infty)\\
&\rho(x,0)=\rho^0(x). 
\end{split}
\end{equation}
Here $x\in \Omega\subset\mathbb{R}^d$, $\Omega$ is a spatial domain with periodic boundary condition or Neumann boundary condition (detailed in later sections), $\rho\colon \Omega \times \mathbb{R}_+\rightarrow \mathbb{R}$ is a scalar non-negative density function satisfying
\begin{equation*}
   \rho(\cdot, t) \in  \mathcal{M}=\Big\{\rho\colon \Omega \rightarrow\mathbb{R}\colon \rho(x,t)\geq 0\Big\},
   \end{equation*}
   for any time $t$, 
 $\mathcal{E}\colon \mathcal{M}\rightarrow\mathbb{R}$ is an energy functional, $V_1, V_2\colon \mathbb{R}\rightarrow\mathbb{R}_+$ are positive mobility functions, $\frac{\delta}{\delta\rho}$ is the  first variation operator in $L^2$ space, and $\rho^0\in\mathcal{M}$ is an initial condition. Equation \eqref{dissipative} forms a class of equations, including Wasserstein gradient flows and the Fisher--KPP equation \cite{Filippo, fisher, kpp}. 
 Detailed examples of $V_1$, $V_2$, and $\mathcal{E}$ are provided in the next section, where we also discuss the extension of \eqref{dissipative} to reaction-diffusion systems.

Equation \eqref{dissipative} is purely dissipative. Denote $\rho(\cdot, t)$ as the solution of the PDE \eqref{dissipative}, then the energy functional $\mathcal{E}$ is a Lyapunov functional. In other words, the first-time derivative of the energy functional $\mathcal{E}$ is nonnegative, satisfying
\begin{equation}
\label{ener-law}
\begin{split}
    &\frac{d}{dt}\mathcal{E}(\rho(\cdot, t))\\
    =&-\int_\Omega\Big[\|\nabla\frac{\delta}{\delta \rho}\mathcal{E}(\rho)(x,t)\|^2V_1(\rho(x,t)) + |\frac{\delta}{\delta \rho}\mathcal{E}(\rho)(x,t)|^2V_2(\rho(x,t))\Big]dx\leq  0 , 
    \end{split}
\end{equation}
where we use the fact that $V_1(\rho)\geq 0$, and $V_2(\rho)\geq 0$ in the above inequality. 
\subsubsection{Metric operators and Distances}
The dissipation of the energy functional also induces a metric function in space $\mathcal{M}$, which further defines distances between two densities $\rho^0, \rho^1\in \mathcal{M}$. This distance designs an implicit time variational problem for computing the gradient flow in metric spaces.
See details among optimal transport type gradient flows, distances, and mean-field control problems in \cite{2005_gradienta, LiLeeOsher22,Mielke11}. 

We directly present generalized optimal transport type distances and the time implicit schemes below for simplicity of discussion. 

\noindent\textbf{Definition}: {\em Distance functional.} Define a distance functional $\mathrm{Dist}_{V_1,V_2}\colon \mathcal{M}\times \mathcal{M}\rightarrow\mathbb{R}_+$ as below. Consider the following optimal control problem: 
\begin{subequations}\label{Dis}    \begin{equation}\label{Dis1}
\begin{split}
    &\mathrm{Dist}_{V_1, V_2}(\rho^0, \rho^1)^2\\
    :=&\inf_{\rho, v_1, v_2}\quad\int_0^1\int_\Omega \Big[\|v_1(x,\tau)\|^2 V_1(\rho(x,\tau))+ |v_2(x,\tau)|^2V_2(\rho(x,\tau))\Big]dxd\tau, 
\end{split}
\end{equation}
where the infimum is taken among $\rho\colon \Omega\times [0, 1]\rightarrow\mathbb{R}_+$,  $v_1, v_2\colon \Omega\times [0,1]\rightarrow\mathbb{R}^d$, such that $\rho$ satisfies a reaction-diffusion type equation with drift vector field $v_1$, drift mobility $V_1$, reaction rate $v_2$, reaction mobility $V_2$, connecting initial and terminal densities $\rho^0$, $\rho^1$: 
\begin{equation}\label{Dis2}
\left\{\begin{aligned}
&\partial_\tau \rho(x,\tau) + \nabla\cdot( V_1(\rho(x,\tau)) v_1(x,\tau))=V_2(\rho(x,\tau))v_2(x,\tau),\quad \tau\in [0,1],\\
&\rho(x,0)=\rho^0(x),\quad \rho(x,1)=\rho^1(x).
\end{aligned}\right.
\end{equation}
\end{subequations}

Variational problem \eqref{Dis} is a generalized Benamou-Brenier formula \cite{benamou2000computational}, where they consider $V_1(\rho)=\rho$, $V_2(\rho)=0$. One common practice is the following change of variable formula, which leads to a linear constraint optimization problem. Denote a moment vector function $m\colon \Omega\times[ 0,1]\rightarrow\mathbb{R}^d$ and a source function $s\colon \Omega\times[0,1]\rightarrow\mathbb{R}$, such that 
\begin{equation*}    m(x,\tau)=V_1(\rho(x,\tau))v_1(x,\tau), \quad s(x,\tau)=V_2(\rho(x,\tau))v_2(x,\tau). 
\end{equation*}
Using variables $m$, $s$, variational problem \eqref{Dis} satisfies
\begin{equation*}
    \mathrm{Dist}_{V_1, V_2}(\rho^0, \rho^1)^2
    :=\inf_{\rho, m, s}\quad\int_0^1\int_\Omega \Big[\frac{\|m(x,\tau)\|^2} {V_1(\rho(x,\tau))}+ \frac{|s(x,\tau)|^2}{V_2(\rho(x,\tau))}\Big]dxd\tau, 
\end{equation*}
such that
\begin{equation*}
\partial_\tau \rho(x,\tau) + \nabla\cdot m(x,\tau) =s(x,\tau),\quad \rho(x,0)=\rho^0(x),\quad \rho(x,1)=\rho^1(x).
\end{equation*}

\subsubsection{Variational time implicit schemes and properties}

We next design a variational implicit-in-time scheme to update gradient flow \eqref{dissipative} iteratively.  

\noindent\textbf{Definition}: {\em Variational time implicit scheme.}
Denote $\Delta t>0$ as a time step size. Consider the scheme below:
\begin{equation}\label{mms}
\begin{split}
\rho^{n}=&\arg\min_{\rho\in\mathcal{M}}\quad\frac{1}{2\Delta t}\mathrm{Dist}_{V_1, V_2}(\rho^{n-1}, \rho)^2+ \mathcal{E}(\rho),
\end{split}
\end{equation}
where  $\mathrm{Dist}_{V_1, V_2}(\rho^{n-1}, \rho)^2$ is the distance functional defined in \eqref{Dis} between current density $\rho$ and previous step density $\rho^{n-1}$. 
After suitable time rescaling, one can show that the minimization scheme \eqref{mms} requires solving the following optimal control problem:
\begin{subequations}\label{GJKO}
\begin{equation}\label{GJKO1}
\begin{split}
\inf_{\rho_{\Delta t}, \rho, m, s}\quad
\underbrace{\frac{1}{2}\int_0^{\Delta t}\Big[\int_\Omega \frac{\|m(x,\tau)\|^2} {V_1(\rho(x,\tau))}+ \frac{|s(x,\tau)|^2}{V_2(\rho(x,\tau))}\Big]dxd\tau}_{=\frac{1}{2\Delta t}\mathrm{Dist}_{V_1, V_2}(\rho^{n-1}, \rho)^2}
+ \mathcal{E}(\rho_{\Delta t}),
\end{split}
\end{equation}
such that
\begin{align}\label{GJKO2}
&\;\partial_\tau \rho(x,\tau) + \nabla\cdot m(x,\tau) =s(x,\tau),\quad \tau\in [0, \Delta t],\\
&\;\rho(x,0)=\rho^{n-1}(x), \quad
\rho(x,\Delta t)=\rho_{\Delta t}(x).
\end{align}
\end{subequations}
The next step solution $\rho^n$ is the density minimizer of \eqref{GJKO}: 
\[\rho^n(x)=\rho_{\Delta t}(x).\]
We demonstrate that the variational scheme \eqref{mms}
is a first-order accurate implicit in time scheme, i.e., 
\begin{equation*}
 \frac{\rho^n-\rho^{n-1}}{\Delta t}=\nabla\cdot(V_1(\rho^n)\frac{\delta}{\delta\rho}\mathcal{E}(\rho^n)) -V_2(\rho^n)\frac{\delta}{\delta\rho}\mathcal{E}(\rho^n)+\mathcal{O}(\Delta t). 
\end{equation*}
\begin{proof}
We write the minimization system of variational problem \eqref{GJKO}. Denote $\Phi(x,\tau)\in\mathbb{R}$, $\tau\in [0, \Delta t]$, as the Lagrange multiplier. The optimal condition of variational problem \eqref{GJKO} satisfies the following saddle point problem:
\begin{equation}
\label{sdsd}
    \inf_{\rho_{\Delta t},\rho, m, s}\sup_{\Phi}~\mathcal{L}( \rho_{\Delta t},\rho, m, s, \Phi),
\end{equation}
where 
\begin{equation*}
\begin{split}
 \mathcal{L}(\rho_{\Delta t},\rho, m, s, \Phi):=& \frac{1}{2}\int_0^{\Delta t}\int_\Omega \Big[\frac{\|m(x,\tau)\|^2} {V_1(\rho(x,\tau))}+ \frac{|s(x,\tau)|^2}{V_2(\rho(x,\tau))}\Big]dxd\tau
+ \mathcal{E}(\rho_{\Delta t}) \\ 
&+\int_0^{\Delta t}\int_\Omega \Phi(x,\tau)\Big(\partial_\tau \rho(x,\tau) + \nabla\cdot m(x,\tau)-s(x,\tau)\Big) dxd\tau
\end{split}
\end{equation*}
We note that from integration by parts, 
\begin{align*}
\int_0^{\Delta t}\int_\Omega \Phi(x,\tau)\partial_\tau \rho(x,\tau) dxd\tau
= &\;
-\int_0^{\Delta t}\int_\Omega \partial_\tau\Phi(x,\tau) \rho(x,\tau) dxd\tau\\
&\;+ \int_\Omega \Phi(x,\Delta t) \rho_{\Delta t}(x) dx
-\int_\Omega \Phi(x,0) \rho^{n-1}(x) dx.
\end{align*}
By computing the saddle point of \eqref{sdsd}, we derive 
\begin{equation*}
\left\{\begin{aligned}
&\frac{\delta}{\delta \rho}\mathcal{L}=0,\quad \textrm{if $\rho>0$,} \\
&\frac{\delta}{\delta m}\mathcal{L}=0,\\
&\frac{\delta}{\delta s}\mathcal{L}=0,\\
&\frac{\delta}{\delta \Phi}\mathcal{L}=0,\\
&\frac{\delta}{\delta \rho_{\Delta t}}\mathcal{L}=0,\\
\end{aligned}\right.\Rightarrow
\left\{\begin{aligned}
&-\frac{\|m\|^2}{2V_1(\rho)^2}V_1'(\rho)-\frac{|s|^2}{2V_2(\rho)^2}V_2'(\rho)-\partial_\tau\Phi=0,\quad\textrm{if $\rho>0$,}\\
&\frac{m}{V_1(\rho)}-\nabla\Phi=0,\\
&\frac{s}{V_2(\rho)}-\Phi=0,\\
&\partial_\tau\rho+\nabla\cdot m=s, \\
&\Phi(x,\Delta t)+\frac{\delta}{\delta\rho_{\Delta t}}\mathcal{E}(\rho_{\Delta t})=0.
\end{aligned}\right.
\end{equation*}
Thus we obtain a minimization system: 
\begin{equation*}
\left\{\begin{aligned}
&\partial_\tau\rho(x,\tau)+\nabla\cdot(V_1(\rho(x,\tau))\nabla \Phi(x,\tau))=V_2(\rho(x,\tau))\Phi(x,\tau),\\
&\rho(0,x)=\rho^{n-1}(x),\quad\Phi(x,\Delta t)=-\frac{\delta}{\delta \rho}\mathcal{E}(\rho)(x),
    \end{aligned}\right.
\end{equation*}
where $\Phi$ satisfies the Hamilton-Jacobi-type equation when $\rho(x,\tau)>0$, such that
\begin{equation*}
\partial_\tau\Phi(x,\tau)+\frac{1}{2}\|\nabla \Phi(x,\tau)\|^2V_1'(\rho(x,\tau))+|\Phi(x,\tau)|^2V_2'(\rho(x,\tau))=0.    
\end{equation*}
We approximate the equation of $\rho(x,\tau)$ at $\tau=\Delta t$: 
\begin{equation*}
\begin{split}
   \rho(x,\Delta t)=&\rho(x,0)-\Delta t\Big[\nabla\cdot(V_1(\rho(x,\tau))\nabla \Phi(x,\tau))-V_2(\rho(x,\tau))\Phi(x,\tau)\Big]|_{\tau=\Delta t}+o(\Delta t)\\
   =&\rho(x,0)+\Delta t\Big[\nabla\cdot(V_1(\rho^n(x))\frac{\delta}{\delta\rho}\mathcal{E}(\rho^n)(x)) -V_2(\rho^n(x))\frac{\delta}{\delta\rho}\mathcal{E}(\rho^n)(x)\Big]+{o}(\Delta t),
   \end{split}
\end{equation*}
where we denote $\rho^n(x)=\rho(x,\Delta t)$.
This finishes the proof. 
\end{proof}
In fact, for first-order implicit time accuracy, one can use the one-step approximated minimization scheme. In other words, we only use a local time approximation of distance functional to compute the implicit time scheme. 

\noindent\textbf{Definition}: {\em One-step relaxation of variational time implicit scheme.} Consider 
\begin{subequations}\label{OJKO}
    \begin{equation}\label{OJKO1}
  \inf_{\rho, m, s}
  \quad\underbrace{\frac{1}{2\Delta t}\int_\Omega\Big[ \frac{\|m(x)\|^2}{V_1(\rho(x))}+\frac{|s(x)|^2}{V_2(\rho(x))} \Big]dx}_{\approx \frac{1}{2\Delta t}\mathrm{Dist}_{V_1, V_2}(\rho, \rho^{n-1})^2}+\mathcal{E}(\rho),   
\end{equation}
where the minimization is over all functions $m\colon \Omega\rightarrow \mathbb{R}^d$, $s\colon \Omega\rightarrow\mathbb{R}$, and $\rho\colon \Omega\rightarrow \mathbb{R}_+$, such that 
\begin{equation}\label{OJKO2}
  \rho(x)-\rho^{n-1}(x)+\nabla\cdot m(x)=s(x).   
\end{equation}
Denote the next step solution $\rho^n$ as  the density minimizer of \eqref{OJKO}.
\end{subequations}

We also demonstrate that the variational scheme \eqref{OJKO} forms a first-order implicit time scheme for the PDE \eqref{dissipative}.
\begin{proof}
The proof is similar to the one in \eqref{GJKO}. 
Denote $\Phi(x)$ as the Lagrange multiplier.
The optimal condition of the variational problem \eqref{OJKO} satisfies the following saddle point problem:
\begin{equation*}
    \inf_{\rho, m, s}\sup_{\Phi}\quad\mathcal{L}(\rho, m, s, \Phi),
\end{equation*}
where 
\begin{equation*}
\begin{split}
 \mathcal{L}(\rho, m, s, \Phi):=& \frac{1}{2}\int_\Omega\Big[ \frac{\|m(x)\|^2} {V_1(\rho(x))}+ \frac{|s(x)|^2}{V_2(\rho(x))}\Big]dx
+\Delta t \mathcal{E}(\rho) \\ 
&+\int_\Omega \Phi(x)\Big(\rho(x)-\rho^{n-1}(x) + \nabla\cdot m(x)-s(x)\Big) dx. 
\end{split}
\end{equation*}
By computing saddle point of the above system, we derive 
\begin{equation*}
\left\{\begin{aligned}
&\frac{\delta}{\delta \rho}\mathcal{L}=0,\\
&\frac{\delta}{\delta m}\mathcal{L}=0,\\
&\frac{\delta}{\delta s}\mathcal{L}=0,\\
&\frac{\delta}{\delta \Phi}\mathcal{L}=0,
\end{aligned}\right.\Rightarrow
\left\{\begin{aligned}
&-\Big[\frac{\|m\|^2}{2V_1(\rho)^2}V_1'(\rho)+\frac{|s|^2}{2V_2(\rho)^2}V_2'(\rho)\Big]+\Delta t\frac{\delta}{\delta\rho}\mathcal{E}(\rho)+\Phi=0,\\
&\frac{m}{V_1(\rho)}-\nabla\Phi=0,\\
&\frac{s}{V_2(\rho)}-\Phi=0,\\
&\rho-\rho^{n-1}+\nabla\cdot m=s.
\end{aligned}\right.
\end{equation*}
One can check that $\Phi=-\Delta t\frac{\delta}{\delta \rho}\mathcal{E}(\rho)+o(\Delta t)$. Thus 
\begin{equation*}
 \frac{\rho^n-\rho^{n-1}}{\Delta t}=\nabla\cdot(V_1(\rho^n)\frac{\delta}{\delta\rho}\mathcal{E}(\rho^n)) -V_2(\rho^n)\frac{\delta}{\delta\rho}\mathcal{E}(\rho^n)+\mathcal{O}(\Delta t).
\end{equation*}
Here $\rho^n=\rho$ is the density minimizer. This finishes the proof. 
\end{proof}

We remark that solving the variational problem \eqref{OJKO} is simpler than optimizing \eqref{GJKO}, since \eqref{OJKO} only involves a local time distance approximation;  see \cite{LiLuWang20, CancesGallouetTodeschi2020_variational}. We also present some properties of the implicit variational scheme \eqref{OJKO}. The algorithm satisfies the entropy dissipation property for any step size $\Delta t\geq 0$.  
\begin{proposition}[Time implicit scheme entropy dissipation]
Denote the solution $\{\rho^n\}_{n\in \mathbb{N}}$ solving the variational implicit scheme \eqref{OJKO}. For any stepsize $\Delta t\geq 0$, we have
\begin{equation*}
   \mathcal{E}(\rho^n)\leq \mathcal{E}(\rho^{n-1}),  \qquad\textrm{for $n\in \mathbb{N}_+$. }
\end{equation*}
\end{proposition}
\begin{proof}
Denote the objective functional \eqref{OJKO1} as 
\begin{equation}
   \mathcal{F}(\rho, m, s)= \frac{1}{2\Delta t}\int_\Omega\Big[ \frac{\|m(x)\|^2}{V_1(\rho(x))}+\frac{|s(x)|^2}{V_2(\rho(x))} \Big]dx+\mathcal{E}(\rho) . 
\end{equation}
Since $(\rho^{n-1}, m=0, s=0)$ is a feasible point satisfying the constraint \eqref{OJKO2}, and $(\rho^n, m^*, s^*)$ is an optimal solution of \eqref{OJKO}, we have 
\begin{equation*}
\mathcal{E}(\rho^n)\leq \mathcal{F}(\rho^n, m^*, s^*)\leq \mathcal{F}(\rho^{n-1}, 0 ,0)=\mathcal{E}(\rho^{n-1}),
\end{equation*}
where we use the fact that
\begin{equation*}
    \mathcal{F}(\rho^n, m^*, s^*)=\mathcal{E}(\rho^n)+\frac{1}{2\Delta t}\int_\Omega\Big[\frac{\|m^*(x)\|^2}{V_1(\rho^n(x))}+\frac{|s^*(x)|^2}{V_2(\rho^n(x))} \Big]dx\geq \mathcal{E}(\rho^n). 
\end{equation*}
We finish the proof. 
\end{proof}

We also remark that there are issues of convexity in computing minimizers of the variational problem \eqref{OJKO}. If $V_1$ and $V_2$ are concave w.r.t. $\rho$, then the minimization problem \eqref{OJKO} is always convex for any positive step size $\Delta t$. In general, this fact may be lost for general mobility functions $V_1$ and $V_2$. In computations, we still apply the first-order optimization algorithm to compute the variational problem \eqref{OJKO}, where we suggest a small stepsize $\Delta t$ in the iterative update. 
\subsection{The abstract ALG2 algorithm}
In this subsection, we formulate saddle point problems  to calculate the variational time implicit schemes \eqref{OJKO};
see also \cite{FortinBook, benamou2000computational}. 

We present the general form of the augmented Lagrangian (ALG2) algorithm \cite{FortinBook} for the following saddle point system:
\begin{align}
\label{saddle-pt}
    \inf_{\bmu} 
    \sup_{\Phi} F(\bmu)-G(\Phi)-(\bmu, \mathcal{D}\Phi)_{\Omega},
\end{align}
where $\mathcal{D}(\Phi)$ is a linear differential operator for $\Phi$, 
and $(\cdot, \cdot)_{\Omega}$ stands for the $L^2$-inner product on the domain $\Omega$. For the problem \eqref{OJKO}, we choose 
\begin{equation*}
  \bmu=(\rho, m, s),
\end{equation*}
with 
\begin{equation*}
    F(\bmu)=\frac{1}{2}\int_\Omega \Big[\frac{\|m\|^2}{V_1(\rho)}+\frac{|s|^2}{V_2(\rho)}\Big]dx+\Delta t\mathcal{E}(\rho),  \quad     G(\Phi)=\int_\Omega \rho^{n-1}\Phi\,dx,
\end{equation*}
and 
\begin{equation*}
\mathcal{D}\Phi=(-\Phi, \nabla\Phi, \Phi). 
\end{equation*}


The  algorithm starts with the dual formulation of the saddle-point problem \eqref{saddle-pt}:
\begin{align}
\label{sd-dual}
    \sup_{\bmu} 
    \inf_{\Phi, \bmu^*}
   F^*(\bmu^*)
+G(\Phi)+(\bmu, \mathcal{D}\Phi-\bmu^*)_{\Omega},
\end{align}
where 
$
   F^*(\bmu^*)
=\sup_{\bmu}(\bmu, \bmu^*)_{\Omega}-F(\bmu)
$
is the Legendre transform.
The saddle point of the above system is equivalent to the saddle point of the following augmented Lagrangian
form:
\begin{align}
\label{sd-aug}
    \sup_{\bmu} 
    \inf_{\Phi, \bmu^*}
L_r(\Phi, \bmu, \bmu^*),
\end{align}
where the augmented Lagrangian
\[
L_r(\Phi, \bmu, \bmu^*):= 
   F^*(\bmu^*)
+G(\Phi)+(\bmu, \mathcal{D}\Phi-\bmu^*)_{\Omega}
+\frac{r}{2}(\mathcal{D}\Phi-\bmu^*,
\mathcal{D}\Phi-\bmu^*)_{\Omega},
\]
in which  $r$ is a positive parameter.

The ALG2 solves the optimization problem \eqref{sd-aug} in a splitting fashion. One iteration  contains the following three steps.
\begin{algorithm}[H]
\caption{One iteration of ALG2 algorithm for variational implicit scheme \eqref{sd-aug}.}
\label{alg:1}
\begin{algorithmic}

\STATE $\bullet$ Step A: update $\Phi$. 
Minimize $L_{r}(\Phi, \bmu, \bmu^*)$ with respect to the first argument by solving the elliptic problem:
Find $\Phi^{\ell}$ such that
it solves 
\begin{align*}
 \inf_{\Phi}L_{r}(\Phi, \bmu^{\ell-1},\bmu^{*,\ell-1}).
\end{align*}

\STATE $\bullet$ Step B: update $\bmu^{*}$.
 Minimize $L_{r}(\Phi, \bmu, \bmu^*)$ with respect to the last argument by solving the nonlinear problem:
Find $\bmu^{*,\ell}$ such that
it solves
\begin{align*}
 \inf_{\bmu^{*}}L_{r}(\Phi^{\ell}, \bmu^{\ell-1}, \bmu^{*}).
\end{align*}

\STATE $\bullet$ Step C: update $\bmu$. This is a simple pointwise update for the Lagrange multiplier $\bmu$ :
\begin{align}
\label{stepC}
\bmu^\ell = \bmu^{\ell-1}+r(\mathcal{D}\Phi^\ell-\bmu^{*,\ell}).
\end{align}
\end{algorithmic}
\end{algorithm}
We note that the key success of the ALG2 algorithm \ref{alg:1} is that
Step A is a simple linear reaction-diffusion equation solve, while the nonlinear Step B can be efficiently solved in a point-wise fashion, provided a good spatial discretization is used for the discretization variables; see Algorithm \ref{alg:2} below. 
We note that for the system case, further splitting in Step A/B for each component calculation will be applied
to further save the computational cost; see Algorithm \ref{alg:3} below.
We will present details of the implementation in Section \ref{sec:alg} where the high-order spatial discretization is introduced.
The error in the Lagrange multipliers in two consecutive iterations $\bmu^{\ell}-\bmu^{\ell-1}$
can be used to monitor the convergence of the ALG2 algorithm.
Typically, a couple of hundred ALG iterations is sufficient for time accuracy. We take 200 ALG iterations in all our numerical results reported in 
Section \ref{sec:num}.


\section{Examples: Wasserstein gradient flow, reaction-diffusion equations, and reversible reaction-diffusion systems}\label{sec:model}

This section presents examples of dissipative dynamic systems that fit in the framework of the previous section: Wasserstein gradient flows, scalar reaction-diffusion equations, 
and reversible reaction-diffusion systems. 
\subsection{Wasserestein gradient flow}
We consider the following $L^2$-Wasserstein gradient flow
for a time-dependent probability density $\rho:\Omega\times \mathbb
R_+\rightarrow \mathbb
R_+$ on a domain $\Omega\subset \mathbb{R}^d$,
\begin{align}
\label{wg}
\partial_t \rho = \nabla\cdot\left(\rho\nabla \frac{\delta}{\delta \rho}\mathcal{E}(\rho)\right),
\end{align}
subject to Neumann boundary conditions.
Typically, the energy functional $\mathcal{E}(\rho)$ takes the following form
\begin{equation}
\label{ener}
\mathcal{E}(\rho):=\int_{\Omega}\left[\alpha U_m(\rho(x))+\rho(x) V(x) + \frac12(W*\rho)(x)\rho(x)\right]  dx,
\end{equation}
where $\alpha\ge0$ is the diffusion coefficient, $U_m(\rho)$ is the diffusion term with 
\[
U_m(\rho)=
\begin{cases}
\rho\log(\rho) &\text{ if }m=1,\\[.2ex]
\frac{\rho^m}{m-1} &\text{ if }m>1,
\end{cases}
\]
$\rho V$ is the drift term with drift potential $V$, and $\frac12(W *\rho)\rho$ is the aggregation term with 
the convolution 
\[
(W*\rho)(x):=\int_{\Omega}W(x-y)\rho(y)\, dy,
\]
in which $W(\cdot)$ is the symmetric interaction kernel.
Its variational derivative is 
\begin{align}
\label{var-d}
\frac{\delta}{\delta \rho} \mathcal{E}=\alpha U_m'(\rho)+V+W*\rho.
\end{align}
The equation \eqref{wg} is mass conserving, positivity preserving, and  satisfies the energy dissipation law
\eqref{ener-law} with 
$V_1(\rho)=\rho$ and $V_2(\rho)=0$.

This model is a special case of \eqref{dissipative}  with $V_1(\rho)=\rho$, $V_2(\rho)=0$, and 
energy functional $\mathcal{E}$ in \eqref{ener}. 
The corresponding one-step variational time implicit scheme \eqref{OJKO} is 
\begin{subequations}\label{OJKO-A}
    \begin{equation}\label{OJKO-A1}
\inf_{\rho, m} \quad
 \frac{1}{2\Delta t} \int_\Omega \frac{\|m(x)\|^2}{\rho(x)} dx
  +\mathcal{E}(\rho),   
\end{equation}
where the minimization is over all functions $m\colon \Omega\rightarrow \mathbb{R}^d$, and $\rho\colon \Omega\rightarrow \mathbb{R}_+$, such that 
\begin{equation}\label{OJKO-A2}
  \rho(x)-\rho^{n-1}(x)+\nabla\cdot m(x)=0.   
\end{equation}
\end{subequations}
The next step solution $\rho^n$ is the density minimizer of \eqref{OJKO}, i.e., $\rho^n(x)=\rho(x)$.
Here the first term in \eqref{OJKO-A1}
is the one-step relaxation approximation of the classical Wasserstein distance in Benamou-Brenier's dynamic formulation \cite{benamou2000computational}, i.e., the distance in 
\eqref{Dis} with $V_1(\rho)=\rho$ and $V_2(\rho)=0$.
We note that such approximation was originally used in 
\cite{LiLuWang20, CancesGallouetTodeschi2020_variational}.

This problem is equivalent to finding the saddle point of \eqref{saddle-pt} in which $\bmu = (\rho, m)$, 
\begin{equation*}
    F(\bmu)=\int_\Omega \frac{\|m\|^2}{2\rho}dx+\Delta t\mathcal{E}(\rho),  \quad     G(\Phi)=\int_\Omega \rho^{n-1}\Phi\,dx,
\end{equation*}
and $\mathcal{D}\Phi= (-\Phi, \nabla\Phi)$, which can be solved using ALG2 Algorithm \ref{alg:1}
after a spatial discretization is used; see Section \ref{sec:alg}.

\subsection{Dissipative reaction-diffusion equation}
Adding a reaction term of  form $-V_2(\rho)\frac{\delta}{\delta \rho}\mathcal{E}$
with a non-negative mobility function $V_2(\rho)\ge0$ to the PDE \eqref{wg}, we get the following reaction-diffusion equation:
\begin{align}
\label{rd}
\partial_t \rho = \nabla\cdot\left(\rho\nabla \frac{\delta}{\delta \rho}\mathcal{E}\right)
-V_2(\rho)\frac{\delta \mathcal{E}}{\delta \rho},
\end{align}
which is again a special case of \eqref{dissipative}, with $V_1(\rho)=\rho$, and a general non-negative function $V_2(\rho)$.
Hence, the corresponding one-step variational time implicit scheme \eqref{OJKO} is 
\begin{subequations}\label{OJKO-B}
    \begin{equation}\label{OJKO-B1}
\inf_{\rho, m, s} \quad
 \frac{1}{2\Delta t} \int_\Omega\Big[ \frac{\|m(x)\|^2}{\rho(x)} 
 +\frac{|s(x)|^2}{V_2(\rho(x))} 
 \Big]dx
  +\mathcal{E}(\rho),   
\end{equation}
where the minimization is over all functions $m\colon \Omega\rightarrow \mathbb{R}^d$,
$s\colon \Omega\rightarrow \mathbb{R}$,
and $\rho\colon \Omega\rightarrow \mathbb{R}_+$, such that 
\begin{equation}\label{OJKO-B2}
  \rho(x)-\rho^{n-1}(x)+\nabla\cdot m(x)=s(x).   
\end{equation}
\end{subequations}
This is the saddle point of \eqref{saddle-pt} in which $\bmu = (\rho, m, s)$, 
\begin{equation*}
    F(\bmu)=\int_\Omega\Big[\frac{\|m\|^2}{2\rho}+\frac{|s|^2}{2V_2(\rho)}\Big]dx+\Delta t\mathcal{E}(\rho),  \quad     G(\Phi)=\int_\Omega \rho^{n-1}\Phi\,dx,
\end{equation*}
and $\mathcal{D}\Phi= (-\Phi, \nabla\Phi, \Phi)$.

We will postpone the introduction of a model with 
a more general $V_1(\rho)\not=\rho$ to Section \ref{sec:ts}, where a two-component reversible reaction-diffusion system with detailed balance is discussed.

Below we list three choices of 
$V_2(\rho)$ along with their corresponding energies that will be used in our numerical experiments:
\begin{itemize}
\item [(i)] 
$V_2(\rho) = c\,\rho^\gamma$ where $c\ge 0$ and $\gamma\in \mathbb{R}$, 
with a general $\mathcal{E}(\rho)$ given in \eqref{ener}.
Here $\gamma=1$ corresponds to the Wasserstein-Fisher-Rao metrics used in \cite{Chizat18,LieroMielkeSavare2016_optimal}, and $\gamma=0$ is related to unnormalized optimal transport \cite{LeeLaiLiOsher21}.
Both cases lead to a convex optimization problem 
\eqref{OJKO-B} when the energy is convex; see Remark \ref{rk:convex} below.
\item [(ii)] 
$V_2(\rho) = c\,\frac{\rho-1}{\log(\rho)}$ where $c\ge 0$ with a general $\mathcal{E}(\rho)$ given in \eqref{ener}.
This choice also leads to a convex optimization problem
for a convex energy.
\item [(iii)]
$V_2(\rho) =\frac{\rho(\rho-1)}{\alpha\log(\rho)}$, with energy 
$\mathcal{E}(\rho):=\int_{\Omega}\alpha\rho(x)(\log(\rho)-1)  dx,
$
where $\alpha>0$.
This model is the following Fisher--KPP equation; see \cite[Example 7]{LiLeeOsher22}:
\begin{align}
\label{kpp}
\frac{\partial \rho}{\partial t}-\nabla\cdot(\alpha\nabla\rho)=\rho(1-\rho).
\end{align}
It, however, does not lead to a convex optimization problem.
\end{itemize}

\subsection{Strongly reversible reaction-diffusion systems}
Our next model deals with the system of strongly reversible reaction-diffusion equations \cite{Mielke11}.
We consider $M$ different chemical species 
$X_1,\dots, X_M$ reacting according to $R$ mass-action laws:
\begin{align}
\label{mal}
\alpha_1^pX_1+\cdots+\alpha_M^pX_M\xrightleftharpoons[k_-^p]{k_+^p}
\beta_1^pX_1+\cdots+\beta_M^pX_M,
\end{align}
where $p=1, \cdots, R$ is the number of possible reactions, 
$\bma^p=(\alpha_1^p,\cdots,\alpha_M^p), \bmb^p=(\beta_1^p,\cdots,\beta_M^p)\in\mathbb{N}_0^M$ are the vectors of the stoichiometric coefficients, and 
$k_+^p, k_-^p$ are the positive forward and backward reaction rates.
For simplicity, we restrict ourselves to the {\it strongly reversible} case where $k_+^p=k_-^p=k^p>0$
in this subsection.
The more general case of reversible reaction-diffusion systems with detailed balance that allows $k_+^p\not=k_-^p>0$ will be discussed in the next subsection. 

Combining the mass-action laws \eqref{mal} with (independent) isotropic linear diffusion 
with energy $\mathcal{E}_i(\rho_i)=\int_{\Omega}\rho_i(\log(\rho_i)-1)\, dx$
for each density $\rho_i$ of species $X_i$, we get the following reaction-diffusion system:
\begin{align}
\label{rds}
\partial_t\rho_i - \nabla\cdot\left(
\gamma_i\rho_i
\nabla\frac{\delta}{\delta \rho}\mathcal{E}_i(\rho_i)\right)
=&\; -\sum_{p=1}^Rk^p(\alpha_i^p-\beta^p_i)(\bmr^{\bma^p}-\bmr^{\bmb^p}),
\end{align}
for $1\le i\le M$, where 
$\bmr=(\rho_1, \cdots, \rho_M)$ and 
the multi-index notation 
$\bmr^{\bma^p}:=\prod_{i=1}^M\rho_i^{\alpha_i^p}$ is used.
Here the potential $\frac{\delta}{\delta \rho}\mathcal{E}_i(\rho_i)=\log(\rho_i)$ is simply the logarithm.

Next, we recast the above system \eqref{rds} back to a system version of the general dissipative form \eqref{dissipative}
using appropriate mobility functions.
We introduce the following function; see \cite{Mielke11}:
\begin{align}
\label{mob}
\ell(x,y)=\begin{cases}
\frac{x-y}{\log(x)-\log(y)}&\text{ for } x\not=y,\\[2ex]
y&\text{ for } x=y,
\end{cases}
\end{align}
and denote the following mobility functions:
\begin{subequations}
\label{mobs}
\begin{align}
\label{mobs-V1}
V_{1, i}(\rho_i)= &\; \gamma_i\rho_i, \quad\forall 1\le i\le M, \\
\label{mobs-V2}
V_{2,p}(\bmr) = &\;k^p\,\ell\left(\bmr^{\bma^p}, \bmr^{\bmb^p}\right),\quad
\forall 1\le p\le R.
\end{align}
\end{subequations}
Using these notations, it can be shown that \eqref{rds} is equivalent to 
\begin{align}
\label{rds-V1V2}
\partial_t\rho_i =&\;\nabla\cdot\left(
V_{1,i}(\rho_i)\nabla\frac{\delta}{\delta \rho}\mathcal{E}_i(\rho_i)\right) 
\nonumber \\
&\; -\sum_{p=1}^RV_{2,p}(\bmr)(\alpha_i^p-\beta^p_i)\sum_{j=1}^M(\alpha_j^p-\beta_j^p)
\frac{\delta}{\delta \rho}\mathcal{E}_i(\rho_i).
\end{align}
It is now clear that the above system is purely dissipative as for the scalar case \eqref{dissipative}. That is, 
the first-time derivative of the total energy functional  is nonnegative and satisfies
\begin{equation}
\label{ener-law2}
\begin{split}
    \frac{d}{dt}\sum_{i=1}^M\mathcal{E}_i(\rho_i(\cdot, t))
    =&\;-\sum_{i=1}^M\int_\Omega\|\nabla\frac{\delta}{\delta \rho}\mathcal{E}_i(\rho)_i(x,t)\|^2V_{1,i}(\rho_i)\,dx \\
    &\;-\sum_{p=1}^R\int_\Omega\left|\sum_{j=1}^M(\alpha_j^p-\beta_j^p)
\frac{\delta}{\delta \rho}\mathcal{E}_i(\rho_i)\right|^2V_{2,p}(\bmr)\,dx.
    \end{split}
\end{equation}

As in the scalar case in Definition \ref{Dis}, we consider an optimal transport type distance: 
\begin{equation*}
\begin{aligned}
\mathrm{\mathbf{Dist}}_{V_1, V_2}(\bmr^0, \bmr^1)^2=& \inf_{\bmr, \bmm, \bms} \Big\{\int_0^1\int_\Omega 
\left(\sum_{i=1}^M\frac{|m_i|^2}{V_{1,i}(\rho_i)}+
\sum_{p=1}^R\frac{|s_p|^2}{V_{2,p}(\bmr)}\right)
 dx\mathrm{d\tau}:\;\;\\
&\hspace{1cm}\begin{tabular}{l}
$\partial_\tau\rho_i+\nabla\cdot m_i=
\sum_{p=1}^R(\alpha_i^p-\beta_i^p)s_p,\forall 1\le i\le M$\\
$\bmr(\cdot, 0)=\bmr^0, \;\;
\bmr(\cdot, 1)=\bmr^1.$
\end{tabular}\Big\},
\end{aligned}
\end{equation*}
where $\bmm=(m_1,\cdots, m_M)$ is the collection of fluxes, and 
$\bms=(s_1,\cdots, s_R)$ is the collection of sources.
Using this distance, the variational time implicit scheme is defined as follows (compare Definition \eqref{mms} for the scalar case).

\noindent\textbf{Definition}: {\em Variational time implicit scheme for system \eqref{rds-V1V2}.}
Denote $\Delta t>0$ as a time step size. Consider the scheme below:
\begin{equation}\label{mms-sys}
\begin{split}
\bmr^{n}=&\arg\min_{\bmr\in[\mathcal{M}]^M}\quad\frac{1}{2\Delta t}\mathrm{{\mathbf{Dist}}}_{V_1, V_2}(\bmr^{n-1}, \bmr)^2+ 
\sum_{i=1}^M\mathcal{E}_i(\rho_i). 
\end{split}
\end{equation}

Its one-step relaxation is 
given as follows, which is the starting point of our spatial discretization to be discussed in the next section.

\noindent\textbf{Definition}: {\em One-step relaxation of variational time implicit schemes for system \eqref{rds-V1V2}.} Consider 
\begin{subequations}\label{OJKOs}
    \begin{equation}\label{OJKOs1}
 \inf_{\bmr, \bmm, \bms}
  \quad
  \frac{1}{2\Delta t}
  \left(\sum_{i=1}^M\int_\Omega\frac{
  \|m_i\|^2}{V_{1,i}(\rho_i)}dx
+  \sum_{p=1}^R\int_\Omega\frac{
  \|s_p\|^2}{V_{2,p}(\bmr)}dx\right)
    +\sum_{i=1}^M\mathcal{E}_i(\rho_i),   
\end{equation}
where the minimization is over all functions $\bmm\colon \Omega\rightarrow [\mathbb{R}^d]^M$, $\bms\colon \Omega\rightarrow[\mathbb{R}]^R$, and $\bmr\colon \Omega\rightarrow [\mathbb{R}_+]^M$, such that 
\begin{equation}\label{OJKOs2}
  \rho_i(x)-\rho_i^{n-1}(x)+\nabla\cdot m_i(x)=
  \sum_{p=1}^R(\alpha_i^p-\beta_i^p)s_p(x),\quad  \forall 1\le i\le M.
\end{equation}
\end{subequations}
The next step solution $\bmr^n$ is the density minimizer of \eqref{OJKOs}.
It is the saddle point of \eqref{saddle-pt} in which 
\[
\bmu = (\rho_1,\cdots, \rho_M,  m_1,\cdots, m_M, s_1, \cdots, s_R), 
\Phi = (\Phi_1, \cdots, \Phi_M),
\]
\begin{equation*}
    F(\bmu)=
     \frac{1}{2\Delta t}
\left(\sum_{i=1}^M\int_\Omega\frac{
  \|m_i\|^2}{V_{1,i}(\rho_i)}dx
+  \sum_{p=1}^R\int_\Omega\frac{
  \|s_p\|^2}{V_{2,p}(\bmr)}dx\right),  
    \quad     G(\Phi)=\sum_{i=1}^M\int_\Omega \rho_i^{n-1}\Phi_i\,dx,
\end{equation*}
and 
\begin{align*}
\mathcal{D}\Phi= &\;(-\Phi_1,\cdots,-\Phi_M, \nabla\Phi_1, \cdots, \nabla\Phi_M, \\
&\;\quad\quad\sum_{i=1}^M(\alpha_i^1-\beta_i^1)\Phi_i, 
,\cdots,\sum_{i=1}^M(\alpha_i^R-\beta_i^R)\Phi_i).
\end{align*}

\subsection{Reversible reaction-diffusion system with detailed balance}
\label{sec:25}
Note that the strongly reversible reaction-diffusion system 
\eqref{rds-V1V2} uses
the same energy
$\mathcal{E}_i(\rho_i)=\int_{\Omega}\rho_i(\log(\rho_i)-1)\, dx$
for all species. 
By simply relaxing this requirement and rescale the energy as 
\begin{align}
\label{db-ener}
\mathcal{E}_i(\rho_i)=\int_{\Omega}\rho_i(\log(\kappa_i\rho_i)-1)\, dx,
\end{align}
with $\kappa_i>0$ being a positive constant to be determined by the reaction rates $k_{\pm}^p$, we 
will recover reversible reaction-diffusion systems with 
detailed balance; see \cite{Glitzky13, LiuWangWang21a, LiuWangWang22, LiuWangWang22x}. 
For the above choice of energy, there holds 
\[
\frac{\delta}{\delta \rho}\mathcal{E}_i(\rho_i)
=\log(\kappa_i\rho_i).
\]
Below we give two specific examples that will be used in the numerical results section. 

\subsubsection{A two species model}
\label{sec:ts}
We consider two species $X_1, X_2$ with a single reversible 
reaction
\[
X_1+2X_2
\xrightleftharpoons[k_-]{k_+}
3X_2,
\]
with $k_-, k_+>0$.
Denoting the following coefficients and mobility functions,
\begin{subequations}
\label{mobX}
    \begin{align}
&\kappa_1 = k_+, \quad\kappa_2 = k_-\\
&V_{1,1}(\rho_1) = \gamma_1(\rho_1)^m,\quad
V_{1,2}(\rho_2) = \gamma_2\rho_2,\\
&V_{2}(\rho_1, \rho_2) = 
\ell(\kappa_1\rho_1\rho_2^2, \kappa_2\rho_2^3),
\end{align}
\end{subequations}
with $\ell(\cdot, \cdot)$ given in \eqref{mob},
$\gamma_1,\gamma_2>0$, $m\ge 1$, 
and using the energy \eqref{db-ener}, the system \eqref{rds-V1V2}
written in component-wise notation is given as follows:
\begin{subequations}
\label{2comp}
\begin{align}
\label{2comp1}
\partial_t\rho_1 =&\; \nabla\cdot\left(V_{1,1}(\rho_1)\nabla\frac{\delta\mathcal{E}_1}{\delta \rho}(\rho_1)\right)
 -V_2(\rho_1, \rho_2)
  \left(
 \frac{\delta\mathcal{E}_1}{\delta \rho}(\rho_1)-\frac{\delta\mathcal{E}_2}{\delta \rho}(\rho_2)\right),\\
 \partial_t\rho_2 =&\; \nabla\cdot\left(V_{1,2}(\rho_2)\nabla\frac{\delta\mathcal{E}_2}{\delta \rho}(\rho_2)\right)
 +V_2(\rho_1, \rho_2)
  \left(
 \frac{\delta\mathcal{E}_1}{\delta \rho}(\rho_1)-\frac{\delta\mathcal{E}_2}{\delta \rho}(\rho_2)\right).
\end{align}
\end{subequations}
This is the following two-component reversible
reaction-diffusion system studied in \cite{LiuWangWang21a,LiuWangWang22}, which has  potential applications in modeling tumor growth (see \cite{Liu18x,Perthame14tumor}):
\begin{align*}
\partial_t\rho_1 - \frac{\gamma_1}{m}
\Delta \rho_1^m
=&\; -(k_+\rho_1\rho_2^2-k_-\rho_2^3),\\
\partial_t\rho_2 - \gamma_2\Delta\rho_2
=&\; (k_+\rho_1\rho_2^2-k_-\rho_2^3).
\end{align*}

\subsubsection{A reversible four-component Gray-Scott model}
\label{rgs-model}
Our final example is the reversible four-component
Gray-Scott model originally proposed in \cite{Liang22GS} and numerically studied in \cite{LiuWangWang22}.
We consider four species $X_1, X_2, X_3, X_4$ with three reversible reactions
\[
X_1+2X_2
\xrightleftharpoons[k_-^1]{k_+^1}
3X_2,
\;\;
X_2
\xrightleftharpoons[k_-^2]{k_+^2}
X_3,
\;\;
X_1
\xrightleftharpoons[k_-^3]{k_+^3}
X_4.
\]
The reaction-diffusion system that combines these reactions with linear diffusion (with $M=4, R=3$)  can be written into the form \eqref{rds-V1V2} by the following specific choices of $\kappa$-values, and mobility functions $V_{1,i}$ and $V_{2,p}$:
\begin{subequations}
\label{rGS-coo}
\begin{align}
&\kappa_1 = 1, \quad\kappa_2 = \frac{k_-^1}{k_+^1}, 
\quad\kappa_3 =\; \frac{k_-^1}{k_+^1}\frac{k_-^2}{k_+^2}, 
\quad\kappa_4 = \frac{k_-^3}{k_+^3},\\
&V_{1,1}(\rho_1) = \gamma_1\rho_1,\quad
V_{1,2}(\rho_2) = \gamma_2\rho_2,
\quad V_{1,3}(\rho_3) =V_{1,4}(\rho_4)=  0,\\
&V_{2,1}(\bmr) = 
\ell(k_+^1\bmr^{\alpha^1}, k_-^1\bmr^{\beta^1})
=\;\frac{k_+^1\rho_1\rho_2^2-k_-^1\rho_2^3}{\log(\kappa_1 \rho_1)-\log(\kappa_2 \rho_2)},\\
&V_{2,2}(\bmr) = 
\ell(k_+^2\bmr^{\alpha^2}, k_-^2\bmr^{\beta^2})
=\;\frac{k_+^2\rho_2-k_-^2\rho_3}{\log(\kappa_2 \rho_2)-\log(\kappa_3 \rho_3)},\\
&V_{2,3}(\bmr) = 
\ell(k_+^2\bmr^{\alpha^3}, k_-^3\bmr^{\beta^3})
=\;\frac{k_+^3\rho_1-k_-^3\rho_4}{\log(\kappa_1 \rho_1)-\log(\kappa_4 \rho_4)}.
\end{align}
\end{subequations}
For completeness, we write down the PDE system \eqref{rds-V1V2}
with the above choice of parameters using a standard component-wise notation in the following:
\begin{subequations}
\label{rGS}
\begin{align}
\partial_t\rho_1 =&\; \gamma_1\Delta \rho_1
 -(k_+^1\rho_1\rho_2^2-k_-^1\rho_2^3)
 -(k_+^3\rho_1-k_-^3\rho_4),\\
 \partial_t\rho_2 =&\; 
 \gamma_2\Delta \rho_2
 +(k_+^1\rho_1\rho_2^2-k_-^1\rho_2^3)
 -(k_+^2\rho_2-k_-^2\rho_3),\\
 \partial_t\rho_3 =&\; 
(k_+^2\rho_2-k_-^2\rho_3),\\
 \partial_t\rho_4 =&\; 
(k_+^3\rho_1-k_-^3\rho_4).
\end{align}
\end{subequations}
This is the reversible Gray-Scott model proposed in \cite{Liang22GS}
to approximate the following two-component irreversible Gray-Scott model \cite{Gray85}:
\begin{subequations}
\label{irrGS}
\begin{align}
\partial_t\rho_1 =&\; \gamma_1\Delta \rho_1
 -k_+^1\rho_1\rho_2^2
 -k_+^3(\rho_1-1),\\
 \partial_t\rho_2 =&\; 
 \gamma_2\Delta \rho_2
 +k_+^1\rho_1\rho_2^2
 -k_+^2\rho_2,
\end{align}
\end{subequations}
which can form spatially complex patterns \cite{Pearson93}, and is widely used to study pattern formations.
We comment that by requiring
\begin{align}
\label{rrr}
\kappa_-^1\rho_2^3\approx 0,\quad 
k_-^3\rho_4\approx k_+^3, \quad \text{ and}\quad \kappa_-^2\rho_3\approx 0,
\end{align}
the reversible Gray-Scott model \eqref{rGS} formally converges to the irreversible Gray-Scott model \eqref{irrGS}.
We refer interested readers to \cite{Liang22GS} for a  theoretical study. Formally, the conditions \eqref{rrr} can be 
achieved by taking very small backward reaction rates $\kappa_-^1, \kappa_-^2, \kappa_-^3\ll1$, and using initial 
value for $\rho_4$ such that  $\rho_4=\frac{\kappa_+^3}{\kappa_-^3}\gg 1$. 
As a side note, we mention that spatially complex patterns were not observed in the numerical results 
\cite[Example 4.3]{LiuWangWang22}, which uses a second-order  operator splitting scheme via an energetic variational formulation.
We found that the reason for no pattern formation in the test case in \cite{LiuWangWang22} was due to inappropriate choices of a too large backward reaction rate $k_-^3$ and the initial condition.
With a more careful choice of diffusion coefficients, reaction rates, and initial conditions, we numerically observe complex pattern formations in both 1D and 2D reversible Gray-Scott models;
see our simulation results in Section \ref{ex6}.

\section{High-order spatial discretization for generalized time implicit schemes}
\label{sec:alg}
In this section, we first apply high-order spatial discretization to the time implicit schemes \eqref{OJKO-A}, \eqref{OJKO-B} and
their system version \eqref{OJKOs}, and then discuss the practical implementation of each step of the ALG2 Algorithm \ref{alg:1}.
We restrict ourselves to the two-dimensional setting with a rectangular domain $\Omega$, 
which is triangulated using a uniform rectangular mesh $\Th=\{T\}$.
While our method can work on general unstructured triangular meshes, see \cite{FuLiu23}, the restriction to uniform rectangular meshes has a huge advantage in computing the convolution term in the  energy \eqref{ener}, where the Fast Fourier transform can be applied.

\subsection{The finite element spaces and notation}
The spatial discretization is adopted from our previous work on high-order schemes for optimal transport and mean field games \cite{FuLiu23}.
Specifically, the high-order $H^1$-conforming finite element space 
\begin{align}
\label{fes-V}
V_h^k:=&\;\{v\in H^1(\Omega): \;\; v|_{T}
\in\mathcal{Q}^k(T)\;\;
\forall T\in\Th\},
\end{align}
is used to approximate the $\Phi$ variable, and  the high-order $L^2$-conforming discontinuous polynomial space,
\begin{align}
\label{fes-W}
W_h^k:=&\;\{w\in L^2(\Omega): \;\; w|_{T}
\in \mathcal{Q}^k(T)\;\;
\forall T\in\Th\},
\end{align}
is used to approximate the other variables where derivative information is not needed.
Here $\mathcal{Q}^k(T)$ is the space of tensor-product polynomial spaces of degree no greater than $k\ge 1$ in each direction.
We equip the space $W_h^k$ with a set of {\it nodal basis}
$\{\varphi_i\}_{i=1}^{N_W}\subset W_h^k$ that satisfies
\begin{align}
\label{W-basis}
\varphi_i(\bmxi_j)=\delta_{ij},\quad \forall 1\le j\le N_W,
\end{align}
where $N_W$ is the dimension of the space $W_h^k$, $\delta_{ij}$
is the Kronecker delta function, and 
$\{\bmxi_i\}_{i=i}^{N_W}$ is the collection of $N_{W}$ Gauss-Legendre integration points 
with corresponding weights $\{\omega_i\}_{i=1}^{N_W}$ 
on the mesh $\Th$.
For the current work, only evaluation on quadrature points for functions in $W_h^k$ is needed in the algorithm, not their derivatives. 
Hence, given a function $u_h\in W_h^k$ expressed as 
$
u_h=\sum_{i=1}^{N_W}{\sf u}_i\varphi_i(x),
$
we simply need to store and update its coefficient vector $[{\sf u}_i,\cdots, {\sf u}_{N_W}]^T$, which makes its practical implementation extremely simple. 
Moreover, we denote the discrete $L^2(\Omega)$-inner product  $(\cdot, \cdot)_h$
as 
\begin{align}
\label{num-int}
(u, v)_h := \sum_{i=1}^{N_W} u(\bmxi_i)v(\bmxi_i)\omega_i,
\end{align}
we have 
$
(u_h, v_h)_h = \sum_{i=1}^{N_W} {\sf u_i}{\sf v_i}\omega_i.
$
for any function $u_h = \sum_{i=1}^{N_W}{\sf u}_i\varphi_i(x)\in W_h^k$ and $v_h = \sum_{i=1}^{N_W}{\sf v}_i\varphi_i(x)\in W_h^k$.


\subsection{High-order FEM for the reaction diffusion equation}
Since the variation time implicit scheme 
for the Wasserstein gradient flow problem 
\eqref{OJKO-A}
is a special case for the reaction-diffusion problem \eqref{OJKO-B} with no reaction $V_2(\rho)=0$. 
We only present the high-order spatial discretization for \eqref{OJKO-B}.
We first write the discrete saddle point problem
in its augmented Lagrangian form \eqref{sd-aug}:
given mesh $\Th$, polynomial degree $k\ge 1$, 
time step size $\Delta t>0$ and 
density approximation $\rho_h^{\text{old}}$ at the previous time step, 
find $\bmu_h, 
\bmu_h^*\in [W_h^k]^{4}$, and $\Phi_h\in V_h^{k}$, such that
\begin{align}
\label{aug-mfp-h}
\inf_{\bmu\in [W_h^k]^{4}} \;\;
    \sup_{\Phi_h\in V_h^{k}, \bmu_h^*\in [W_h^k]^{4}}L_{r,h}(\Phi_h, \bmu_h, \bmu_h^*),
\end{align}
where  
$\bmu_h=(\rho_{h}, m^0_{h}, m^1_{h}, s_{h})$ is the collection of density $\rho_h$, (two-dimensional) flux
$m_h=(m_{h}^0, m_{h}^1)$, and source term $s_h$, 
$\bmu_h^*=(\rho^*_{h}, m^{0,*}_{h},m^{1,*}_{h}, s^*_{h})$
is its dual, and the discrete augmented Lagrangian is 
\begin{align}
\label{Lrh}
L_{r,h}(\Phi_h, \bmu_h, \bmu_h^*):= &\;
   F_h^*(\bmu_h^*)
+G_h(\Phi_h)+(\bmu_h, \mathcal{D}\Phi_h-\bmu_h^*)_h\nonumber\\
&\;+\frac{r}{2}(\mathcal{D}\Phi_h-\bmu_h^*,
\mathcal{D}\Phi_h-\bmu_h^*)_h.
\end{align}
Here 
$(\cdot, \cdot)_h$ is the volume integration rule given in \eqref{num-int},
the operators
\begin{align}
\label{op}
\mathcal{D}\Phi_h:=&\;(-\Phi_{h}, 
\partial_{x_0}\Phi_{h},
\partial_{x_1}\Phi_{h},
\Phi_{h}),
\\
\label{Gh}
   G_h(\Phi_h)
:=&\;(\rho_h^{\text{old}}, \Phi_h)_h,\\
\label{Fhs}
   F_h^*(\bmu_h^*)
:=&\;
\sup_{
\bmu_h\in [W_h^k]^4
}(\bmu_h^*, \bmu_h)_h- 
   F_h(\bmu_h),
\end{align}
where $(\partial_{x_0}, \partial_{x_1})=\nabla$ is the gradient, 
and $F_h$ is given as 
\begin{align}
\label{Fh}
F_h(\bmu_h):=
\left(\frac{|m_h^0|^2+|m_h^1|^2}{2\rho_{h}}
+\frac{s_{h}^2}{2V_2(\rho_{h})}, 1\right)_h
+\Delta t \,\mathcal{E}_h(\rho_{h}),
\end{align}
in which the discrete total energy 
\begin{align}
\label{Eh}
\mathcal{E}_h(\rho_{h}):=
\left(\alpha U_m(\rho_{h})+\rho_{h} V(x), 1\right)_h 
+ \frac12\left(W*\rho_{h}, \rho_{h}\right)_h
\end{align}
for energy of the form \eqref{ener}.
We note that when the interaction kernel $W(x)$ is smooth, the convolution term $W*\rho_{h}$ in the above expression can be simply evaluated using the same integration rule \eqref{num-int}.
On the other hand, for singular kernels with $W(0)=\pm\infty$, we shall use alternative integration rules to avoid the evaluation of $W(0)$ when evaluating this convolution term. 

Note that a similar formulation can be used for the more general case \eqref{OJKO} for the equation \eqref{dissipative} where the denominator in the first term in \eqref{Fh} is replaced by a general mobility function $V_1(\rho_h)$.


\begin{remark}[On polynomial degree for $\Phi_h$ and $\bmu_h$]
We note that  in our previous work \cite{FuLiu23}, the polynomial degree for 
the discontinuous functions $\bmu_h$
associated with the integration rule space
$W_h^k$
is taken to be one order lower than that for the 
the continuous function $\Phi_h$.
Here we find that increasing the integration rule space order to be the same as the continuous space $V_h^k$
leads to a more accurate result. Hence we use equal order approximations for all our numerical results.
\end{remark}

We next provide a practical implementation of  each step of the ALG2 Algorithm \ref{alg:1} for solving the saddle point problem.
\subsubsection{Step A: scalar case}
\label{stepA1}
Taking infinium of $L_{r,h}$ with respect to $\Phi_h$, we arrive at 
a constant coefficient reaction-diffusion equation:
find $\Phi_h^\ell\in V_h^k$ such that
\begin{align}
\label{scalar-rdh}
(\mathcal{D}\Phi_h^{\ell}, \mathcal{D}\Psi_h)_h=
(\bmu_h^{*,\ell-1}-\frac{1}{r}\bmu_h^{\ell-1}, \mathcal{D}\Psi_h)_h
- \frac{1}{r}(\rho_h^{\text{old}}, \Psi_h)_h,\quad \forall \Psi_h\in V_h^k.
\end{align}
Using the definition in \eqref{op}, we 
write the above equation using physical variables: 
\begin{align*}
2(\Phi_h^{\ell}, \Psi_h)_h+
(\nabla\Phi_h^{\ell}, \nabla\Psi_h)_h
=&\;
(s_h^{*,\ell-1}-\rho_h^{*,\ell-1}+\frac{\rho_h^{\ell-1}-s_h^{\ell-1}-\rho_h^{\text{old}}}{r},
\Psi_h)_h\\
&\;+(\bmm_h^{*,\ell-1}-\frac{\bmm_h^{\ell-1}}{r},
\nabla\Psi_h)_h.
\end{align*}
This symmetric positive definite linear system can be efficiently solved using, e.g., a multigrid algorithm \cite{Bramble93,XuASP}.

\subsubsection{Step B/C: scalar case}
\label{stepB1}

The next step is to take infinium of $L_{r,h}$ with respect to 
$\bmu_h^*$. Find $\bmu_h^{*,\ell}\in [W_h^k]^4$, such that 
it solves
\begin{align*}
\argmin_{\bmu_h^{*}\in [W_h^k]^4}
 F_h^*(\bmu_h^*)-(\bmu_h^{\ell-1}, \bmu_h^*)_h+\frac{r}{2}(\mathcal{D}\Phi_h^{\ell}-\bmu_h^*,
\mathcal{D}\Phi_h^{\ell}-\bmu_h^*)_h.
\end{align*}
Without loss of generality, we abuse the notation and denote
$\mathcal{D}\Phi_h^\ell$ as its interpolation onto the space $[W_h^k]^4$.
We further denote 
\begin{align}
\label{ubar}
\overline{\bmu}_h:=\mathcal{D}\Phi_h^\ell+\frac{1}{r}\bmu_h^{\ell-1}
\in [W_h^k]^4.
\end{align}
Then the above minimization problem is equivalent to 
\begin{align}
\label{stepB}
\argmin_{\bmu_h^{*}\in [W_h^k]^4}
 F_h^*(\bmu_h^*)+\frac{r}{2}(\bmu_h^*-\overline{\bmu}_h,\bmu_h^*-\overline{\bmu}_h)_h.
\end{align}
After this minimizer is computed, the last step is to update the Lagrangian multiplier
$\bmu_h^{\ell}$ according to \eqref{stepC}:
\begin{align}
\label{stepCh}
\bmu^\ell_h = \bmu^{\ell-1}_h+r(\mathcal{D}\Phi^\ell_h-\bmu^{*,\ell}_h)
=r(\overline{\bmu}_h-\bmu^{*,\ell}_h)\in [W_h^k]^4,
\end{align}
where we used the definition \eqref{ubar} in the last step.

Due to the complicated form of the energy \eqref{ener}, it might be  challenging to compute an explicit expression of the convex conjugate $F_h^*(\bmu_h^*)$.
Here we present a practical way to solve the minimization problem \eqref{stepB} without explicitly computing this convex conjugate using duality.
The main idea is presented in the next result.

\begin{proposition}
\label{thm:stepB}
Let $\bmu_h^{*,\ell}\in [W_h^k]^4$
be the minimizer to the problem \eqref{stepB}, and let 
$\bmu^\ell_h$ be given according to \eqref{stepCh}.
Then, 
$\bmu^\ell_h$ is the minimizer to the following problem
\begin{align}
\label{dualB}
\bmu^\ell_h = 
\argmin_{\bmu_h\in [W_h^k]^4}
 F_h(\bmu_h)+\frac{1}{2r}(\bmu_h-r\overline{\bmu}_h,\bmu_h-r\overline{\bmu}_h)_h,
\end{align}
which we refer to as the dual problem of \eqref{stepB}. 
Furthermore, there holds 
\begin{align}
\label{utous}
\bmu_h^{*,\ell}=\overline{\bmu}_h-\bmu^\ell_h/r.
\end{align}
\end{proposition}
\begin{proof}
The equation \eqref{utous} is a simple rewriting of \eqref{stepCh}.
Let us now prove \eqref{dualB}.
By definition \eqref{Fhs}, we have $\bmu_h^{*,\ell}$
is part of the saddle point solution 
\begin{align}
\label{pd}
\inf_{\bmu_h^{*}\in[W_h^k]^4}\sup_{\bmu_h\in[W_h^k]^4}
(\bmu_h,\bmu_h^*)_h-F_h(\bmu_h)
+\frac{r}{2}(\bmu_h^*-\overline{\bmu}_h,\bmu_h^*-\overline{\bmu}_h)_h.
\end{align}
Taking the derivative with respect to $\bmu_h^*$ in the above expression, we get
\[
\bmu_h^*= \overline{\bmu}_h-\bmu_h/r.
\]
Plugging this expression back to \eqref{pd}, we easily see that 
the primal variable ${\bmu}_h$ is the minimizer to the 
dual problem \eqref{dualB}. By \eqref{stepCh}, it is clear that
this optimizer is nothing but the solution $\bmu_h^\ell$.
This completes the proof.
\end{proof}
Proposition \ref{thm:stepB}  suggests to first solve for the primal variable $\bmu_h^\ell$ using the minimization problem \eqref{dualB}, then update 
$\bmu_h^{*,\ell}$ using \eqref{utous}, which is the approach we adopt in our implementation. 
It is in general more convenient than
the (equivalent) original ALG2 algorithm that first solve for the dual variable $\bmu_h^{*,\ell}$ using \eqref{stepB} then update 
$\bmu_h^\ell$ using \eqref{stepCh}, which requires the computation of the dual functional \eqref{Fhs}. 

Next, using the particular form of $F_h$ in \eqref{Fh}, we 
show that the minimization problem \eqref{dualB} can be efficiently solved by first locally expressing 
flux $m_{h}^0, m_{h}^1$  and source $s_{h}$ in terms of density $\rho_{h}$ and then solving a nonlinear optimization problem 
for $\rho_{h}$ alone.
We record this procedure in the following result.
\begin{proposition}
\label{thm:uh0}
Let $\bmu_h^\ell$ be the solution to 
\eqref{dualB}. 
Then there holds
\begin{align}
\label{u123}
m_{h}^{0,\ell} = &\; \frac{r\rho_{h}^\ell}{r+\rho_{h}^\ell}\overline{m}_{h}^0,
\quad
m_{h}^{1,\ell} = \; \frac{r\rho_{h}^\ell}{r+\rho_{h}^\ell}\overline{m}_{h}^1,\quad
s_{h}^\ell = \; \frac{rV_2(\rho_h^\ell)}{r+V_2(\rho_h^\ell)}\overline{s}_{h},
\end{align}
where 
\[
\overline{\bmu}_h=(\overline{\rho}_h, 
\overline{m}_h^0, \overline{m}_h^1, 
\overline{s}_h),
\]
and $\rho_h^\ell$
is the minimizer to the following reduced problem:
\begin{align}
\label{dualB-red}
\argmin_{\rho_h\in W_h^k}&\;
 \frac{1}{2r}\left(|\rho_h-r\overline{\rho}_{h}|^2,1\right)_h+
\left(\frac{r^2(|\overline{m}_{h}^0|^2+|\overline{m}_{h}^1|^2)}{2(r+\rho_h)}
 ,1\right)_h\nonumber\\
&\; \quad
+\left(\frac{r^2\,|\overline{s}_{h}|^2}{2(r+V_2(\rho_h))}
 ,1\right)_h
 +\Delta t\, \mathcal{E}_h(\rho_h).
\end{align}
\end{proposition}
\begin{proof}
The derivatives of 
the functional in \eqref{dualB} at the saddle point vanishes. 
Taking derivatives with respect to $m_{h}^0,m_{h}^1$ and $s_{h}$, we get the relations 
\eqref{u123}. Plugging these relations back to \eqref{dualB} and simplifying, 
we get the optimziation problem \eqref{dualB-red} for $\rho_h^\ell$.
\end{proof}

\begin{remark}[On pointwise update for \eqref{dualB-red}]
\label{rk:pu}
The problem \eqref{dualB-red} can be solved by computing its critical point.
Taking the variation of the function in \eqref{dualB-red}
with respect to $\rho_h$, we have 
\begin{align}
\label{critical}
\frac{1}{r}(\rho_h-r\overline{\rho}_{h})
-
\frac{r^2(|\overline{m}_{h}^0|^2+|\overline{m}_{h}^1|^2)}{2(r+\rho_h)^2}
-
\frac{r^2V_2'(\rho_h)\overline{s}_{h}^2}{2(r+V_2(\rho_h))^2}
+\Delta t \,\frac{\delta\mathcal{E}_h}{\delta \rho}
(\rho_h) = 0.
\end{align}
By the choice of the function space \eqref{fes-W}, it is clear that \eqref{critical}
is satisfied on all quadrature points $\bmxi_i$ for $1\le i\le N_W$.
Using definition of the energy \eqref{Eh}, we have 
\[
\frac{\delta \mathcal{E}_h}{\delta \rho}(\rho_h)
= \alpha U_m'(\rho_h)+ V(x)+W*\rho_h. 
\]
In the absence of interaction kernel where $W(x)=0$, the equation \eqref{critical}  can be solved  in a pointwise fashion
per quadrature point thanks to the particular choice of the nodal basis \eqref{W-basis} for the space \eqref{fes-W}, using, e.g., Newton's method.

On the other hand, when aggregation effects are included, the term 
$W*\rho_h$ prohibits such pointwise update due to the nonlocal effect of this convolution. 
In this case, we treat the convolution term  $W*\rho_h$ explicitly in \eqref{critical} by evaluating it at the previous time step, i.e., 
\[
W*\rho_h\approx W*\rho_h^{\text{old}},
\]
and then solve the modified pointwise local problem \eqref{critical}
using the Newton's method.
This is the choice we use in all our simulation results with aggregation effects.
Similar treatment was used in, e.g., \cite{Carlier16,benamou2016augmented}.
\end{remark}

\begin{remark}[On convexity]
\label{rk:convex}
Let us briefly comment on convexity of the problem \eqref{dualB-red}. 
When aggregation effects are included, we extrapolate the nonlocal convolution term according to Remark~\ref{rk:pu}. 
The problem \eqref{dualB-red} is  a pointwise minimization problem per quadrature point. 
Taking its second-order variation, we obtain
\begin{align}
\label{hess}
\frac{1}{r}+
\frac{r^2(|\overline{m}_{h}^0|^2+|\overline{m}_{h}^1|^2)}{(r+\rho_h)^3}
+
\frac{r^2|\overline{s}_{h}|^2\left(2V_2'(\rho_h)^2-(r+V_2(\rho_h))V_2''(\rho_h)\right)}{2(r+V_2(\rho_h))^3}
+\alpha\Delta t U_m''(\rho_h)
\end{align}
It is clear that the first, second, and last term of the above
expression are always nonnegative as long as $\rho_h\ge0$.
Moreover, if 
\begin{align}
\label{convex}
2V_2'(\rho_h)^2-(r+V_2(\rho_h))V_2''(\rho_h)\ge0,
\end{align}
then the third term is also nonnegative.
For such a choice of mobility $V_2$, the minimization problem is convex, and uniqueness of the solution is  guaranteed unconditionally for any time step size $\Delta t$. 
In the absence of aggregation effects,  the overall ALG2 algorithm with $V_2$ satisfying \eqref{convex}
can also be shown to be unconditionally convergent; see, e.g., \cite{Eckstein92}.

We note that the convexity condition \eqref{convex} is ensured if we take 
 $V_2(\rho)=c\rho^\gamma$ for $c>0$ and $0\le \gamma\le 1$, 
 or 
 $V_2(\rho)=\frac{\rho-\bar{\rho}}{\log(\rho)-\log(\bar{\rho})}$
 for any $\bar{\rho}>0$. The latter choice will be used in the system case.
  On the other hand, 
the mobility $V_2(\rho)=\frac{\rho(1-\rho)}{\log(\rho)}$ for the Fisher-KPP equation \eqref{kpp} does not satisfy the convexity condition \eqref{convex}. 
For this case, we may use a small time step size $\Delta t$ to get a stable simulation.

  We finally note that small time step size $\Delta t$ may also be needed for the general case with an interaction potential $W$, where extrapolation is used to approximate the  problem \eqref{dualB-red}  as mentioned in Remark \eqref{rk:pu}. 
\end{remark}

For completeness, we collect one iteration of this algorithm as follows.
\begin{algorithm}[H]
\caption{One iteration of ALG2 algorithm for \eqref{aug-mfp-h}.}
\label{alg:2}
\begin{algorithmic}
\STATE $\bullet$ Step A: update $\Phi_h^\ell$.
Find $\Phi_h^\ell\in V_h^k$ such that the equation \eqref{scalar-rdh} holds.
\STATE $\bullet$ Step B/C: update $\bmu_h^{\ell}, \bmu_h^{*,\ell}$.
First, find $\rho_h^\ell$ such that it is the minimizer to
\eqref{dualB-red}. 
Then update $m_{h}^{0,\ell}, m_{h}^{1,\ell}, s_{h}^{\ell}$ according to \eqref{u123}.
Finally, update $\bmu_h^{*,\ell}$ according to \eqref{utous}.
\end{algorithmic}
\end{algorithm}
We  note that positivity of density approximation $\rho_h$
 can be easily enforced in the pointwise
 optimization problem \eqref{dualB-red}.

\subsection{High-order FEM for strongly reversible reaction diffusion systems}
We now present the high-order FEM discretization of 
the variational time implicit scheme \eqref{OJKOs}
and discuss its practical (modified) ALG2 implementation.
Given time step size $\Delta t>0$ and 
density approximations \[\bmr_h^{\text{old}}=(\rho_{1,h}^{\text{old}},\cdots, \rho_{M,h}^{\text{old}})
\in [W_h^k]^M\] at the previous time step, 
find $\bmu_h, 
\bmu_h^*\in [W_h^k]^{3M+R}$, and $\bmp_h\in [V_h^{k}]^M$, such that
\begin{align}
\label{aug-mfp-hs}
\inf_{\bmu_h\in [W_h^k]^{3M+R}} \;\;
    \sup_{\bmp_h\in [V_h^{k}]^M, \bmu_h^*\in [W_h^k]^{3M+R}} 
   \underline{L_{r,h}}(\bmp_h, \bmu_h, \bmu_h^*),
\end{align}
where  
\[
\bmu_h=(\rho_{1,h}, m_{1,h}^0, m_{1,h}^1,
\cdots, 
\rho_{M,h}, m_{M,h}^0, m_{M,h}^1,
s_{1,h},\cdots, s_{R,h})
\]
is the collection of densities $\bmr_h$, fluxes
\[\bmm_h=(m_{1,h}^0, m_{1,h}^1,\cdots, m_{M,h}^0, m_{M,h}^1),\] and source terms 
$\bms_h=(s_{1,h}, \cdots, s_{R,h})$, 
$\bmu_h^*$ is its dual, $\bmp_h=(\Phi_{1, h}, \dots, \Phi_{M,h})$,
and the discrete augmented Lagrangian is 
\begin{align}
\label{LrhZ}
\underline{L_{r,h}}(\bmp_h, \bmu_h, \bmu_h^*):= &\;
   \underline{F_h}^*(\bmu_h^*)
+\underline{G_h}(\bmp_h)+(\bmu_h, \underline{\mathcal{D}}\bmp_h-\bmu_h^*)_h\nonumber\\
&\;+\frac{r}{2}(\underline{\mathcal{D}}\bmp_h-\bmu_h^*,
\underline{\mathcal{D}}\bmp_h-\bmu_h^*)_h.
\end{align}
Here 
the operators
\begin{align}
\label{ops}
\underline{\mathcal{D}}\bmp_h:=&\;\Big(-\Phi_{1,h}, 
\partial_{x_0}\Phi_{1, h},
\partial_{x_1}\Phi_{1,h},
\cdots, 
-\Phi_{1,h}, 
\partial_{x_0}\Phi_{1, h},
\partial_{x_1}\Phi_{1,h},\nonumber\\
&\;
\quad\quad\sum_{i=1}^M(\alpha_i^1-\beta_i^1)\Phi_{i,h}, \cdots,
\sum_{i=1}^M(\alpha_i^R-\beta_i^R)\Phi_{i,h}
\Big),\\
\label{Ghs}
\underline{G_h}(\bmp_h)
:=&\;\sum_{i=1}^M(\rho_{i,h}^{\text{old}}, \Phi_{i,h})_h,\\
\label{Fhrs}
   \underline{F_h}^*(\bmu_h^*)
:=&\;
\sup_{
\bmu_h\in [W_h^k]^{3M+R}
}(\bmu_h^*, \bmu_h)_h- 
   \underline{F_h}(\bmu_h),
\end{align}
and $\underline{F_h}$ is given as 
\begin{align}
\label{Fhr}
\underline{F_h}(\bmu_h):=&\;
\left(\sum_{i=1}^M\frac{|m_{i,h}^0|^2+|m_{i,h}^1|^2}{2V_{1,i}(\rho_{i,h})}+
\sum_{p=1}^R\frac{|s_{p,h}|^2}{2V_{2,p}(\bmr_h)}, 1\right)_h
\nonumber\\
&\;
+\Delta t \sum_{i=1}^M\mathcal{E}_{i,h}(\rho_{i,h}),
\end{align}
where the mobility functions are given in \eqref{mobs} and 
the discrete energy 
\[
\mathcal{E}_{i,h}(\rho_{i,h})=(\rho_{i,h}(\log(\rho_{i,h})-1), 1)_h.
\]


We now discuss a modified implementation 
of the ALG2 algorithm \ref{alg:1} for the saddle point system \eqref{aug-mfp-hs}, where further componentwise splitting is introduced to drive down the overall computational cost.

\subsubsection{Step A: system case}
\label{stepA2}
Taking infinium of $\underline{L_{r,h}}$ with respect to $\bmp_h$, we arrive at 
a coupled system of constant coefficient reaction-diffusion equations:
find $\bmp_h^\ell\in [V_h^k]^M$ such that 
\begin{align}
\label{sys-rdh}
(\underline{\mathcal{D}}\bmp_h^{\ell}, \underline{\mathcal{D}}{\bm\Psi}_h)_h=
(\bmu_h^{*,\ell-1}-\frac{1}{r}\bmu_h^{\ell-1}, \underline{\mathcal{D}}{\bm\Psi}_h)_h
- \frac{1}{r}(\bmr_h^{\text{old}}, {\bm\Psi}_h)_h,
\end{align}
for all ${\bm\Psi}_h\in [V_h^k]^M$.
Using the definition in \eqref{ops}, we 
write the above system back using the physical variables: 
\begin{align*}
&(\Phi_{i,h}^{\ell}, \Psi_{i,h})_h+
(\nabla\Phi_{i,h}^{\ell}, \nabla\Psi_{i,h})_h
+\sum_{p=1}^R\sum_{j=1}^M\left((\alpha_j^p-\beta_j^p)\Phi_{j,h}^\ell,
(\alpha_i^p-\beta_i^p)\Psi_{i,h} 
\right)_h\\
&=\;
(-\rho_{i,h}^{*,\ell-1}+\frac{\rho_{i,h}^{\ell-1}-\rho_{i,h}^{\text{old}}}{r},
\Psi_{i,h})_h
\;+(\bmm_{i,h}^{*,\ell-1}-\frac{\bmm_{i,h}^{\ell-1}}{r},
\nabla\Psi_{i,h})_h\\
&\; \quad\quad+\sum_{p=1}^R\left(
s_{r,h}^{*,\ell-1}-
\frac{1}{r}s_{r,h}^{\ell-1},
(\alpha_i^p-\beta_i^p)\Psi_{i,h} 
\right)_h,
\end{align*}
for all $1\le i\le M$.
This coupled linear system might be expensive to solve. 
Here we propose to solve these $M$ equations in parallel by treating the coupling term on the left hand side of the above equation explicitly.
Specifically, for each $1\le i\le M$, we compute $\Phi_{i,h}\in V_h^k$
such that it solves the following scalar linear reaction-diffusion equation:
\begin{align}
\label{lag-coef}
&(\Phi_{i,h}^{\ell}, \Psi_{i,h})_h+
(\nabla\Phi_{i,h}^{\ell}, \nabla\Psi_{i,h})_h
+\sum_{p=1}^R\left((\alpha_i^p-\beta_i^p)\Phi_{i,h}^\ell,
(\alpha_i^p-\beta_i^p)\Psi_{i,h} 
\right)_h\nonumber\\
&=\;
(-\rho_{i,h}^{*,\ell-1}+\frac{\rho_{i,h}^{\ell-1}-\rho_{i,h}^{\text{old}}}{r},
\Psi_{i,h})_h
\;+(\bmm_{i,h}^{*,\ell-1}-\frac{\bmm_{i,h}^{\ell-1}}{r},
\nabla\Psi_{i,h})_h\nonumber\\
&\; \quad\quad+\sum_{p=1}^R\left(
s_{r,h}^{*,\ell-1}-
\frac{1}{r}s_{r,h}^{\ell-1},
(\alpha_i^p-\beta_i^p)\Psi_{i,h} 
\right)_h\nonumber\\
&\;\quad\quad-\sum_{p=1}^R\sum_{\overset{j=1}{j\not=i}}^M\left(
(\alpha_j^p-\beta_j^p)\Phi_{j,h}^{\ell-1},
(\alpha_i^p-\beta_i^p)\Psi_{i,h} 
\right)_h,
\end{align}
for all $\Phi_{i,h}\in V_h^k$. These are $M$ decoupled scalar constant-coefficient linear reaction-diffusion equations, which are easy to solve.

One may also solve the equation \eqref{lag-coef} sequentially (in a Gauss-Seidel manner), which uses the updated $\Phi_{j,h}^\ell$ for $j< i$ when computing the variable $\Phi_{i,h}^\ell$.

\subsubsection{Step B/C: system case}
\label{stepB2}
Similar to the scalar case in Subsection \ref{stepB1}, 
we first compute the solutions $\bmu_h^\ell$ 
according to the following system version of \eqref{dualB}:
\begin{align}
\label{dualBs}
\bmu^\ell_h = 
\argmin_{\bmu_h\in [W_h^k]^{3M+R}}
 \underline{F_h}(\bmu_h)+\frac{1}{2r}(\bmu_h-r\overline{\bmu}_h,\bmu_h-r\overline{\bmu}_h)_h,
\end{align}
where 
\[
\overline{\bmu}_h:=\underline{\mathcal{D}}\bmp_h^\ell+\frac{1}{r}\bmu_h^{\ell-1},
\]
with the understanding that 
$\underline{\mathcal{D}}\bmp_h^\ell$ is its interpolation onto the space $[W_h^k]^{3M+R}$, 
and then update $\bmu_h^{*,\ell}$ according to 
\begin{align}
\label{utouss}
\bmu_h^{*,\ell}=\overline{\bmu}_h-\bmu^\ell_h/r.
\end{align}

Again, we solve the 
problem \eqref{dualBs} by first locally expressing all other variables in terms of the densities, and then solve pointwise optimization problems 
for these densities on each quadrature point.
\begin{proposition}
    \label{thm:uh0s}
Let $\bmu_h^\ell$ be the solution to 
\eqref{dualBs}. 
Then there holds
\begin{align}
\label{u123s}
m_{i,h}^{k,\ell} = &\; \frac{r V_{1,i}(\rho_{i,h}^\ell)}{r+V_{1,i}(\rho_{i,h}^\ell)}\overline{m}_{i,h},
\forall k=0,1, \text{and } 1\le i\le M,\\ 
\label{u123X}
s_{p,h}^\ell = &\; \frac{rV_{2,p}(\bmr_h^\ell)}{r+
V_{2,p}(\bmr_h^\ell)}\overline{s}_{p,h},\quad \forall 1\le p\le R,
\end{align}
and 
the collection of densities $\bmr_h^\ell$
is the minimizer of the following reduced problem:
\begin{align}
\label{dualB-reds}
\argmin_{\bmr_{h}\in [W_h^k]^M}&\;
\sum_{i=1}^M \frac{1}{2r}\left(|\rho_{i,h}-r\overline{\rho}_{i,h}|^2,1\right)_h
+\sum_{i=1}^M  \left(\frac{r^2(|\overline{m}_{i,h}^0|^2+\overline{m}_{i,h}^0|^2)}{2(r+
V_{1,i}(\rho_{i,h}))}
 ,1\right)_h
\nonumber\\
&\; 
+\sum_{p=1}^R\left(\frac{r^2|\overline{s}_{p,h}|^2}{2(r+V_{2,p}(\bmr_{h}))}
 ,1\right)_h+
\Delta t\,
 \sum_{i=1}^M \mathcal{E}_{i,h}(\rho_{i,h}).
\end{align}
\end{proposition}

By the choice of the integration rule space \eqref{fes-W} and its nodal basis \eqref{W-basis}, 
it is clear that the minimization problem \eqref{dualB-reds}
 can be solved in a pointwise fashion per quadrature point. 
On each quadrature point, it is an $M$-dimensional minimization problem, where the coupling is introduced in the 
reaction term in the second row of \eqref{dualB-reds}. 
Again, we propose to solve $M$ independent single-variable minimization problems in parallel by treating the reaction term semi-implicitly. 
Specifically, the solution 
$\rho_{i,h}^\ell$ for each $1\le i\le M$ is obtained by solving 
the following problems in parallel:
\begin{align}
\label{dualB-reds1D}
\argmin_{\rho_{i,h}\in W_h^k}&\;
\frac{1}{2r}\left(|\rho_{i,h}-r\overline{\rho}_{i,h}|^2,1\right)_h
+\left(\frac{r^2(|\overline{m}_{i,h}^0|^2+\overline{m}_{i,h}^0|^2)}{2(r+
V_{1,i}(\rho_{i,h}))}
 ,1\right)_h
\nonumber\\
&\; 
+\sum_{p=1}^R\left(\frac{r^2|\overline{s}_{p,h}|^2}{2(r+V_{2,p}(\tilde{\bmr}_{h}))}
 ,1\right)_h+
\Delta t\mathcal{E}_{i,h}(\rho_{i,h}).
\end{align}
Here 
\[
\tilde{\bmr}_{h}^i
=(\rho_{1,h}^{\ell-1},\cdots, 
\rho_{i-1,h}^{\ell-1}, 
\rho_{i,h}, 
\rho_{i+1,h}^{\ell-1},\cdots,
\rho_{M,h}^{\ell-1}), 
\]
i.e., all other densities are evaluated explicitly at level $\ell-1$.
By the choice of mobility functions in \eqref{mobs}, it is easy to show that 
the problem \eqref{dualB-reds1D} is convex and hence has a unique global minimizer.
We collect this modified ALG2 implementation in the following algorithm.

\begin{algorithm}[H]
\caption{One iteration of modified ALG2 algorithm for \eqref{aug-mfp-hs}.}
\label{alg:3}
\begin{algorithmic}
\STATE $\bullet$ Step A: update $\bmp_{h}^\ell$.
Find $\Phi_{i,h}^\ell\in V_h^k$ such that the equation \eqref{lag-coef} holds for each $1\le i\le M$.
\STATE $\bullet$ Step B/C: update $\bmu_h^{\ell}, \bmu_h^{*,\ell}$.
First, find $\rho_{i,h}^\ell$ such that it is the minimizer to
\eqref{dualB-reds1D} for each $1\le i\le M$. 
Then update $m_{i,h}^{k,\ell}$ for $k=0,1$ according to \eqref{u123s}
and update $s_{p,h}^\ell$ for $1\le p\le R$
according to \eqref{u123X}.
Finally, update $\bmu_h^{*,\ell}$ according to \eqref{utouss}.
\end{algorithmic}
\end{algorithm}
\subsection{High-order FEM for reversible reaction-diffusion systems with detailed balance}
For a reversible reaction-diffusion system with detailed balance, the spatial discretization and the corresponding practical ALG2 implementation 
are the same as the one in a strongly reversible case, with the only change that the discrete energy
now takes the following form:
\[
\mathcal{E}_{i,h}(\rho)=(\rho(\log(\kappa_i\rho)-1), 1)_h,
\]
where 
$\kappa_i>0$ depends on the reaction rates. 

Here a small modification (with a reduced cost) is need to simulate the 
reversible Gray-Scott model in Example \ref{rgs-model}
since it does not include diffusion for the last two species. 
Specifically, we do not need flux approximations for the last two species, and 
the variables and operators in the fully discrete algorithm 
\eqref{aug-mfp-hs} for the system \eqref{rGS} is recorded below for completeness:
\begin{subequations}
    \label{rGS-full}
    \begin{align}
    \bmu_h =&\; (\rho_{1,h}, m_{1,h}^0, m_{1,h}^1, 
    \rho_{2,h}, m_{2,h}^0, m_{2,h}^1,\rho_{3,h}, \rho_{4,h}, 
    s_{1,h},s_{2,h},s_{3,h}),\\
    \Phi_h = &\;(\Phi_{1,h},\Phi_{2,h},\Phi_{3,h},\Phi_{4,h}),\\
    \underline{\mathcal{D}}\bmp_h:=&\;\Big(-\Phi_{1,h}, 
\partial_{x_0}\Phi_{1, h},
\partial_{x_1}\Phi_{1,h},
-\Phi_{2,h}, 
\partial_{x_0}\Phi_{2, h},
\partial_{x_1}\Phi_{2,h},\nonumber\\
&\;-\Phi_{3,h}, 
-\Phi_{4,h},
\Phi_{1,h}-\Phi_{2,h}, 
\Phi_{2,h}-\Phi_{3,h}, 
\Phi_{1,h}-\Phi_{4,h} 
\Big),\\
\underline{G_h}(\bmp_h)
:=&\;\sum_{i=1}^4(\rho_{i,h}^{\text{old}}, \Phi_{i,h})_h,\\
   \underline{F_h}(\bmu_h)
:=&\;
\left(\sum_{i=1}^2\frac{|m_{i,h}^0|^2+|m_{i,h}^1|^2}{2V_{1,i}(\rho_{i,h})}+
\sum_{p=1}^3\frac{|s_{p,h}|^2}{2V_{2,p}(\bmr_h)}, 1\right)_h\nonumber\\
&\;+\Delta t \sum_{i=1}^4\mathcal{E}_{i,h}(\rho_{i,h}),
    \end{align}
\end{subequations}
where the parameters and mobility functions are given in \eqref{rGS-coo}.
Note that Step A of Algorithm \ref{alg:3} now becomes 
two scalar linear reaction-diffusion equation updates for 
$\Phi_{1,h}^\ell$
and 
$\Phi_{2,h}^\ell$, and two simple mass matrix updates 
for 
$\Phi_{3,h}^\ell$
and 
$\Phi_{4,h}^\ell$.

\section{Numerical experiments}
\label{sec:num}
In this section, we conduct comprehensive 2D experiments to show the efficiency and effectiveness of the proposed numerical algorithms. 
Throughout, we take the augmented Lagrangian parameter to be $r=1$, and perform 200 ALG iterations  in each time step for all test cases.
Our numerical simulations are performed using the open-source finite-element software {\sf NGSolve} \cite{Schoberl16}, \url{https://ngsolve.org/}.

\subsection{Spatial convergence rates}
\label{ex1}
We first consider the nonlinear Fokker-Plank equation 
\[
\partial_t \rho - \triangle \rho^3 = \nabla\cdot (\rho \,x),
\]
on the domain $\Omega=[-1,1]\times[-1,1]$
with homogeneous Neumann boundary conditions.
It is a Wasserestein gradient flow of the form \eqref{wg} with energy 
\[
\mathcal{E}(\rho):=\int_{\Omega}\left(\frac12\rho(x)^3+
\frac12(x_0^2+x_1^2)\rho(x) \right) dx,
\]
where $x=(x_0,x_1)$.
This problem reaches a steady state solution 
\[
\rho_{\text{steady}}(x_1, x_2) =\sqrt{\frac{(2C-(x_0^2+x_1^2))_+}{3}},
\]
that satisfies either 
\[
\frac{\delta \mathcal{E}}{\delta \rho}=\frac32\rho^2+\frac12(x_0^2+x_1^2)=C,
\] 
or $\rho=0$. 
Here the constant $C$ depends on the total mass of the initial condition, which we set to be $C=2$ so that the solution on $\Omega$ is positive and smooth.

We perform a mesh convergence study 
for the scheme \eqref{aug-mfp-h} using  Algorithm \ref{alg:2}
 with polynomial degree $k=1, 2, 4$ on a sequence of uniformly refined meshes. 
 The coarse mesh is of size $8\times 8$ for $k=1$, 
 $4\times 4$ for $k=2$, and $2\times 2$ for $k=4$, so that the total number of degrees of freedom for $\Phi$ is the same on each mesh level for different polynomial degrees.
 We take large time step size with $\Delta t=1$, and perform 10 time steps of simulation where the numerical solution reaches the steady state. The $L^2$-convergence in the density $\rho$
 is recorded in Table~\ref{tab:1}. We clearly observe the 
 $k+1$-th order of convergence for each case. 
 In particular, the higher order method leads to a smaller error when a same number of total degrees of freedom is used.

\begin{table}[tb]
\centering
\caption{Convergence rates of scheme~\ref{aug-mfp-h} 
with different polynomial degree $k$ applied to a 2D steady Fokker Plank equation.}
\begin{tabular}{c|ll|ll|ll}
\toprule
$\mathrm{dim}(V_h^k)$&
\multicolumn{2}{c|}{$k=1$} &
\multicolumn{2}{c|}{$k=2$} &
\multicolumn{2}{c}{$k=4$}\\
\midrule
81&2.362e-03 & -- &2.409e-04 & -- &2.628e-05&--\\
289&5.923e-04 & 2.00 &3.298e-05 & 2.87 &1.424e-06&4.21\\
1089&1.482e-04 & 2.00 &4.232e-06 & 2.96 &5.589e-08&4.67\\
4225&3.705e-05 & 2.00 &5.326e-07 & 2.99 &1.884e-09&4.89\\
\bottomrule
\end{tabular}
\label{tab:1}
\end{table}

\subsection{Aggregation-drift-diffusion equations}
\label{ex2}
We consider Wasserstein gradient flow \eqref{wg} with 
five choices of energies \eqref{ener} that include aggregation effects.
The specific form of the energy, along with the  
domain size $L$ where  the computational domain $\Omega=[-L,L]\times [-L, L]$, and the initial conditions are given in Table \ref{tab:2}.
Here $\chi_{[-3,3]\times[-3,3]}$ is the characteristic function on $[-3,3]\times[-3,3]$
for Case 5.
All cases were considered in \cite{LiWang22}, 
except Case 4 which adds an additional diffusion to the energy in Case 3.
\begin{table}[tb]
\centering
\caption{Example \ref{ex2}. Five choices of energies, domain size, 
and initial condition.}
\begin{tabular}{c|c|c|c|c|c}
\toprule
Case& $\alpha U_m(\rho)$ &
$V(x)$ &
$W(x)$ &$L$&I.C.
\\
\midrule
1 & 0 & 0 & $\frac{|x|^4}{4}-\frac{|x|^2}{2}$&1&
$\frac{25}{2\pi}\exp(-\frac{25}{2}|x|^2)
$
\\[.4ex]
2 & 0 & 0 & $\frac{|x|^2}{2}-\log(|x|)$&1.5
&
$\frac{25}{8\pi}\exp(-\frac{25}{8}|x|^2)
$
\\[.4ex]
3 & 0 & $-\frac14\log(|x|)$ & $\frac{|x|^2}{2}-\log(|x|)$&1.5
&
$\frac{25}{8\pi}\exp(-\frac{25}{8}|x|^2)
$
\\[.4ex]
4 & 
$0.1\rho^2$
& $-\frac14\log(|x|)$ 
& $\frac{|x|^2}{2}-\log(|x|)$&1.5
&
$\frac{25}{8\pi}\exp(-\frac{25}{8}|x|^2)
$
\\[.4ex]
5 & 
$0.1\rho^3$
& 0
& $-\exp(-|x|^2)/\pi$&4
&
$0.25\chi_{[-3,3]\times[-3,3]}
$
\\[.4ex]
\bottomrule
\end{tabular}
\label{tab:2}
\end{table}

Note that the interaction kernel $W(x)$ 
for Cases 2/3/4 is singular at zero.
Here we use a higher-order numerical integration rule, which avoids the evaluation of $W(x)$ at zero to compute the convolution
\[
W*\rho(\bmxi_i), \quad \forall 1\le i\le N_W,
\]
at the quadrature points $\{\bmxi_i\}_{i=1}^{N_W}$.
Fast Fourier transform is used to evaluate these 
convolutions all together. 

For all cases, we take the computational mesh to be a $32\times 32$ uniform square mesh, and use polynomial degree $k=4$ in the scheme 
\eqref{aug-mfp-h}. 
We take time step size $\Delta t=0.05$ for the first four cases, and 
$\Delta t = 0.5$ for the last case.
The final time of simulation is $T=10$
for Case 1, $T=3$ for Cases 2/3/4, and $T=15$
for Case 5.
Snapshots of the density contours at different times are shown in Figure \ref{fig:1}.
We find the results for Cases 1/2/3 and 5 are  qualitatively similar to the results reported in \cite{LiWang22}. 
In particular, Case 1 converges to a steady Dirac ring solution; Case 2 converges to a steady constant solution with a circular shape; Case 3 converges to a characteristic function for the torus due to the drift effects that pushes away the density from the origin; 
and the competition between median range aggregation with 
short/long range diffusion are observed for Case 5.
Moreover,  the diffusion effects of Case 4 comparing with Case 3 are also clearly seen. 
\begin{figure}[tb]
\centering
\subfigure[Case 1. Left to right time: $t=0.5, 1.5, 3.0, 6.0, 10$]{
\label{fig:11}
\includegraphics[width=0.192\textwidth]{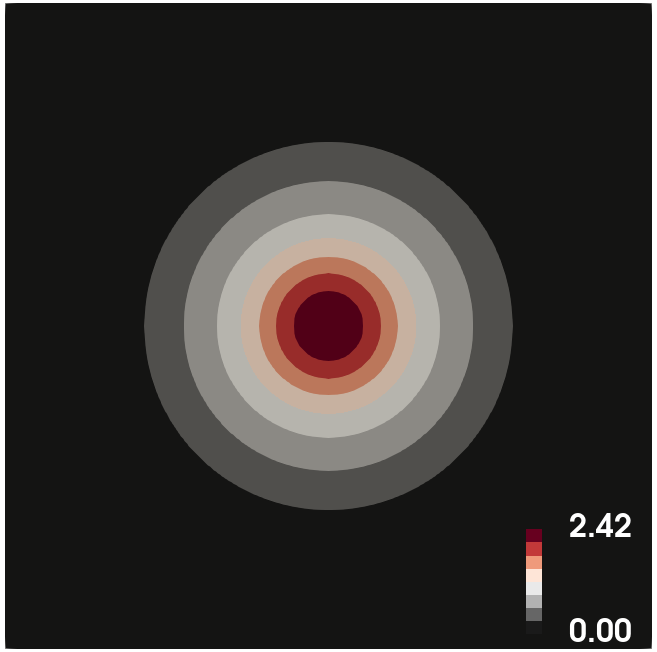}
\includegraphics[width=0.192\textwidth]{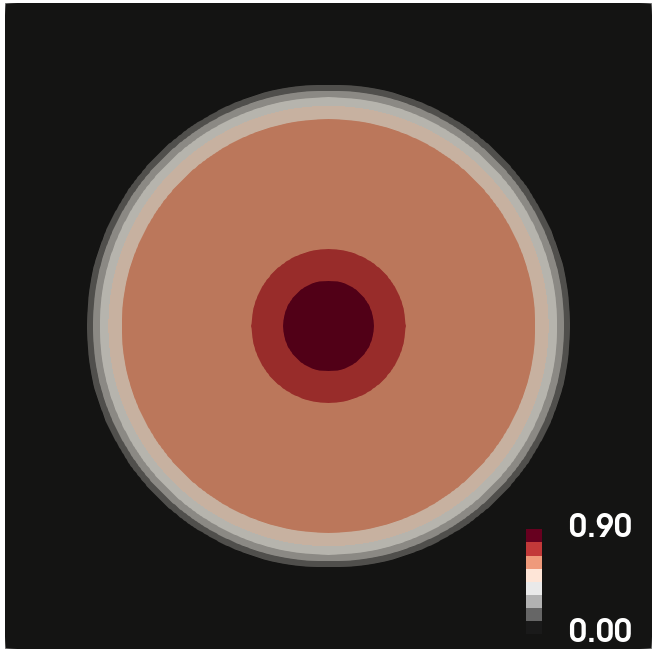}
\includegraphics[width=0.192\textwidth]{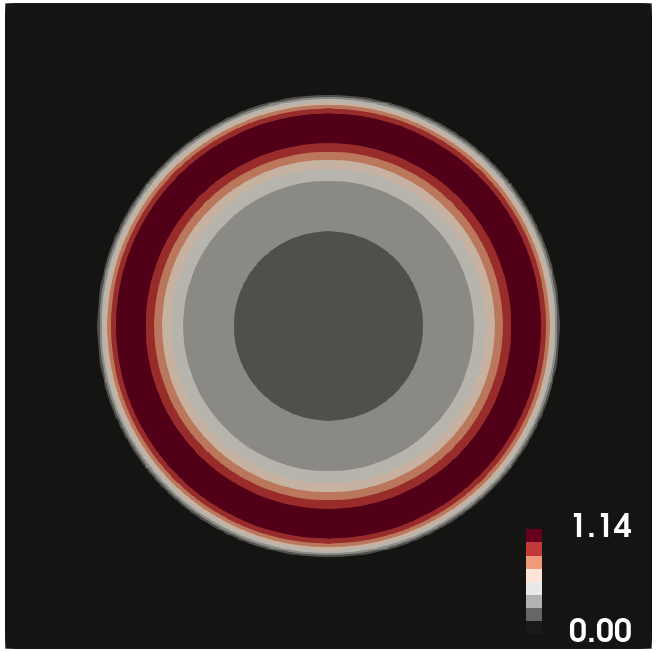}
\includegraphics[width=0.192\textwidth]{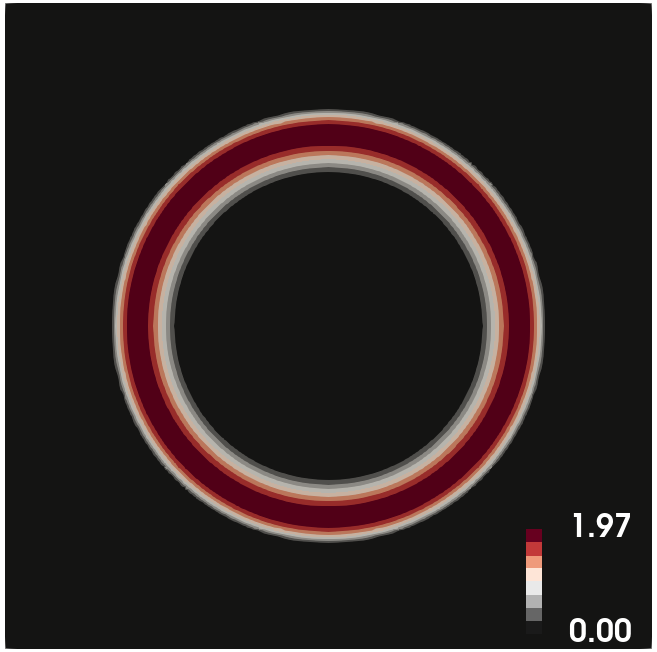}
\includegraphics[width=0.192\textwidth]{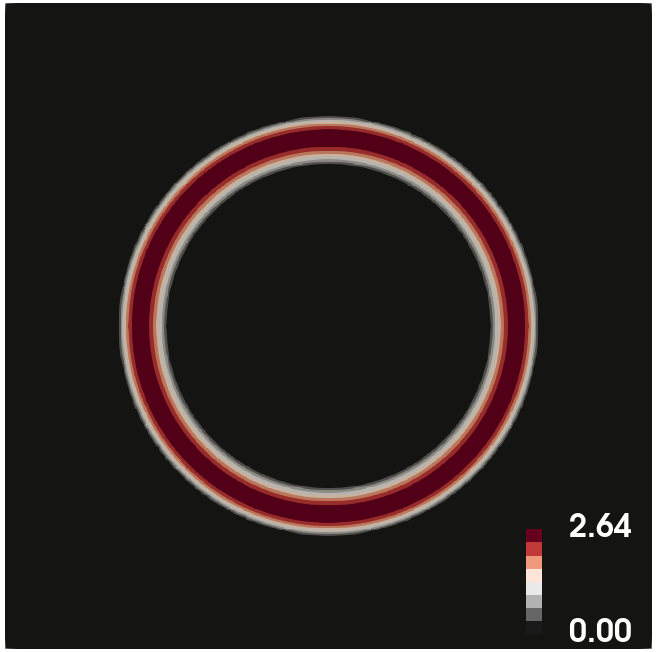}
}
\subfigure[Case 2. Left to right time: $t=0.2, 0.5, 1.5, 2.0, 3.0$]{
\label{fig:12}
\includegraphics[width=0.192\textwidth]{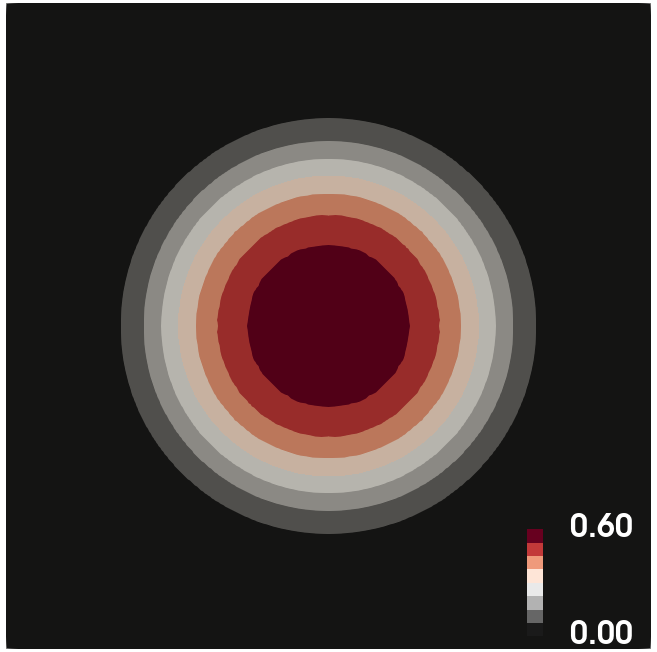}
\includegraphics[width=0.192\textwidth]{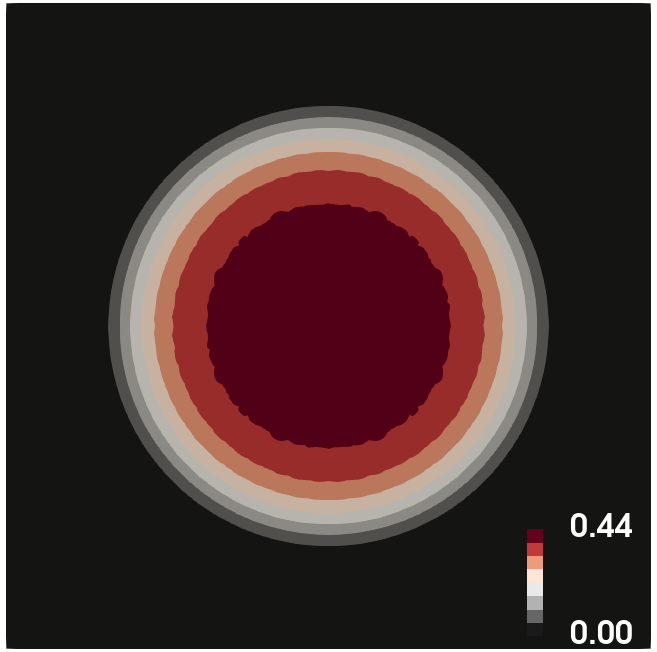}
\includegraphics[width=0.192\textwidth]{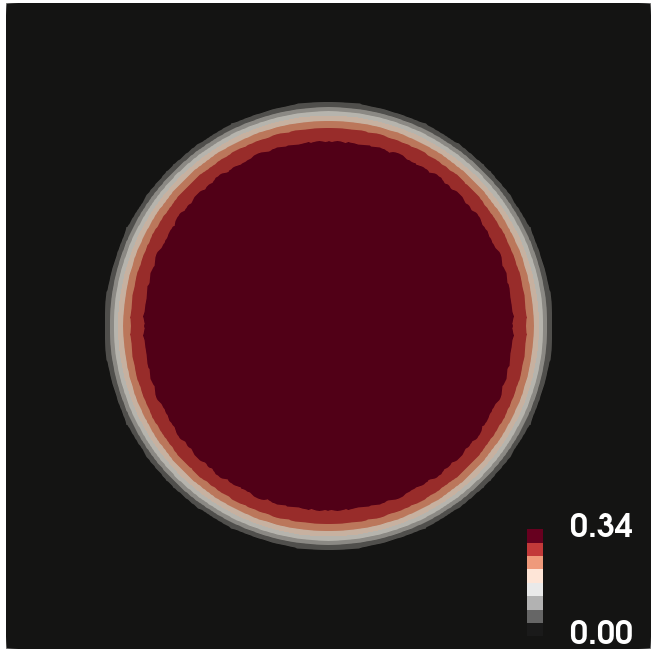}
\includegraphics[width=0.192\textwidth]{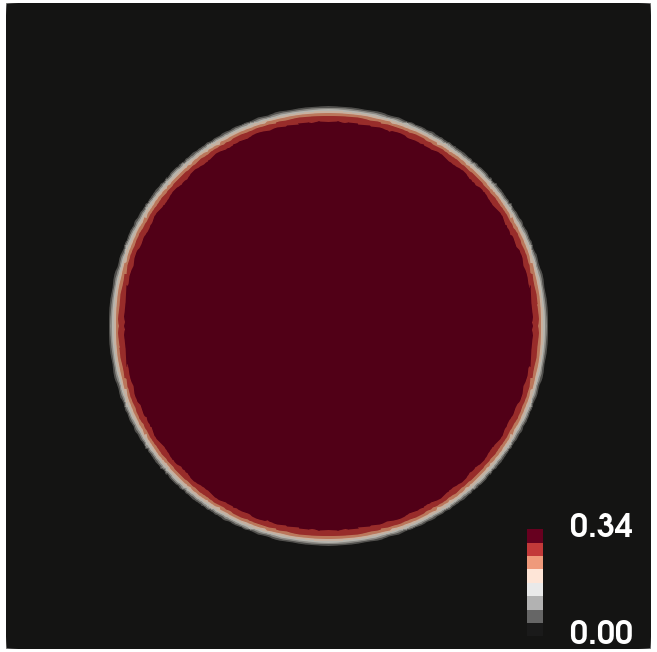}
\includegraphics[width=0.192\textwidth]{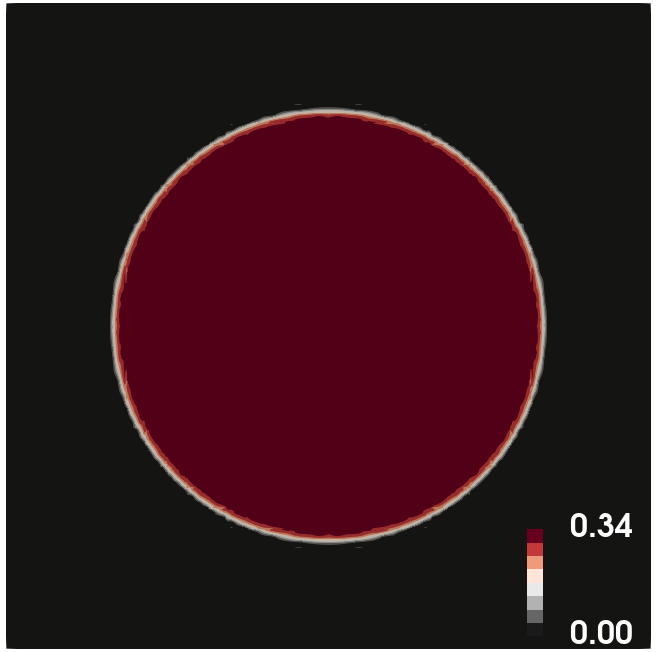}
}
\subfigure[Case 3. Left to right time: $t=0.2, 0.5, 1.5, 2.0, 3.0$]{
\label{fig:13}
\includegraphics[width=0.192\textwidth]{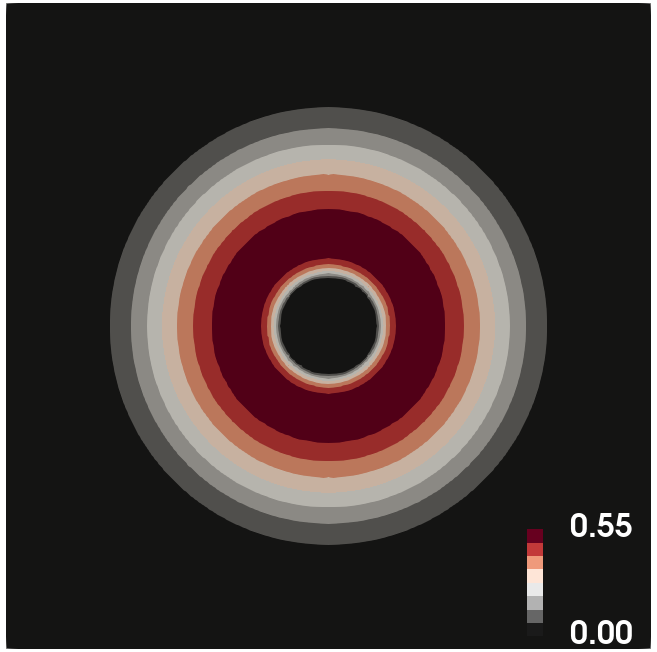}
\includegraphics[width=0.192\textwidth]{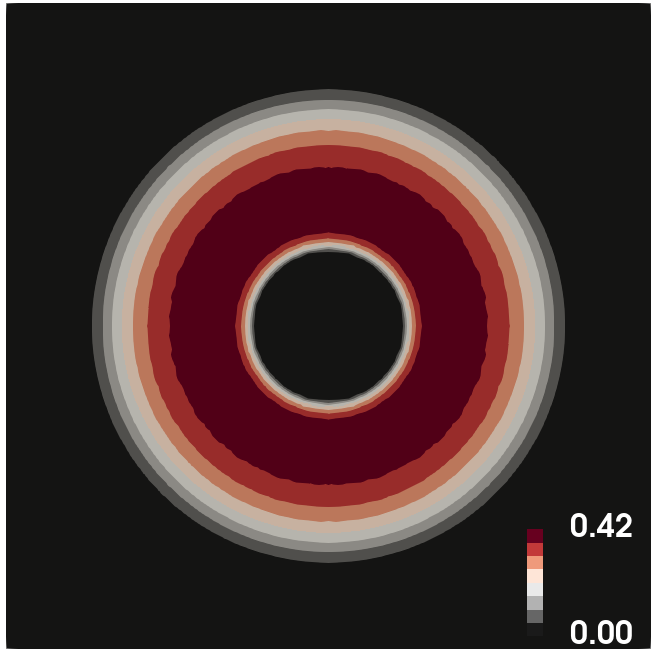}
\includegraphics[width=0.192\textwidth]{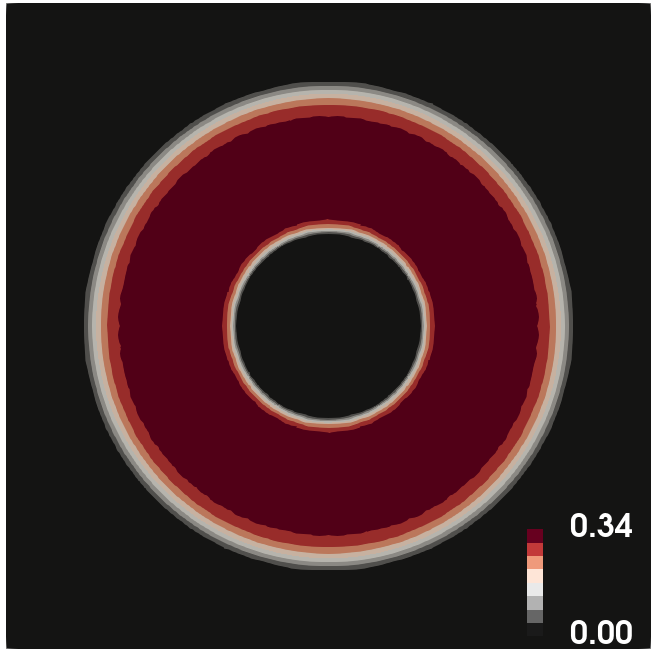}
\includegraphics[width=0.192\textwidth]{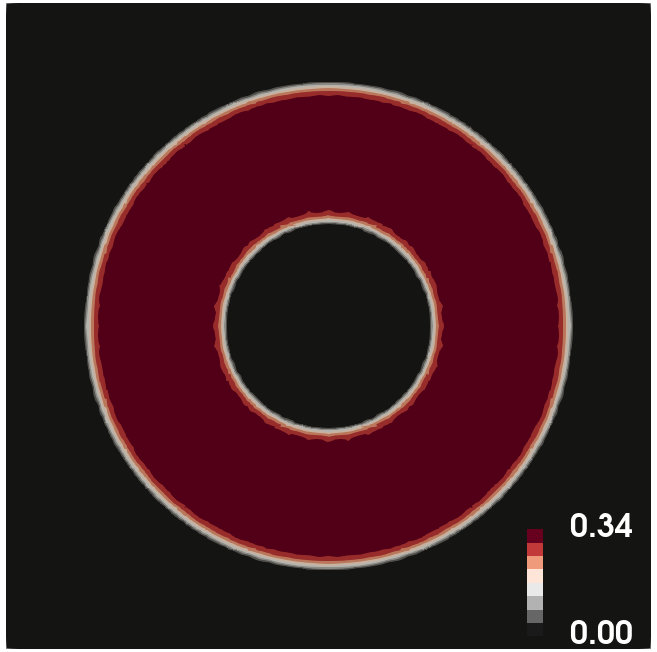}
\includegraphics[width=0.192\textwidth]{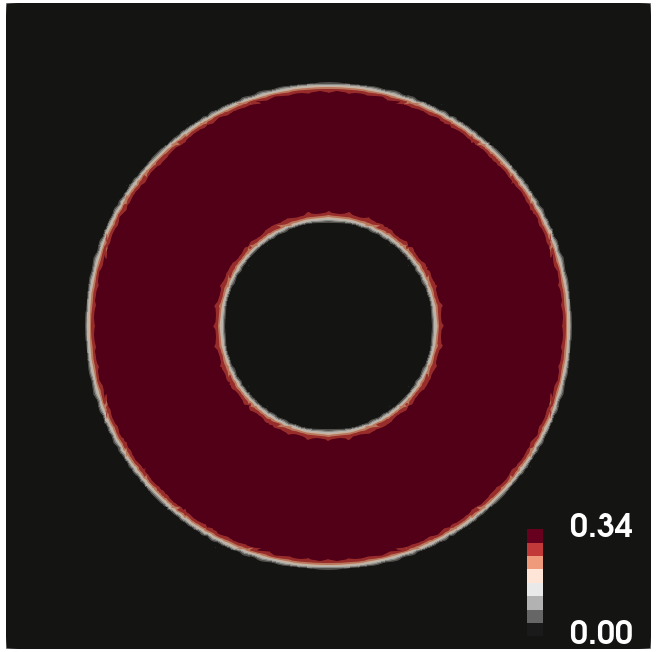}
}
\subfigure[Case 4. Left to right time: $t=0.2, 0.5, 1.5, 2.0, 3.0$]{
\label{fig:14}
\includegraphics[width=0.192\textwidth]{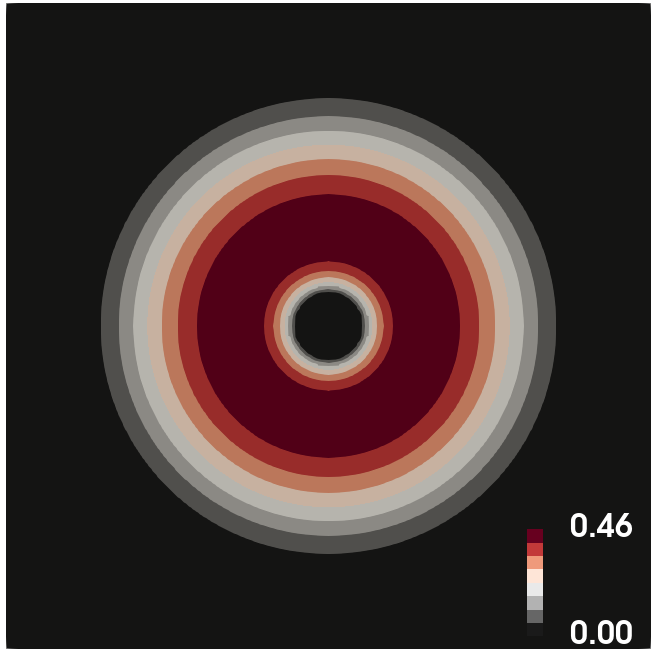}
\includegraphics[width=0.192\textwidth]{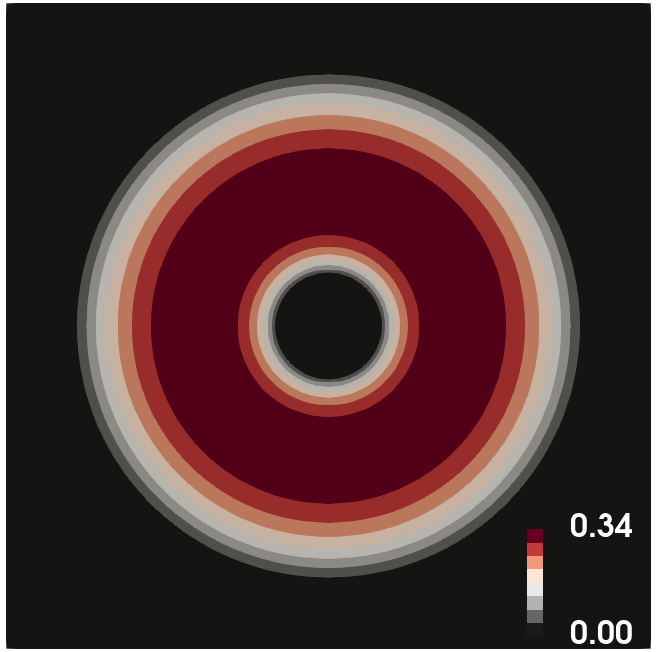}
\includegraphics[width=0.192\textwidth]{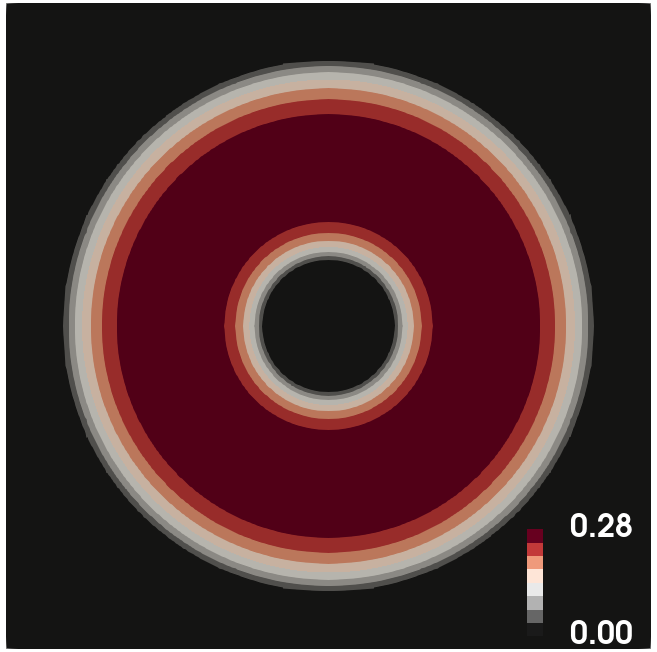}
\includegraphics[width=0.192\textwidth]{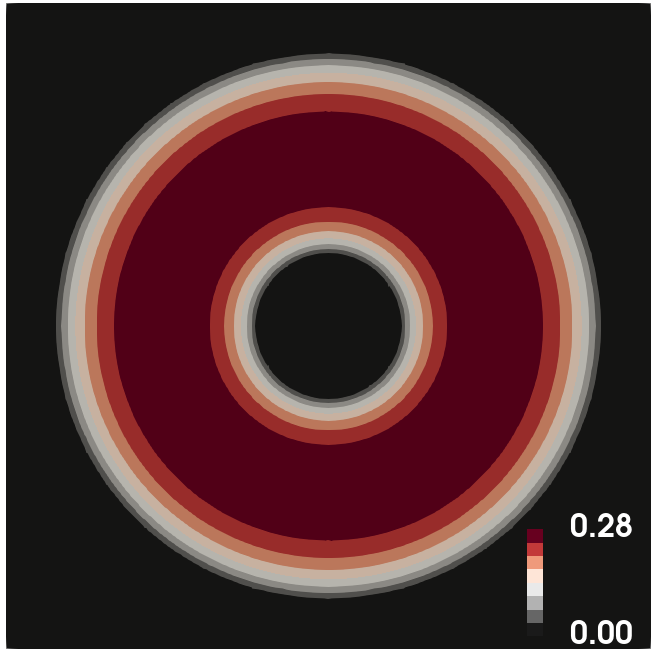}
\includegraphics[width=0.192\textwidth]{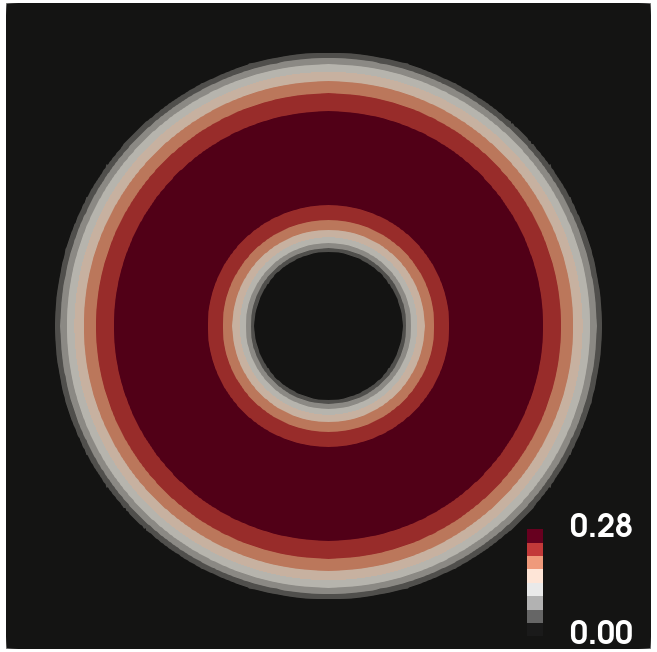}
}
\subfigure[Case 5. Left to right time: $t=2, 4, 6, 10, 15$]{
\label{fig:15}
\includegraphics[width=0.192\textwidth]{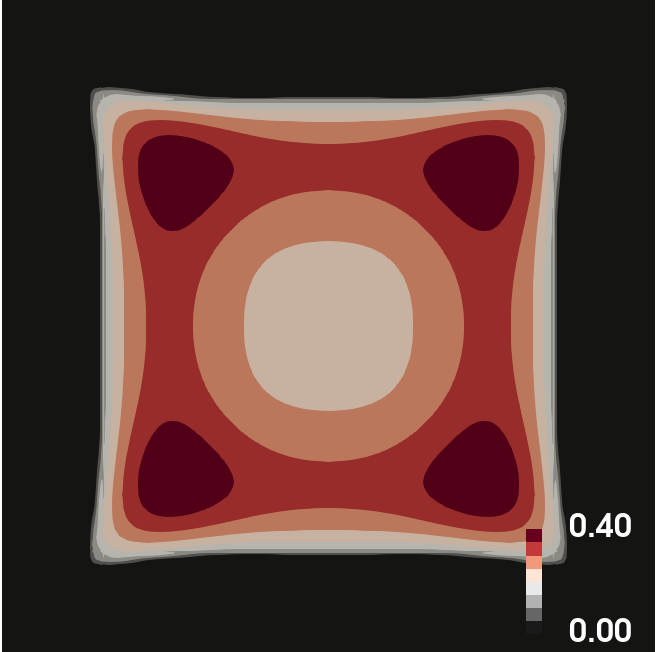}
\includegraphics[width=0.192\textwidth]{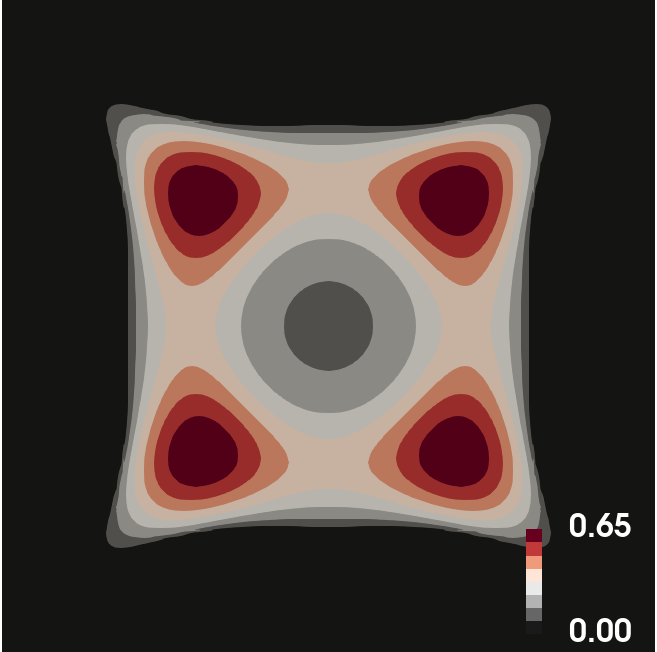}
\includegraphics[width=0.192\textwidth]{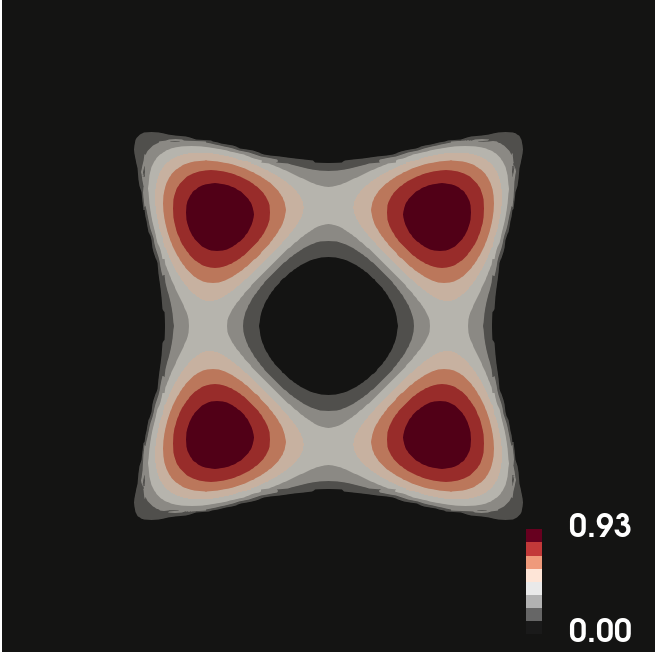}
\includegraphics[width=0.192\textwidth]{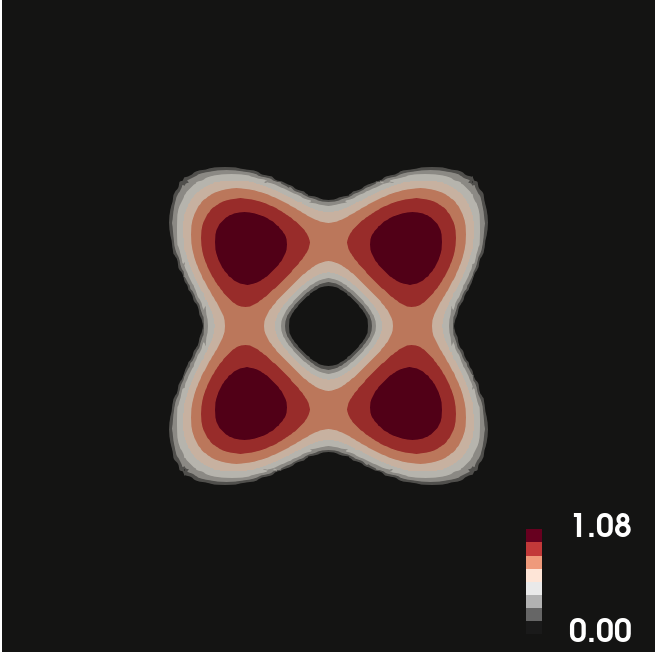}
\includegraphics[width=0.192\textwidth]{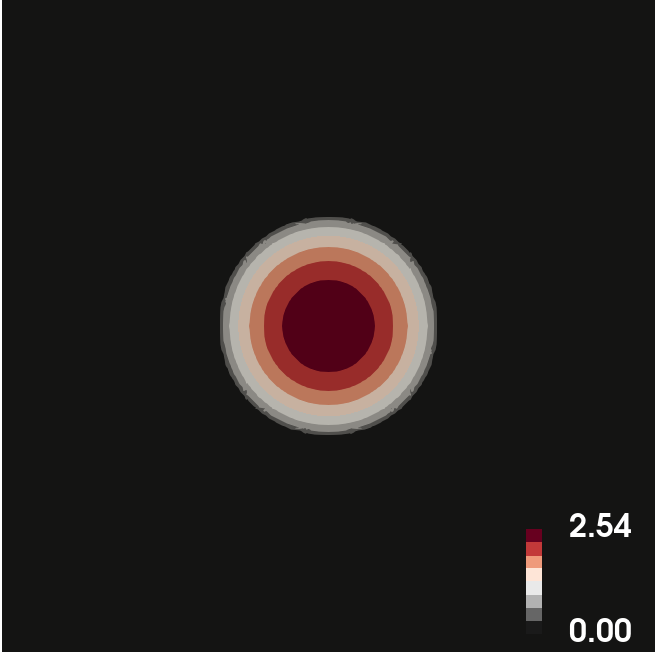}
}
\caption{Example \ref{ex2}. Snapshots of density contours at different times for different test cases.}
\label{fig:1}
\end{figure}

\subsection{Scalar reaction-diffusion equation}
\label{ex3}
We take the Case 4 energy in Table \ref{tab:2}, but 
 consider the reaction-diffusion equation \eqref{rd}.
Three choices of mobility coefficient $V_2(\rho)$ are used in this example, namely, 
\begin{align}
\begin{cases}
\text{Type 1: } V_2(\rho) = 0.1,\\[.4ex]
\text{Type 2: } V_2(\rho) = 0.1\rho,\\[.4ex]
\text{Type 3: } V_2(\rho) = 0.1\frac{\rho-1}{\log(\rho)}.
\end{cases}
\end{align}
The same discretization setup as in the previous example is used, i.e.,
using polynomial degree $k=4$ on a $32\times 32$ uniform mesh with
time step size $\Delta t = 0.05$, and final time $T=3$.

Snapshots of the density contours for each case at different times are shown in Figure \ref{fig:2}.
It is clear from the color range of these plots that reaction effects leads to mass loss, with the Type 1 reaction has the most mass loss, followed by Type 3 reaction.
\begin{figure}[tb]
\centering
\subfigure[Case 4 energy, Type 1 reaction. Left to right time: $t=0.2, 0.5, 1.5, 2.0, 3.0$]{
\label{fig:21}
\includegraphics[width=0.192\textwidth]{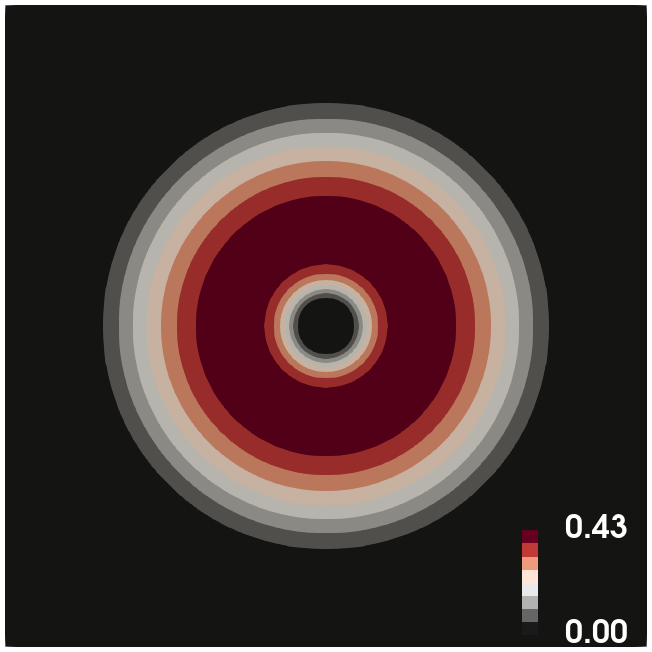}
\includegraphics[width=0.192\textwidth]{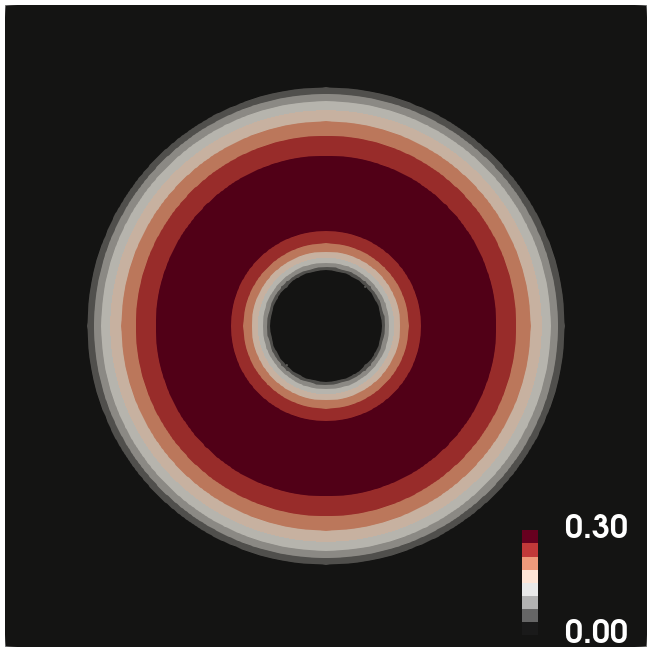}
\includegraphics[width=0.192\textwidth]{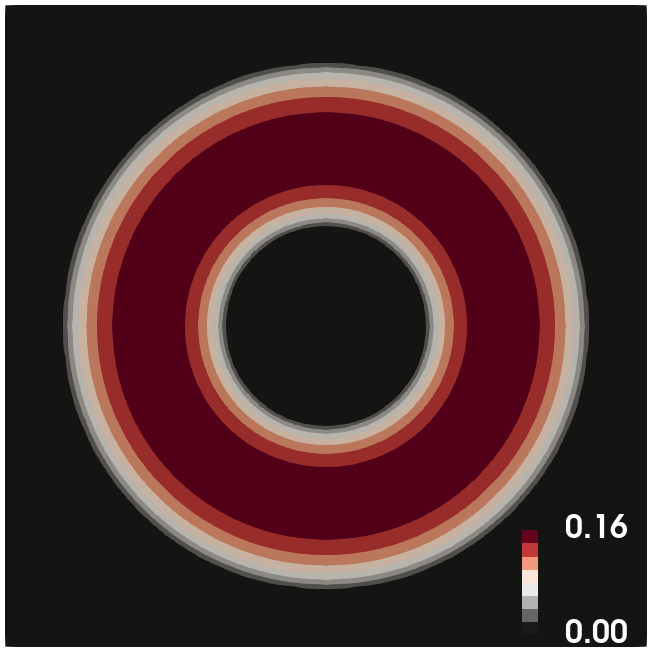}
\includegraphics[width=0.192\textwidth]{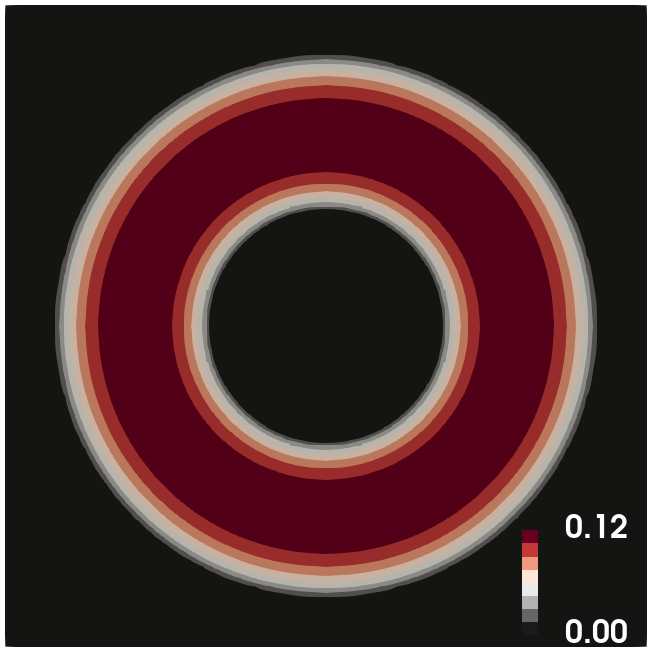}
\includegraphics[width=0.192\textwidth]{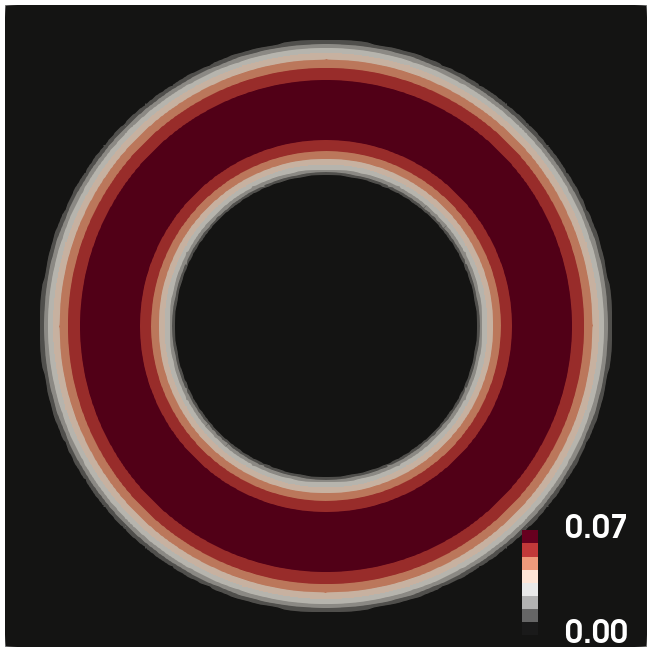}
}
\subfigure[Case 4 energy, Type 2 reaction. Left to right time: $t=0.2, 0.5, 1.5, 2.0, 3.0$]{
\label{fig:22}
\includegraphics[width=0.192\textwidth]{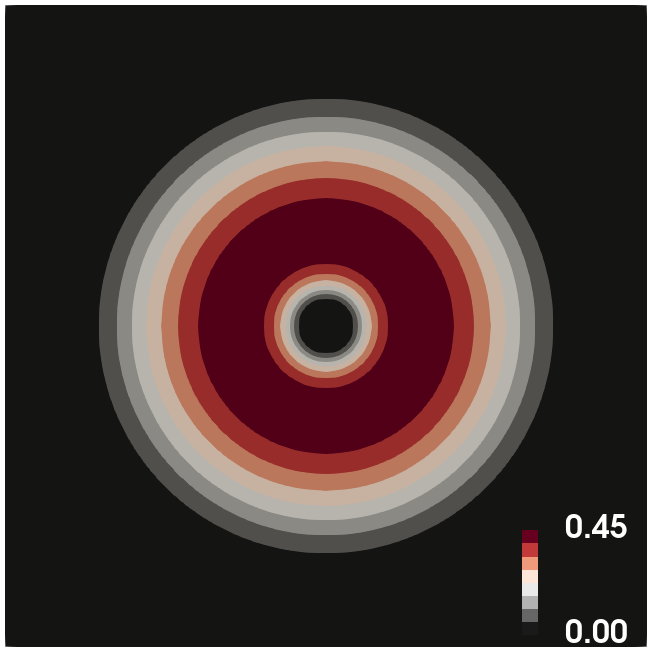}
\includegraphics[width=0.192\textwidth]{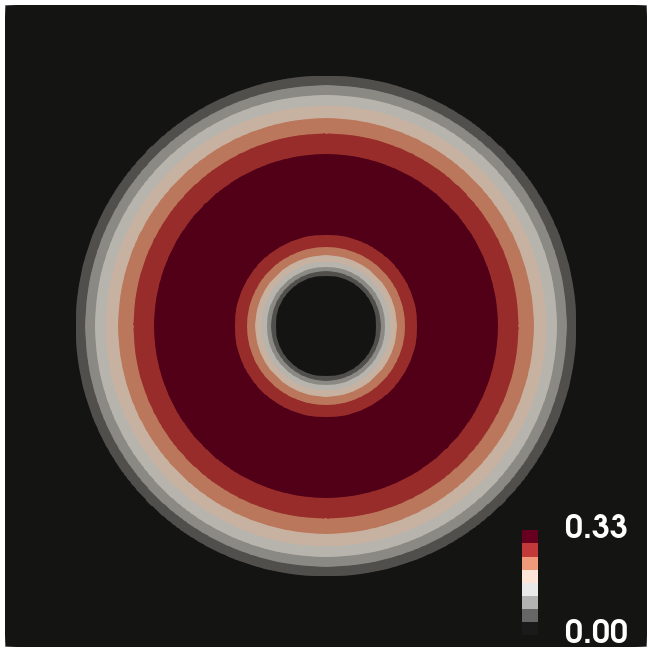}
\includegraphics[width=0.192\textwidth]{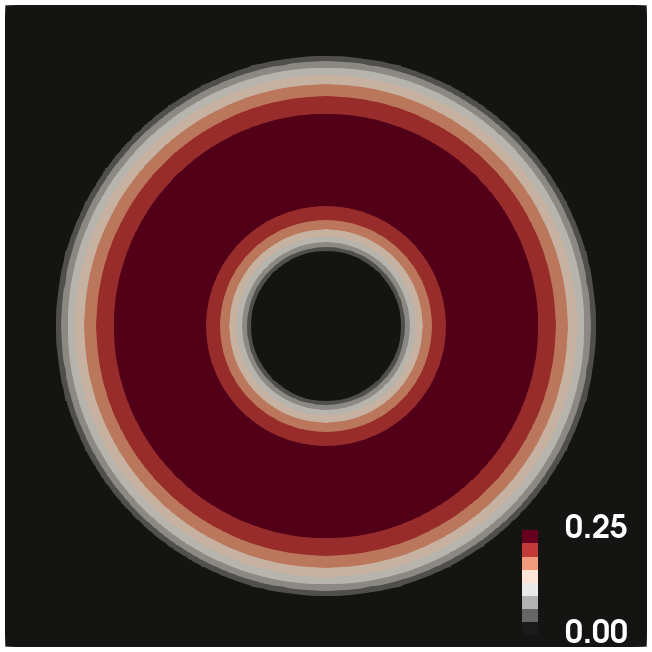}
\includegraphics[width=0.192\textwidth]{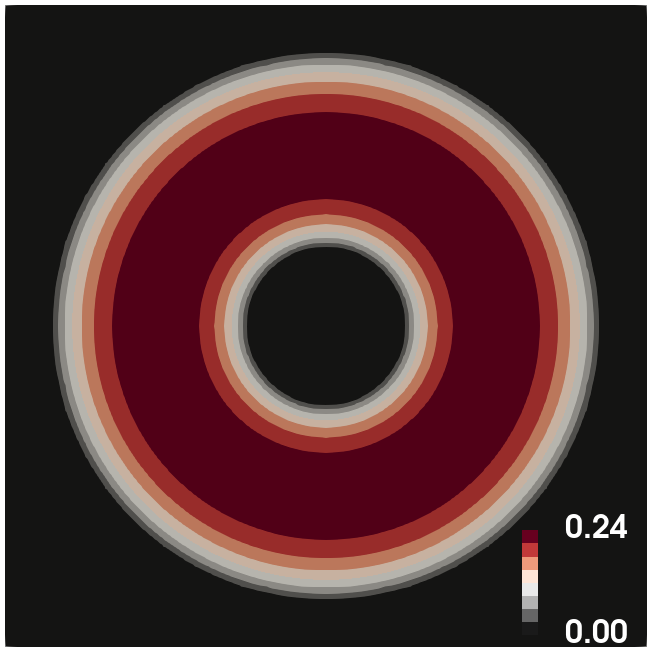}
\includegraphics[width=0.192\textwidth]{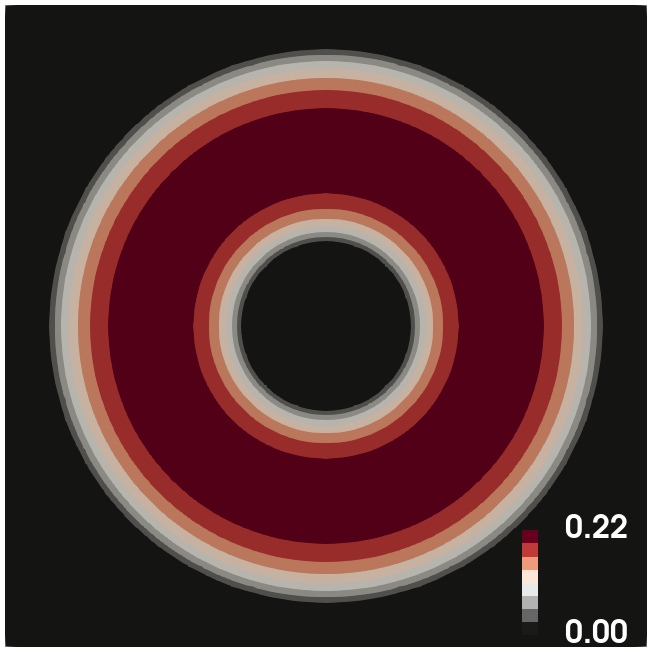}
}
\subfigure[Case 4 energy, Type 3 reaction. Left to right time: $t=0.2, 0.5, 1.5, 2.0, 3.0$]{
\label{fig:23}
\includegraphics[width=0.192\textwidth]{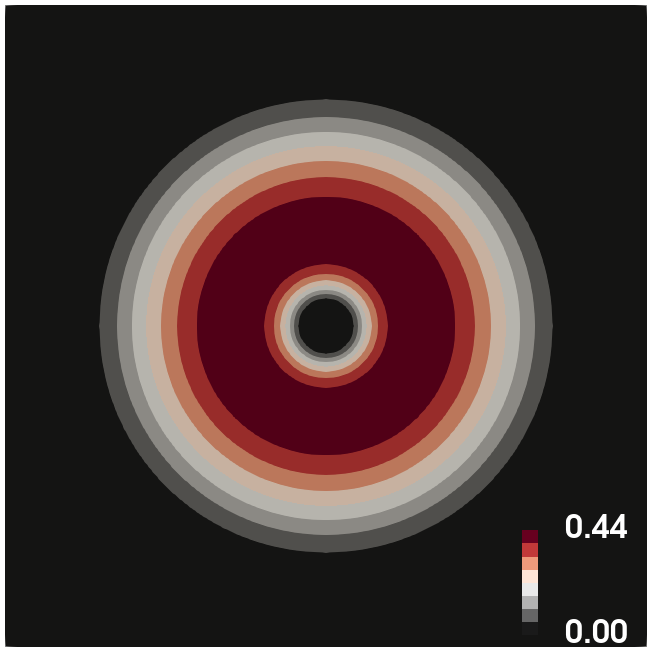}
\includegraphics[width=0.192\textwidth]{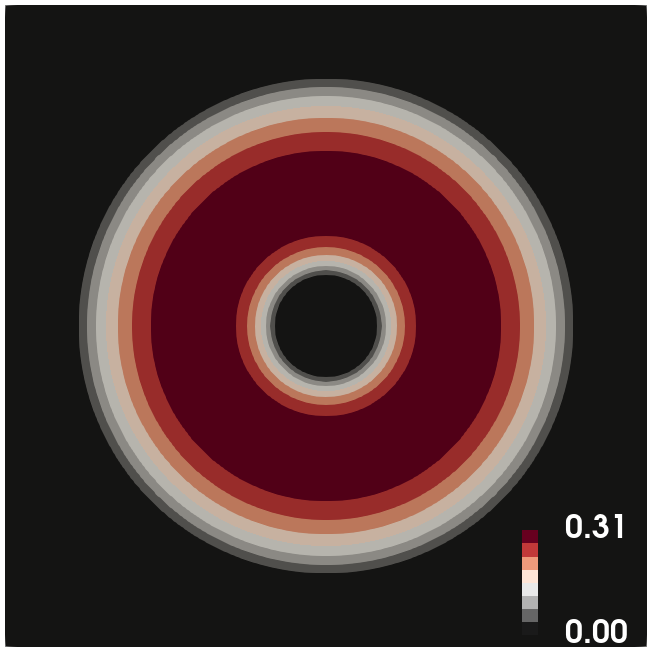}
\includegraphics[width=0.192\textwidth]{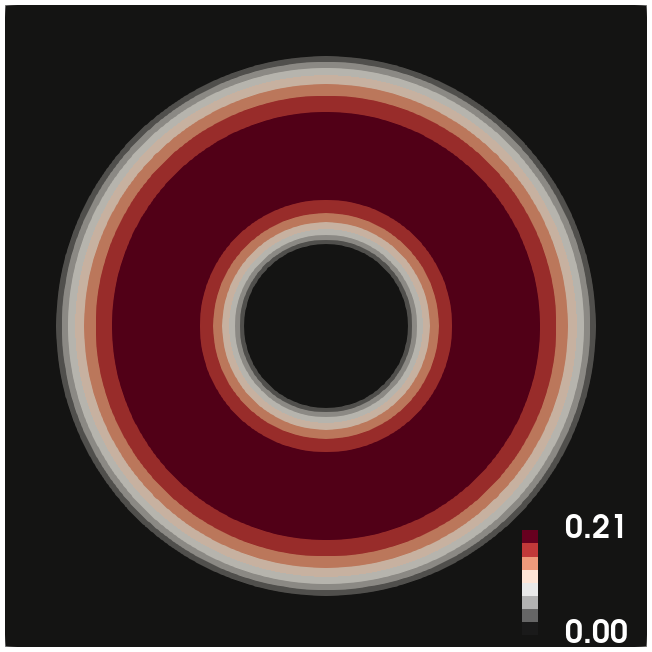}
\includegraphics[width=0.192\textwidth]{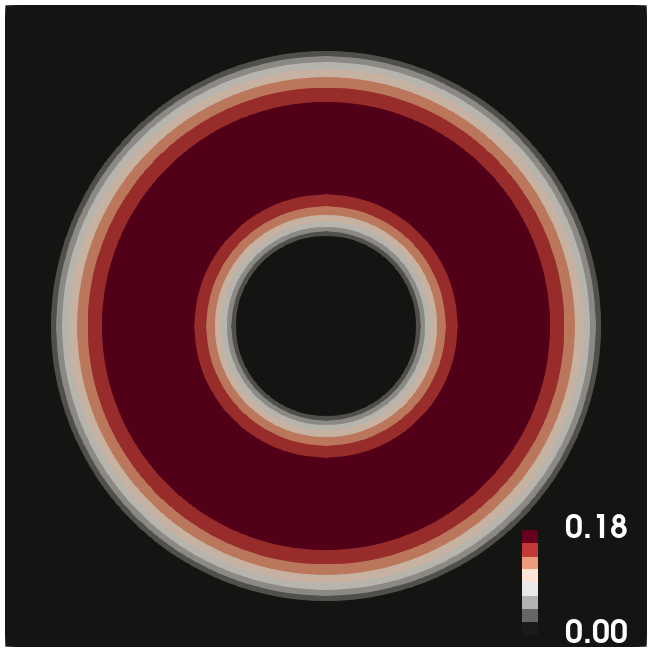}
\includegraphics[width=0.192\textwidth]{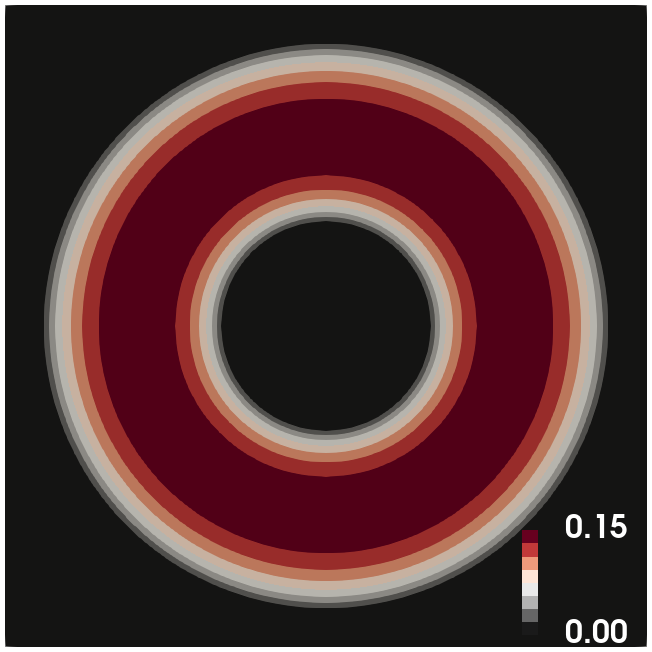}
}
\caption{Example \ref{ex3}. Snapshots of density contours at different times for different reaction mobility functions.}
\label{fig:2}
\end{figure}

\subsection{Fisher-KPP equation}
\label{ex4}
Our next example deals with the Fisher-KPP equation \eqref{kpp}. 
Here we slightly modify the PDE \eqref{kpp} to allow for anisotropic diffusion:
\[
\partial_t\rho-\lambda_1\partial_{x_0x_0}\rho-\lambda_2\partial_{x_1x_1}\rho=\mu \rho(1-\rho).
\]
We use a similar setup as in \cite[Secion 3.1]{kpp2d}, where 
the diffusion parameters are taken to be  
$\lambda_1=0.1, \lambda_2=0.01$, and $\mu>0$ is the
reaction coefficient to be specified.
Initial condition is a flat top Gaussian:
\begin{align*}
\rho_0(x_0,x_1)
=\begin{cases}
1, &\text{if} \;\;x_0^2+4x_1^2\le 0.25\\[.4ex]
\exp(-10(x_0^2+4x_1^2-0.25)), &\text{otherwise} 
\end{cases}
\end{align*}
The computational domain is a rectangle $\Omega=[-2,2]\times[-1,1]$, which
is discretized with a $32\times 16$ square mesh.
We use polynomial degree $k=4$ for the  scheme \eqref{aug-mfp-h}, in which
the functional $F_h$ in \eqref{Fh} is adjusted as follows to allow for anisotropic diffusion:
\[
F_h(\bmu_h):=
\left(
\frac{|m_{h}^0|^2}{2V_{1,0}(\rho_h)}
+
\frac{|m_{h}^1|^2}{2V_{1,1}(\rho_h)}
+\frac{|s_{h}|^2}{2V_2(\rho_h)}, 1\right)_h
+\Delta t\, \mathcal{E}_h(u_{0,h}),
\]
where 
$V_{1,0}(\rho):=\lambda_1\rho$,
$V_{1,1}(\rho):=\lambda_2\rho$,
$V_2(\rho):=\mu \frac{\rho(\rho-1)}{\log(\rho)}$,
and the energy satisfies 
\[
\mathcal{E}_h(\rho)=
\left(\rho (\log(\rho)-1),1\right)_h.
\]
We take time step size  $\Delta t=0.1$, and the final time is $T=4$.

Snapshots of the density contours for $\mu=0.1$ (weak reaction) 
$\mu=0.5$ (medium reaction),
and $\mu=1.0$ (strong reaction) at different times are shown in Figure \ref{fig:3}.
We further plot the evolution of energy 
$\mathcal{E}_h(\rho_h)$ and total mass 
$\int_{\Omega}\rho_{h}\, dx$ over time 
for the three cases in Figure~\ref{fig:4}.
It is clear that the energy is monotonically decreasing for all three cases and the total mass is monotonically increasing, where a faster decay of energy  is observed when the reaction coefficient $\mu$ is larger.

\begin{figure}[tb]
\centering
\subfigure[Reaction coefficient $\mu=0.1$. Left to right time: $t=1, 2, 3, 4$]{
\label{fig:31}
\includegraphics[width=0.24\textwidth]{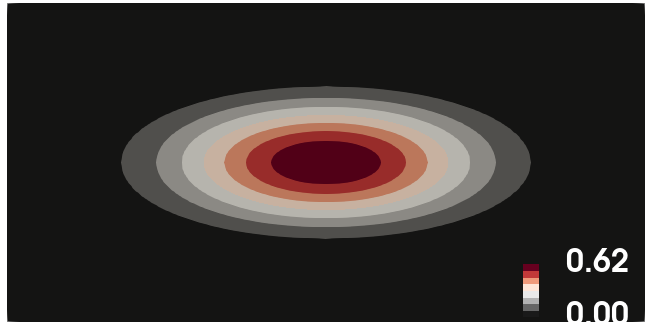}
\includegraphics[width=0.24\textwidth]{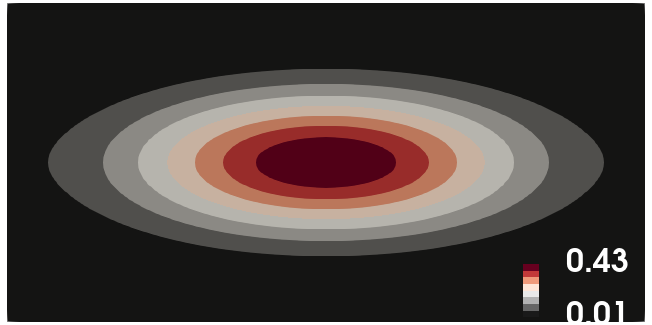}
\includegraphics[width=0.24\textwidth]{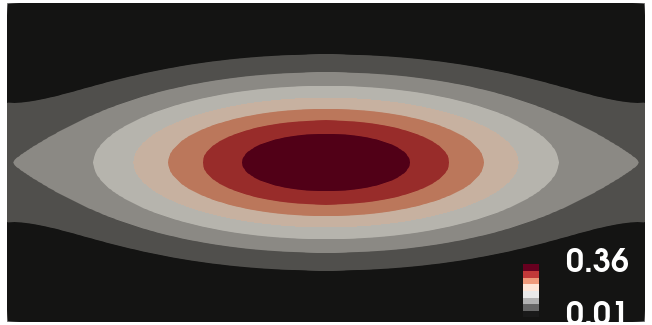}
\includegraphics[width=0.24\textwidth]{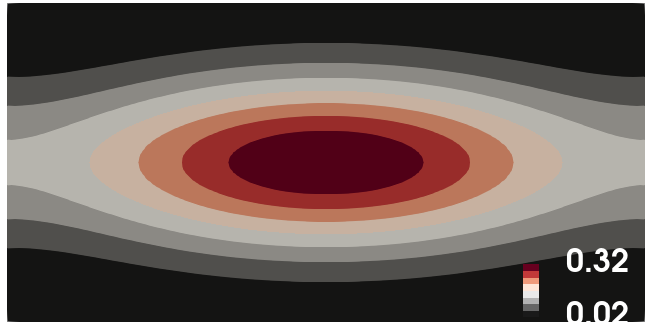}
}
\subfigure[Reaction coefficient $\mu=0.5$. Left to right time: $t=1, 2, 3, 4$]{
\label{fig:32}
\includegraphics[width=0.24\textwidth]{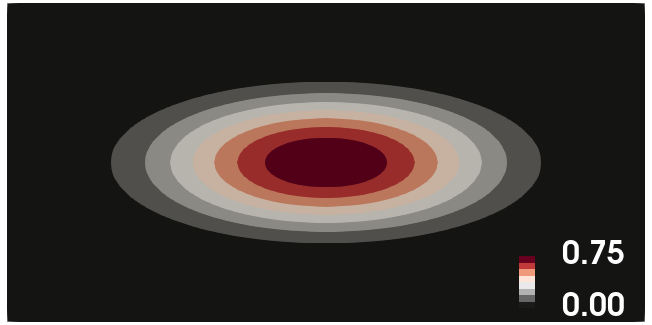}
\includegraphics[width=0.24\textwidth]{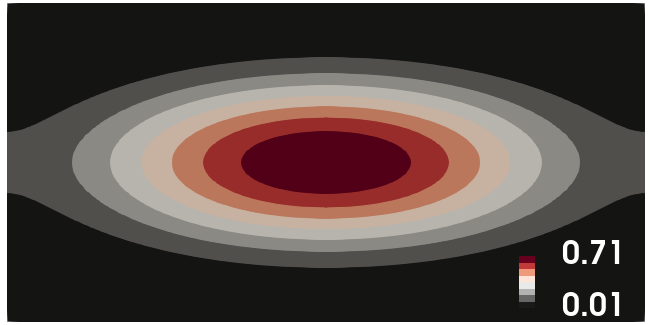}
\includegraphics[width=0.24\textwidth]{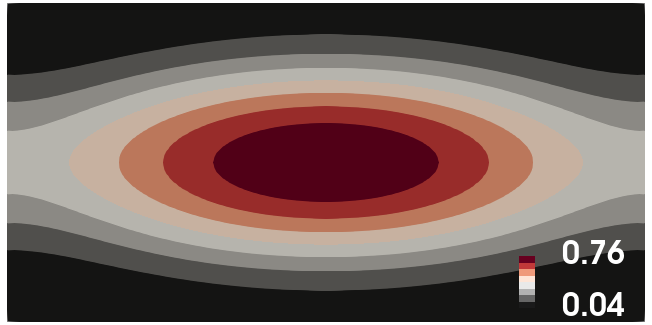}
\includegraphics[width=0.24\textwidth]{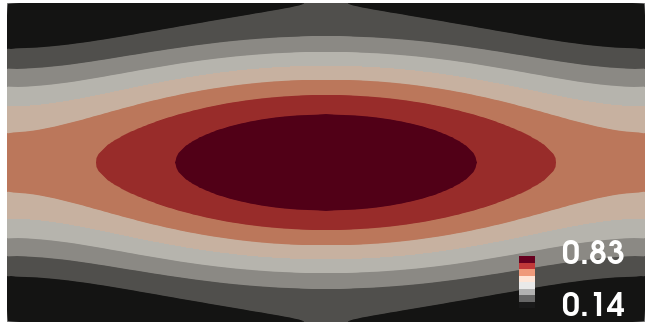}
}
\subfigure[Reaction coefficient $\mu=1.0$. Left to right time: $t=1, 2, 3, 4$]{
\label{fig:33}
\includegraphics[width=0.24\textwidth]{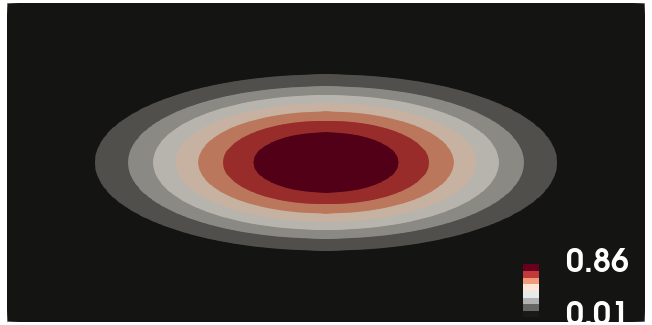}
\includegraphics[width=0.24\textwidth]{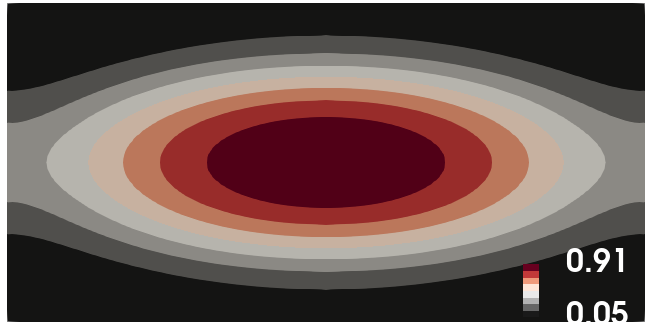}
\includegraphics[width=0.24\textwidth]{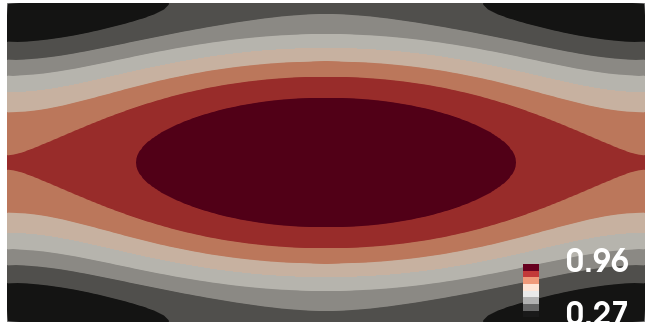}
\includegraphics[width=0.24\textwidth]{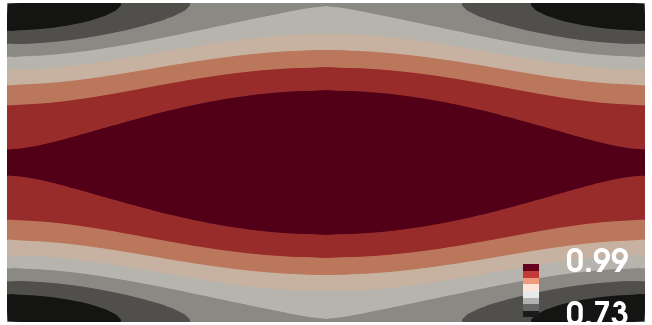}
}
\caption{Example \ref{ex3}. Snapshots of density contours at different times for different reaction coefficients.}
\label{fig:3}
\end{figure}

\begin{figure}[tb]
\centering
\includegraphics[width=0.45\textwidth]{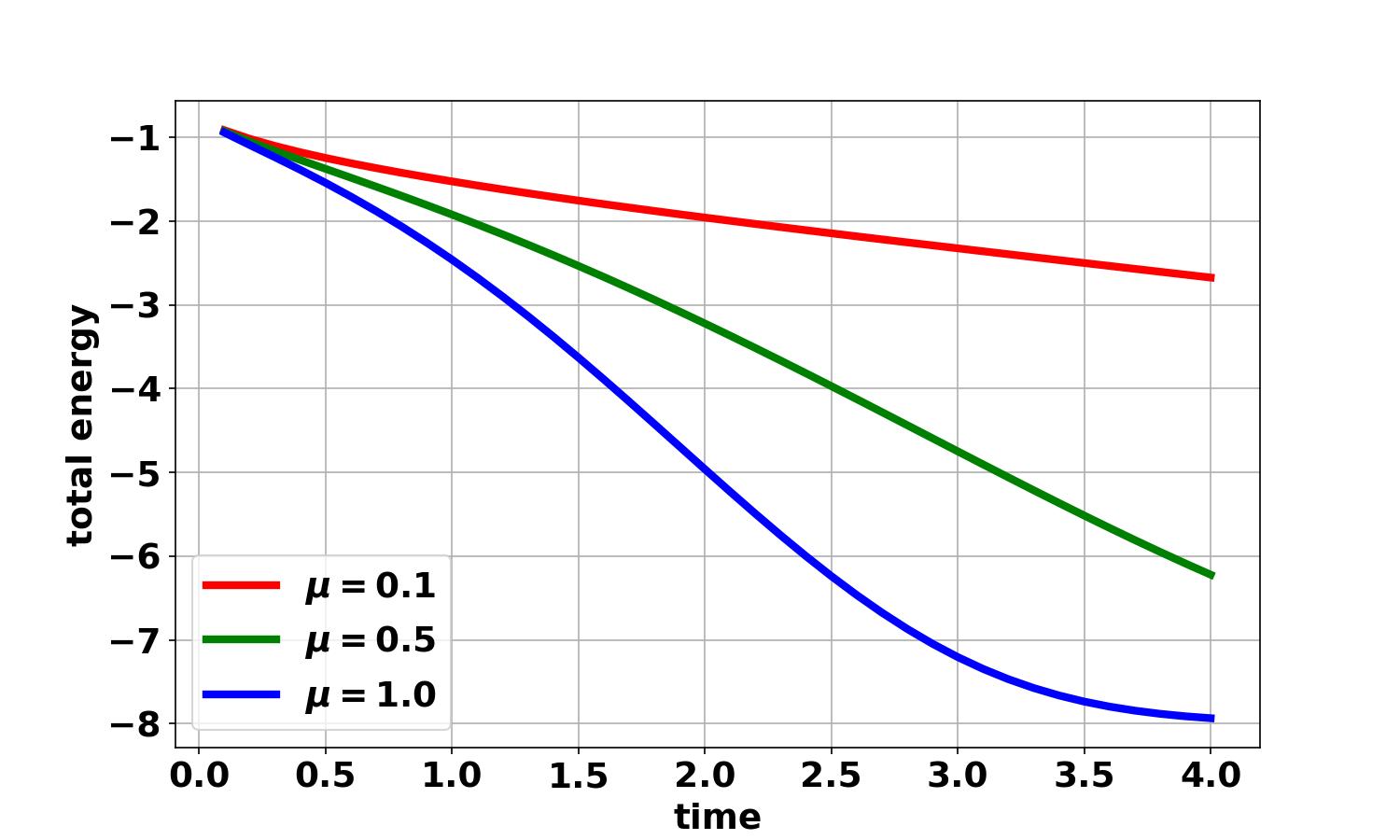}
\includegraphics[width=0.45\textwidth]{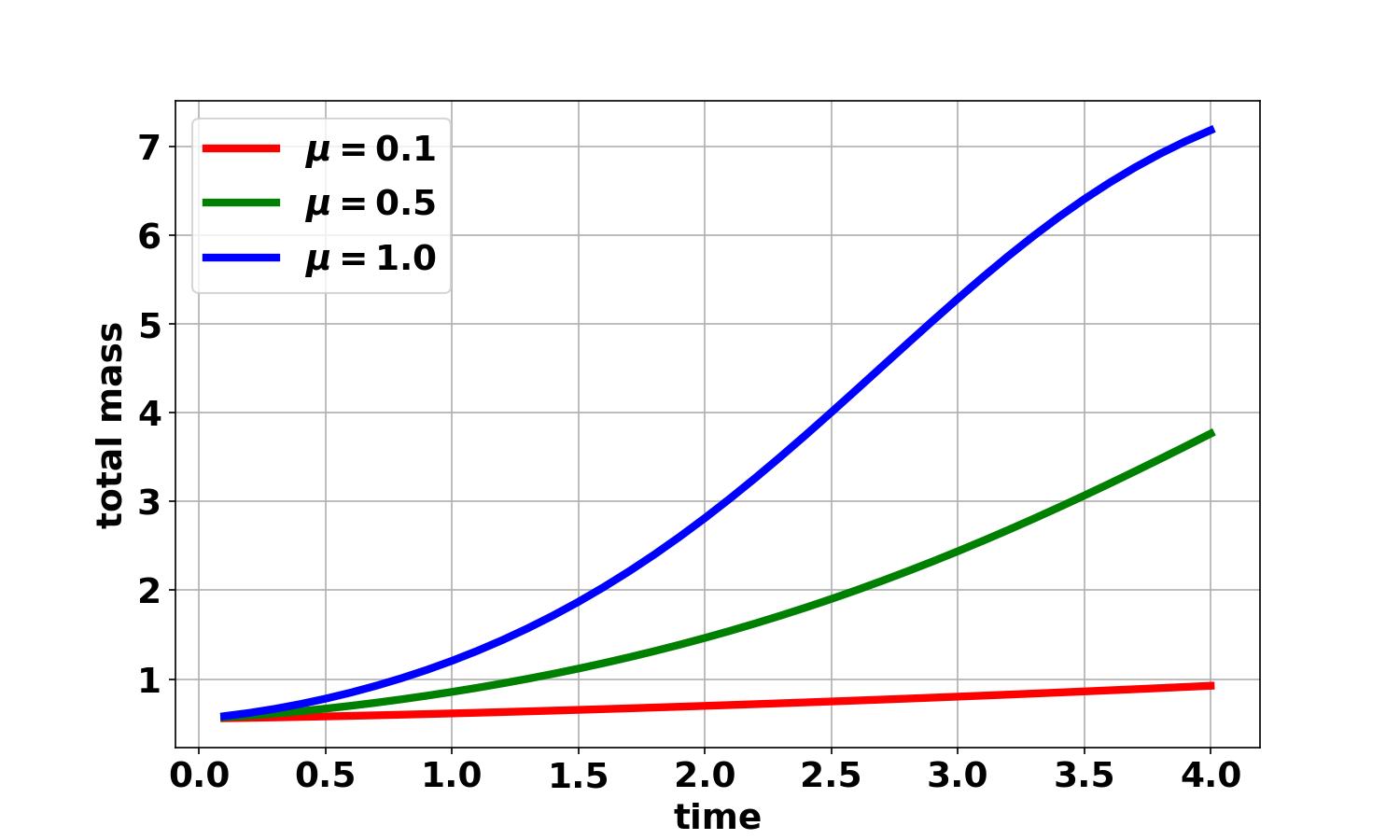}
\caption{Example \ref{ex3}. Evolution of total 
energy (left) and total mass (right) over time.}
\label{fig:4}
\end{figure}

\subsection{Two-component reversible reaction-diffusion system with detailed balance}
\label{ex5}
We consider the two-species model discussed in Section \ref{sec:25}.
In particular, we consider the system \eqref{2comp} with parameters
$k_+=1$ and $k_-=0.1$, $\gamma_1= 0.2$, $\gamma_2=0.1$, and 
$V_{1,1}(\rho)=\gamma_1\rho^m$ and $V_{1,2}(\rho)=\gamma_2\rho$
with four choices of $m\in\{1,2,3,4\}$.
Here porous medium type diffusion is used for the first species with density $\rho_1$ and 
linear diffusion is used for the second species with 
density $\rho_2$. Similar model was used in \cite{LiuWangWang21a,LiuWangWang22}.
The problems are solved on the domain $\Omega=[-1,1]\times[-1,1]$ with the following initial data 
\begin{align*}
\rho_1(x, 0)=&\; \frac{1}{2}\left(
1-\mathrm{tanh}(10(\sqrt{x_0^2+x_1^2}-0.2))
\right),\\
\rho_2(x, 0)= &\;\frac{1}{2}\left(
1+\mathrm{tanh}(10(\sqrt{x_0^2+x_1^2}-0.2)).
\right).
\end{align*}
Final time is taken to be $T=2$.

We use the scheme \eqref{aug-mfp-hs} 
with polynomial degree $k=4$ on a $16\times 16$ mesh with time step size $\Delta t=0.05$.
We apply Algorithm \ref{alg:3} to solve the resulting saddle point problem.
Snapshots of the density contours at different times are shown in Figure \ref{fig:5a} for the first component, and in Figure \ref{fig:5b} for the second component.
It is clear that increasing the power $m$ leads to a slower diffusion for the first species. 
\begin{figure}[tb]
\centering
\subfigure[$V_{1,1}(\rho)=\gamma_1\rho$. Left to right time: $t=0,0.5,1,1.5,2$]{
\label{fig:51}
\includegraphics[width=0.192\textwidth]{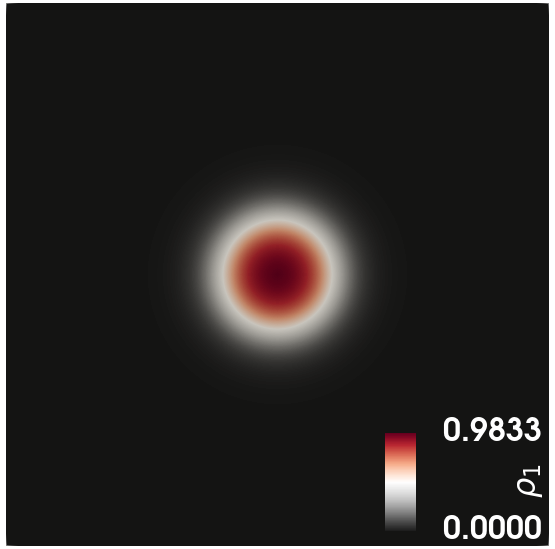}
\includegraphics[width=0.192\textwidth]{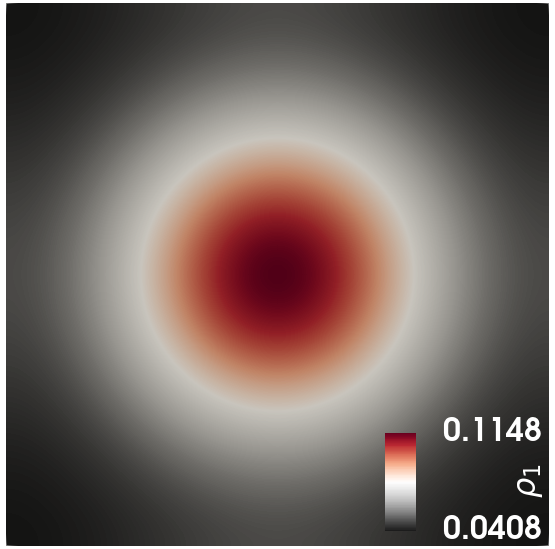}
\includegraphics[width=0.192\textwidth]{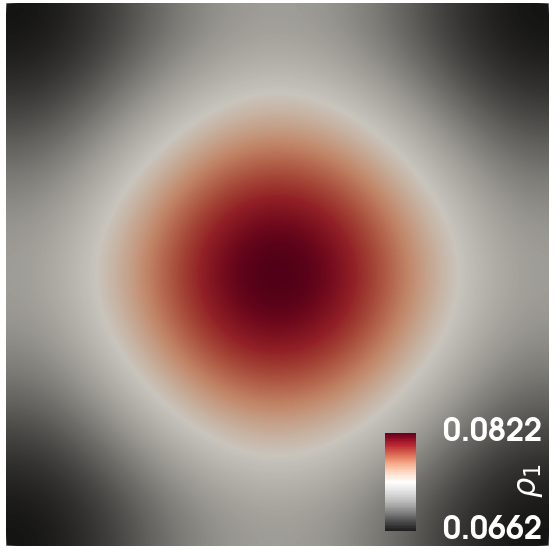}
\includegraphics[width=0.192\textwidth]{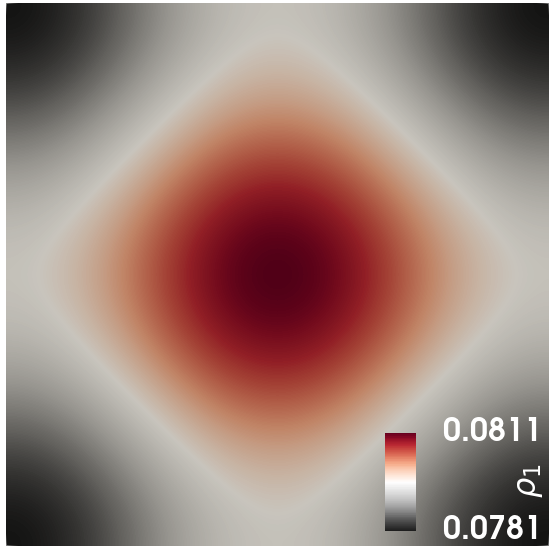}
\includegraphics[width=0.192\textwidth]{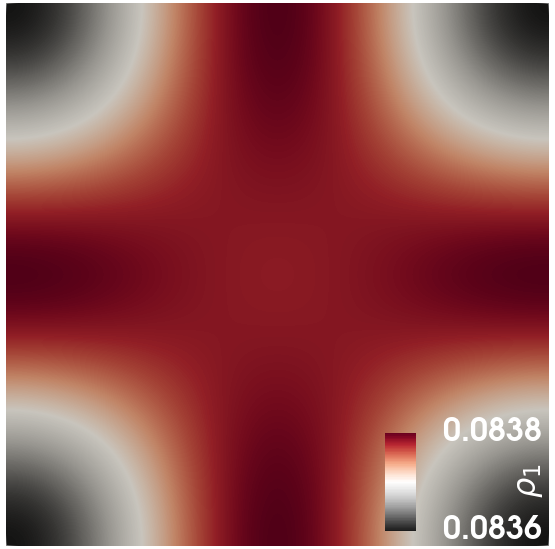}
}
\subfigure[$V_{1,1}(\rho)=\gamma_1\rho^2$. Left to right time: $t=0,0.5,1,1.5,2$]{
\label{fig:52}
\includegraphics[width=0.192\textwidth]{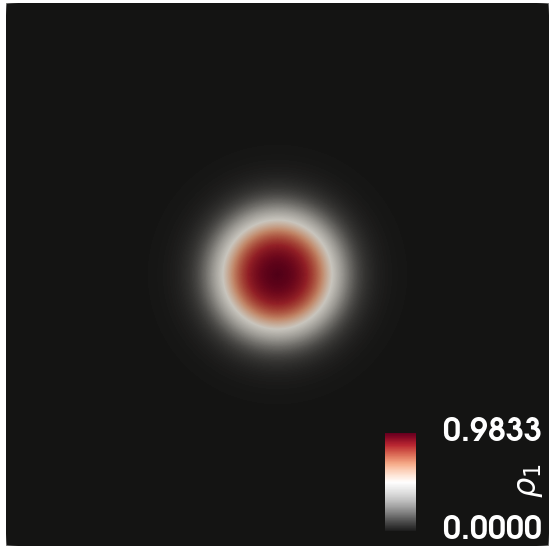}
\includegraphics[width=0.192\textwidth]{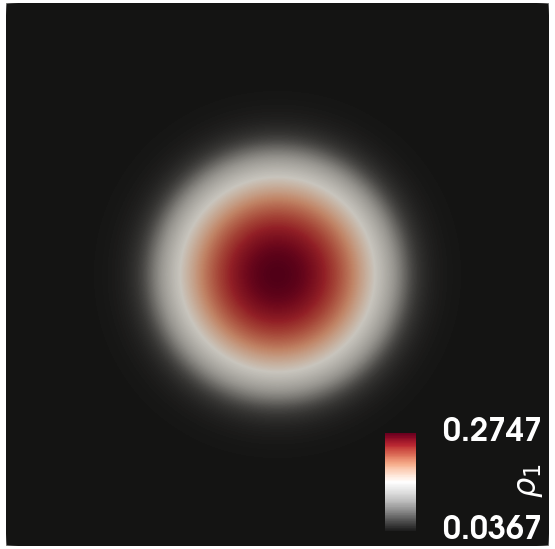}
\includegraphics[width=0.192\textwidth]{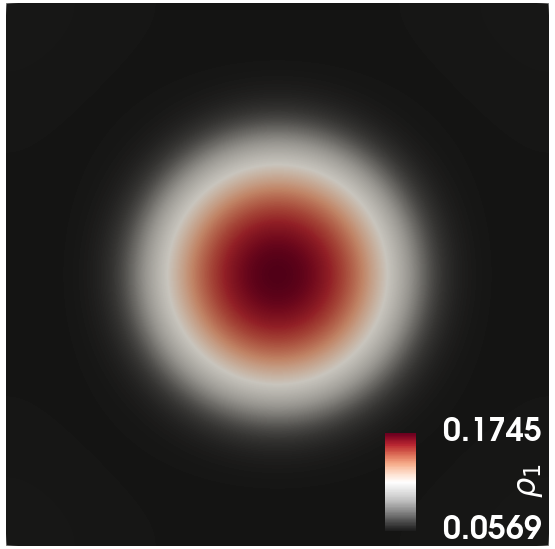}
\includegraphics[width=0.192\textwidth]{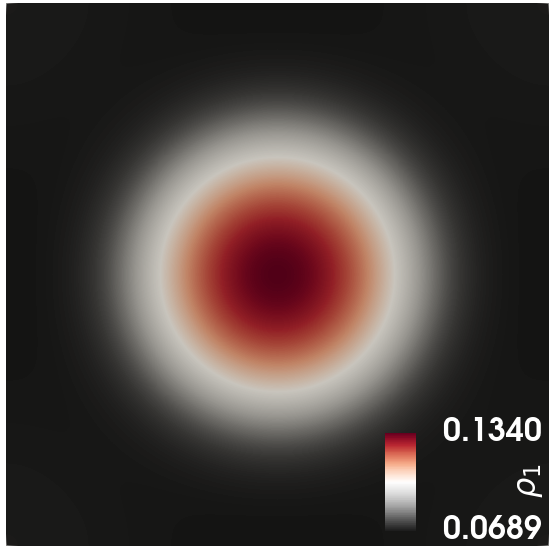}
\includegraphics[width=0.192\textwidth]{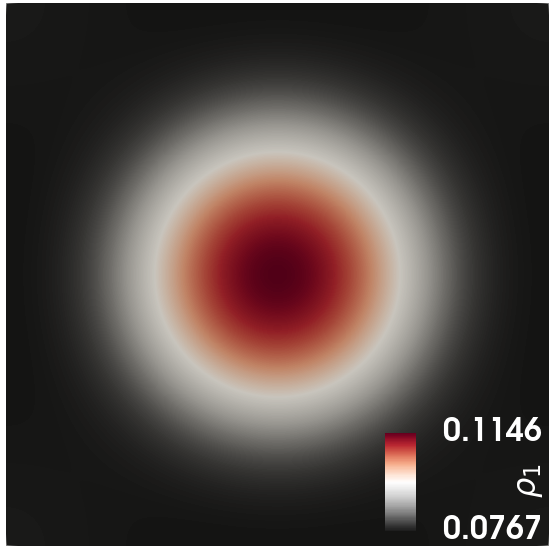}
}
\subfigure[$V_{1,1}(\rho)=\gamma_1\rho^3$. Left to right time: $t=0,0.5,1,1.5,2$]{
\label{fig:53}
\includegraphics[width=0.192\textwidth]{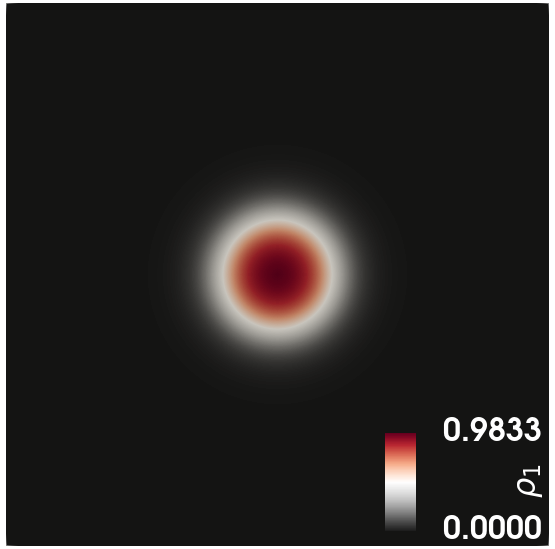}
\includegraphics[width=0.192\textwidth]{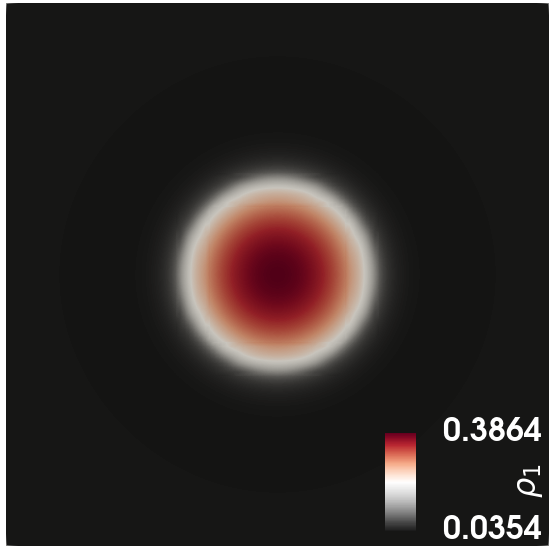}
\includegraphics[width=0.192\textwidth]{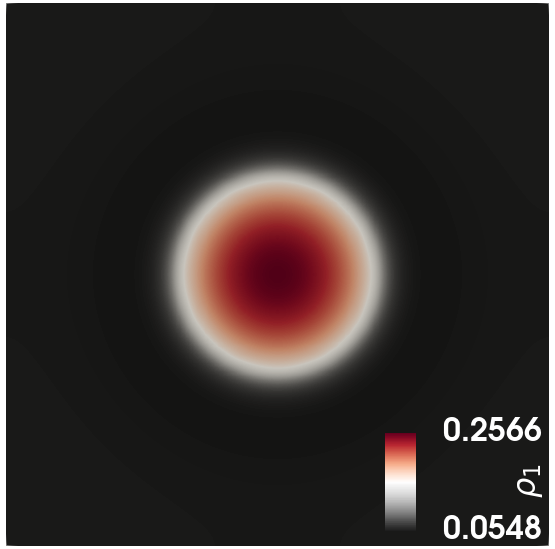}
\includegraphics[width=0.192\textwidth]{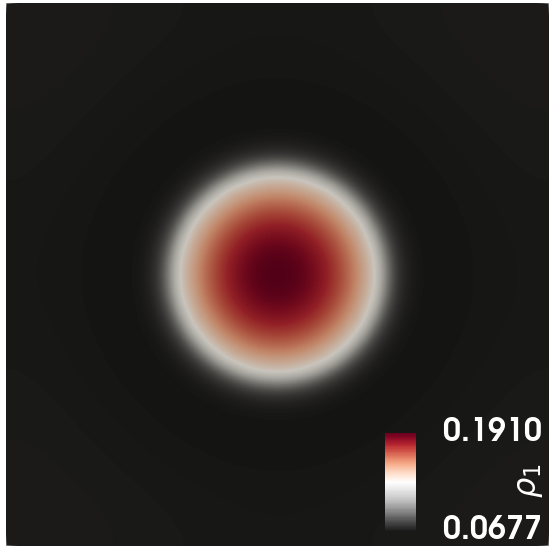}
\includegraphics[width=0.192\textwidth]{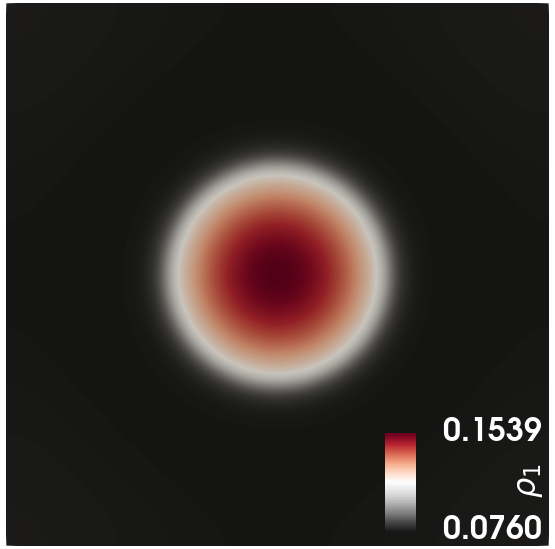}
}
\subfigure[$V_{1,1}(\rho)=\gamma_1\rho^4$. Left to right time: $t=0,0.5,1,1.5,2$]{
\label{fig:54}
\includegraphics[width=0.192\textwidth]{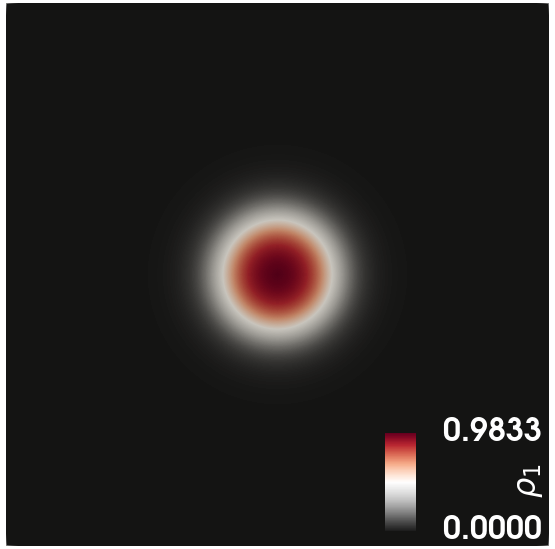}
\includegraphics[width=0.192\textwidth]{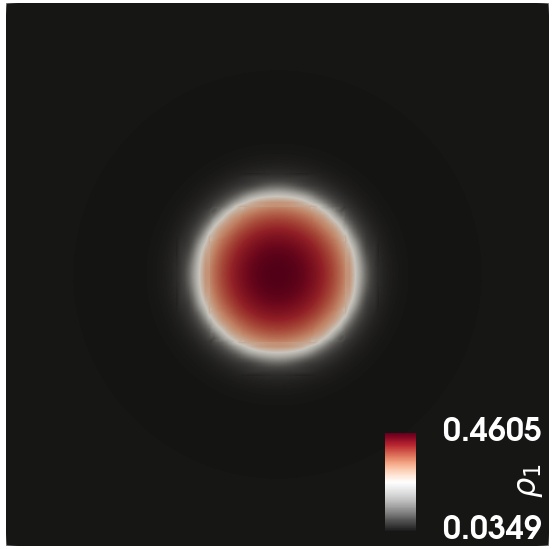}
\includegraphics[width=0.192\textwidth]{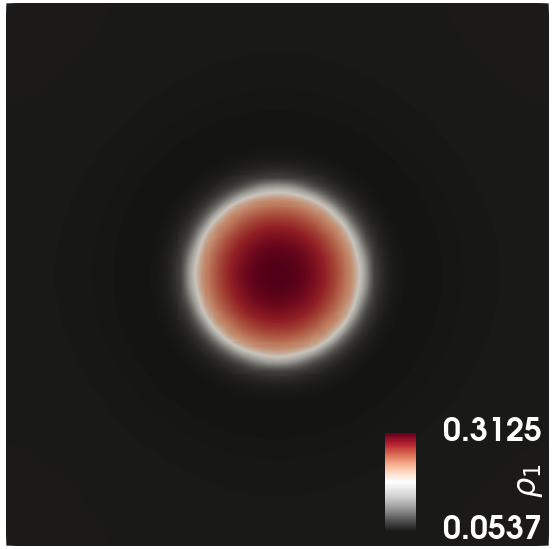}
\includegraphics[width=0.192\textwidth]{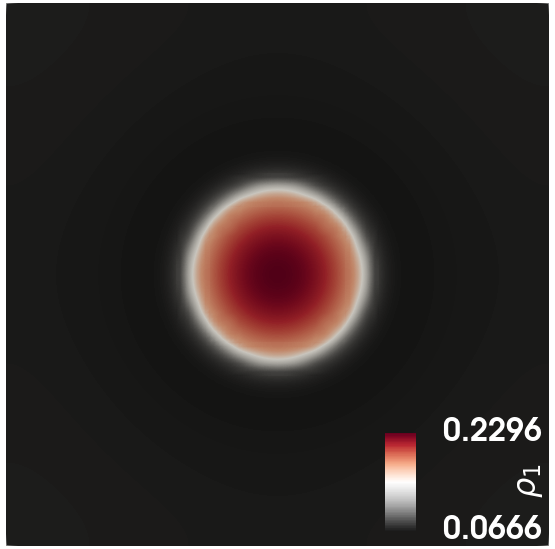}
\includegraphics[width=0.192\textwidth]{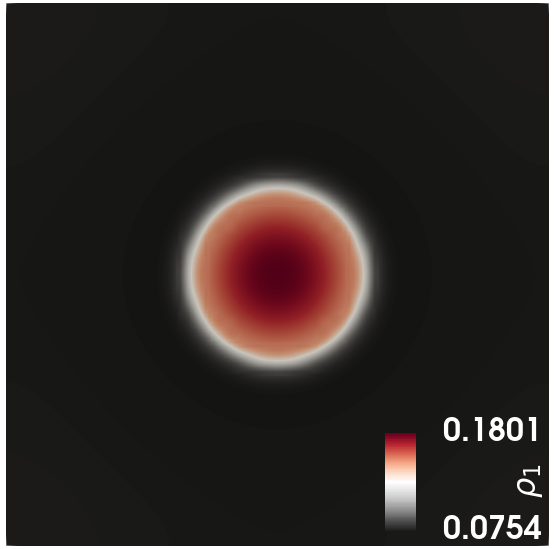}
}
\caption{Example \ref{ex5}. Snapshots of first-component density contours at different times for different $V_{1,1}(\rho)$.}
\label{fig:5a}
\end{figure}

\begin{figure}[tb]
\centering
\subfigure[$V_{1,1}(\rho)=\rho$. Left to right time: $t=0,0.5,1,1.5,2$]{
\label{fig:51b}
\includegraphics[width=0.192\textwidth]{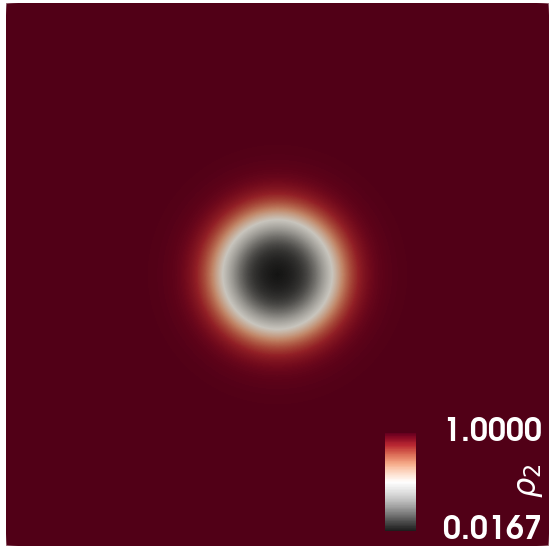}
\includegraphics[width=0.192\textwidth]{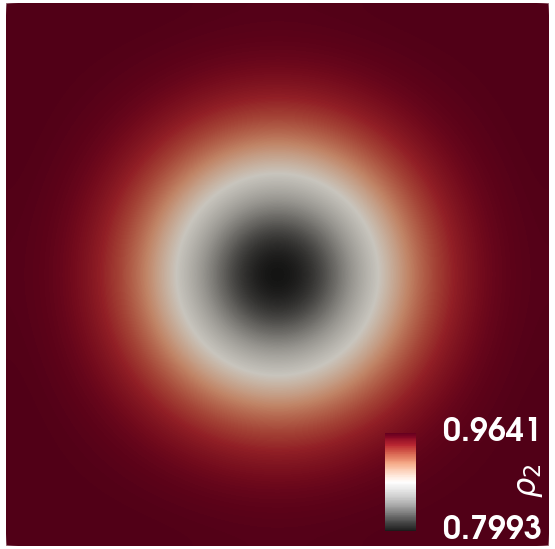}
\includegraphics[width=0.192\textwidth]{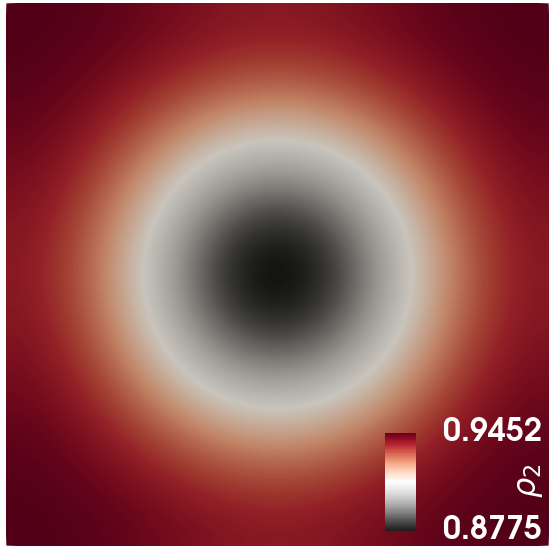}
\includegraphics[width=0.192\textwidth]{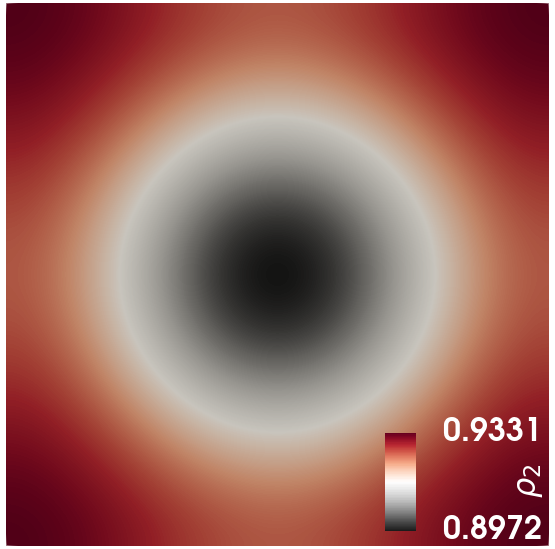}
\includegraphics[width=0.192\textwidth]{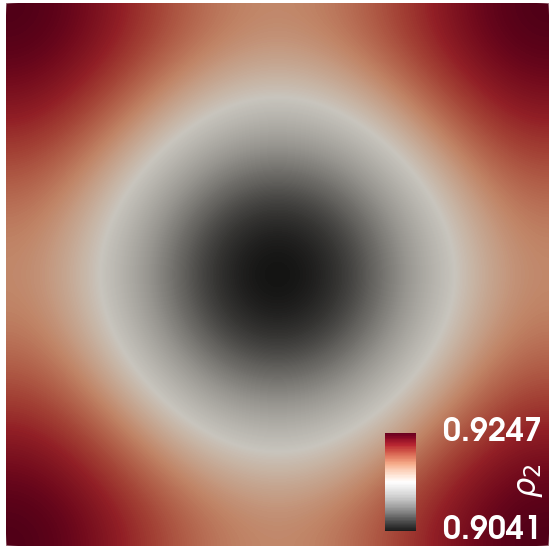}
}
\subfigure[$V_{1,1}(\rho)=\rho^2$. Left to right time: $t=0,0.5,1,1.5,2$]{
\label{fig:52b}
\includegraphics[width=0.192\textwidth]{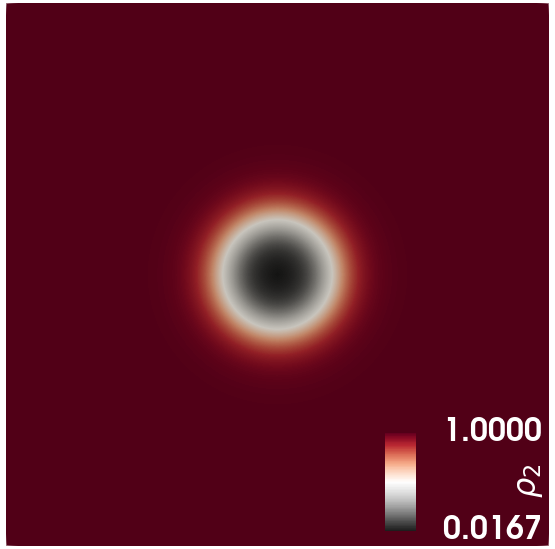}
\includegraphics[width=0.192\textwidth]{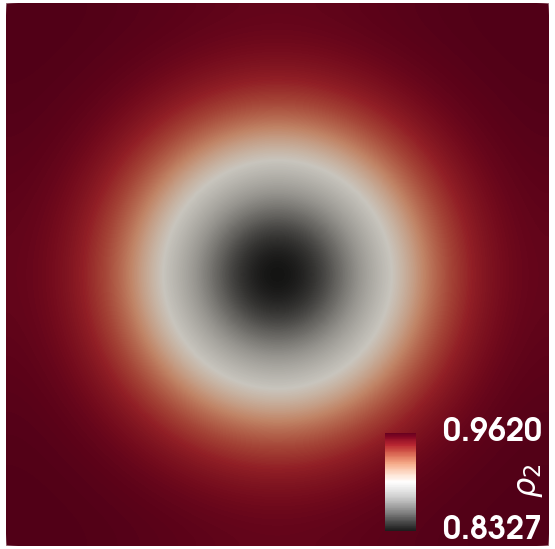}
\includegraphics[width=0.192\textwidth]{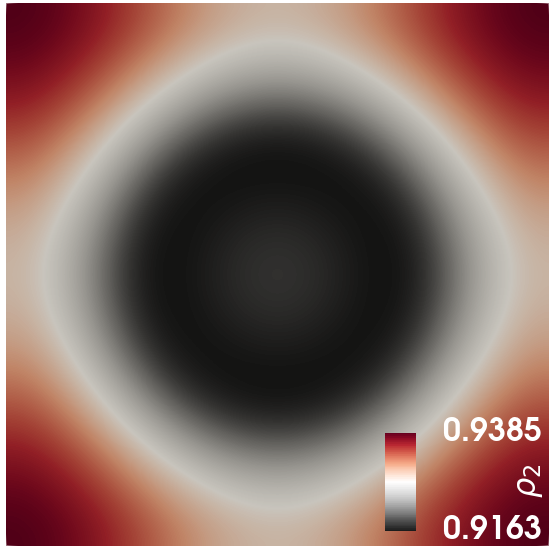}
\includegraphics[width=0.192\textwidth]{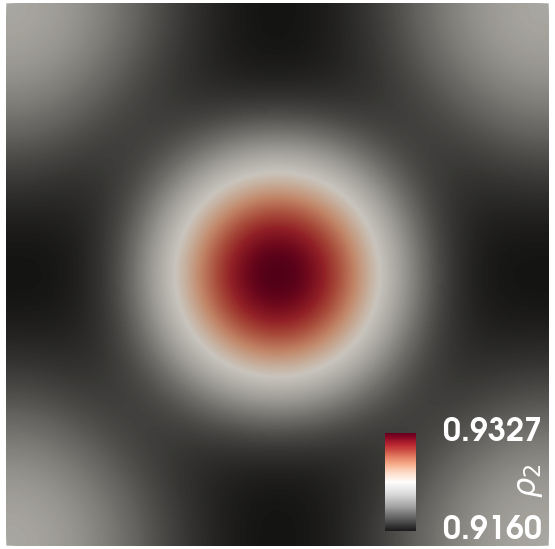}
\includegraphics[width=0.192\textwidth]{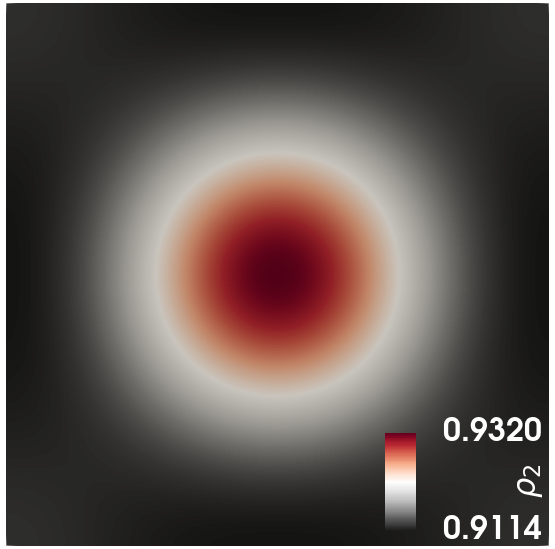}
}
\subfigure[$V_{1,1}(\rho)=\rho^3$. Left to right time: $t=0,0.5,1,1.5,2$]{
\label{fig:53b}
\includegraphics[width=0.192\textwidth]{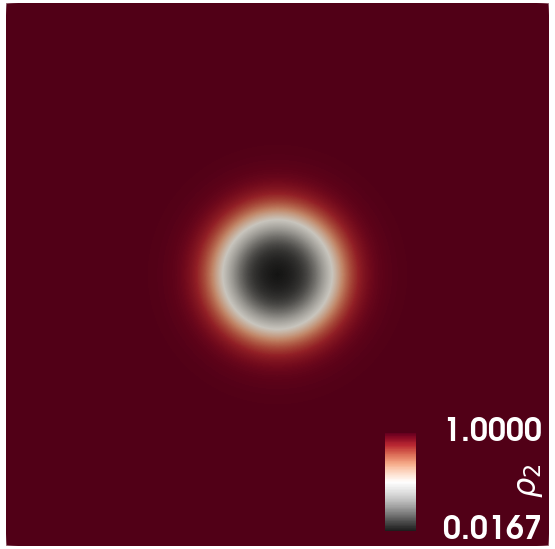}
\includegraphics[width=0.192\textwidth]{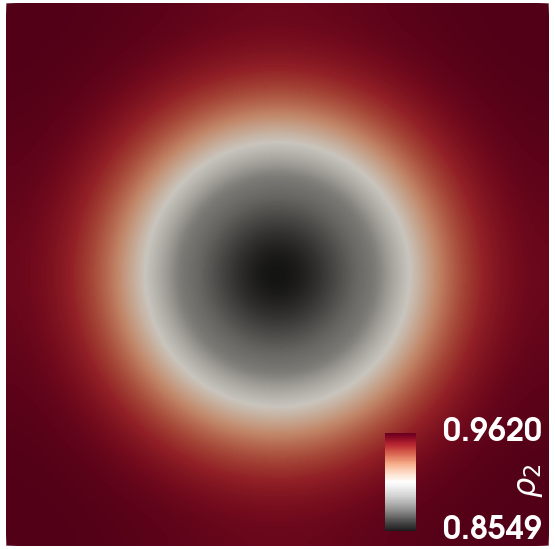}
\includegraphics[width=0.192\textwidth]{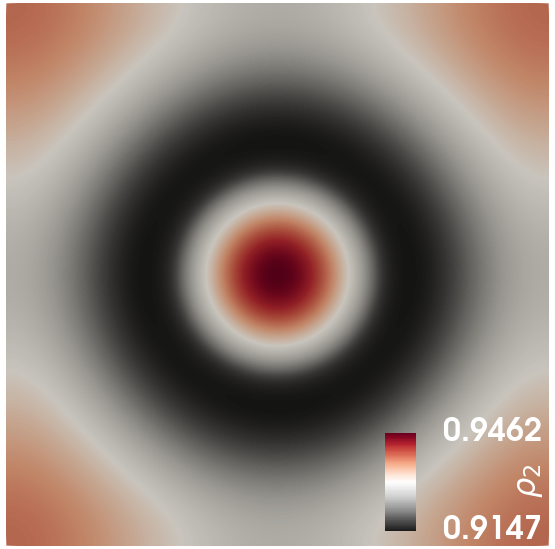}
\includegraphics[width=0.192\textwidth]{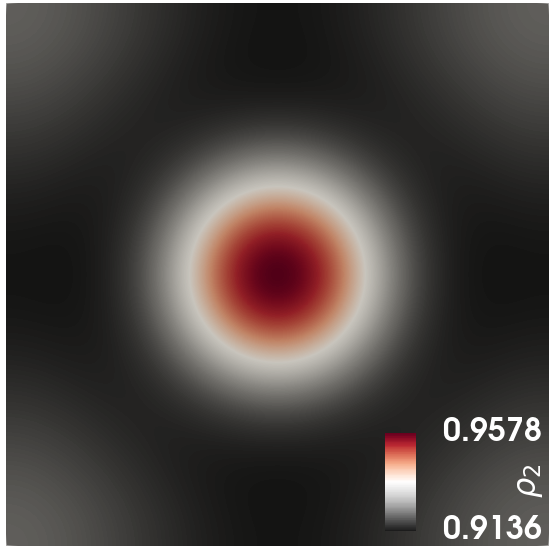}
\includegraphics[width=0.192\textwidth]{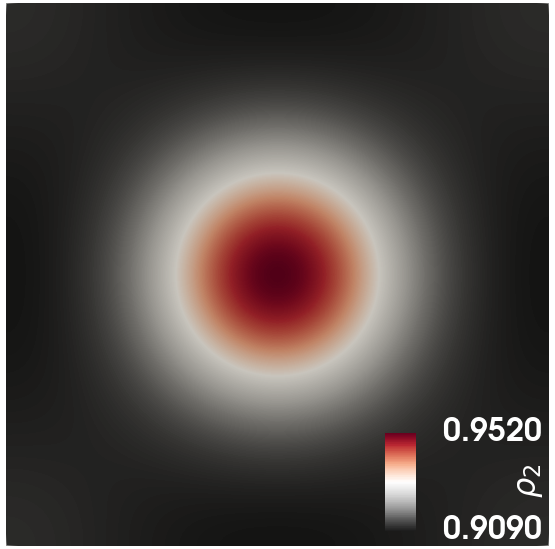}
}
\subfigure[$V_{1,1}(\rho)=\rho^4$. Left to right time: $t=0,0.5,1,1.5,2$]{
\label{fig:54b}
\includegraphics[width=0.192\textwidth]{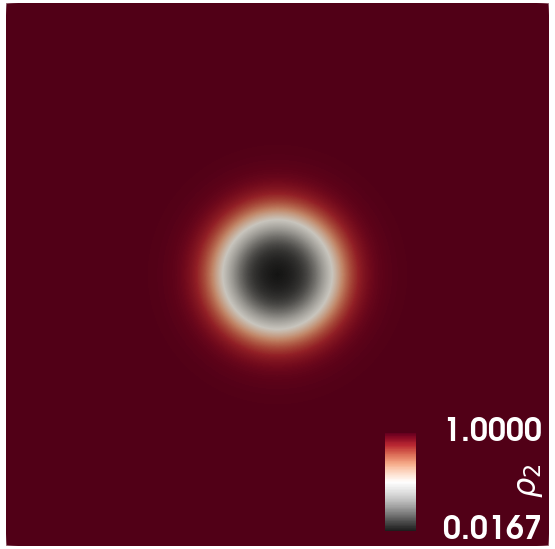}
\includegraphics[width=0.192\textwidth]{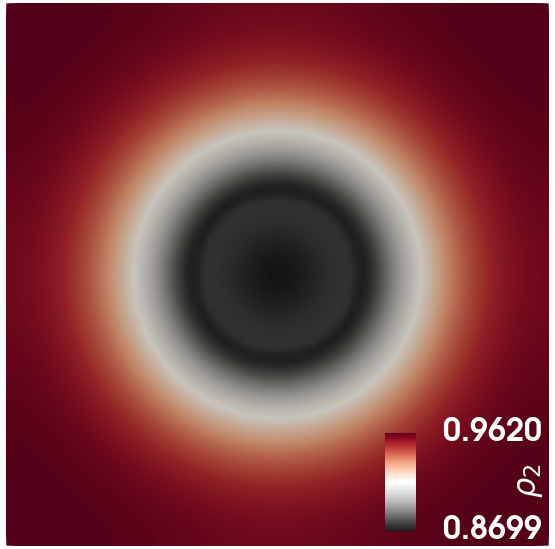}
\includegraphics[width=0.192\textwidth]{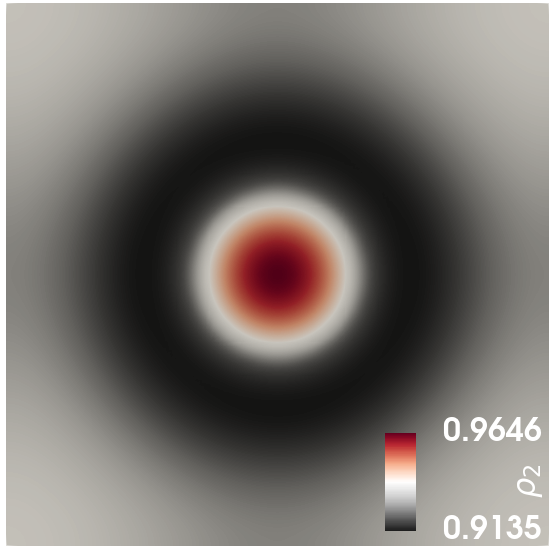}
\includegraphics[width=0.192\textwidth]{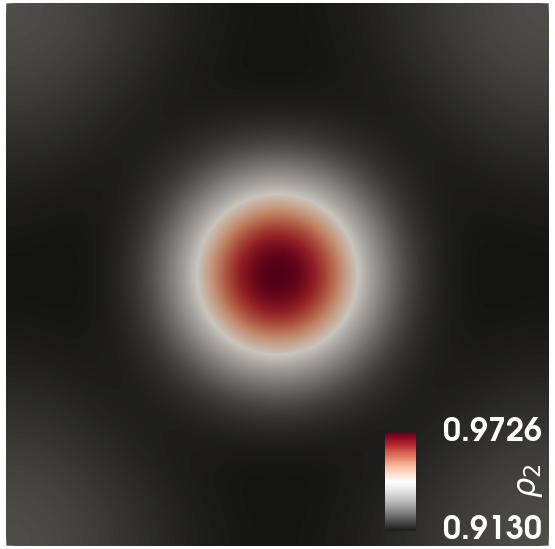}
\includegraphics[width=0.192\textwidth]{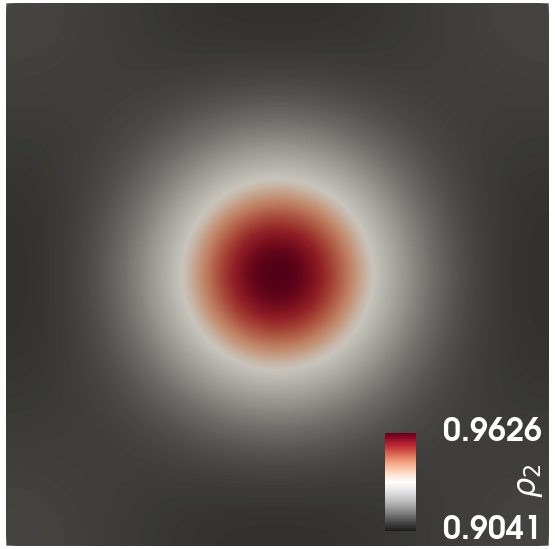}
}
\caption{Example \ref{ex5}. Snapshots of second-component density contours at different times for different $V_{1,1}(\rho)$.}
\label{fig:5b}
\end{figure}

We further plot the time evolution of the total energy 
$\mathcal{E}_{\mathrm{total}}=\mathcal{E}_{1,h}(\rho_{1,h})+
\mathcal{E}_{2,h}(\rho_{2,h})$ and total mass 
$\int_{\Omega}(\rho_{1,h}+\rho_{2,h})\, dx$ 
for the four cases in Figure~\ref{fig:5c}. 
Moreover, the total mass conservation is kept well within an error of $10^{-4}$ for all cases.

\begin{figure}[tb]
\centering
\includegraphics[width=0.48\textwidth]{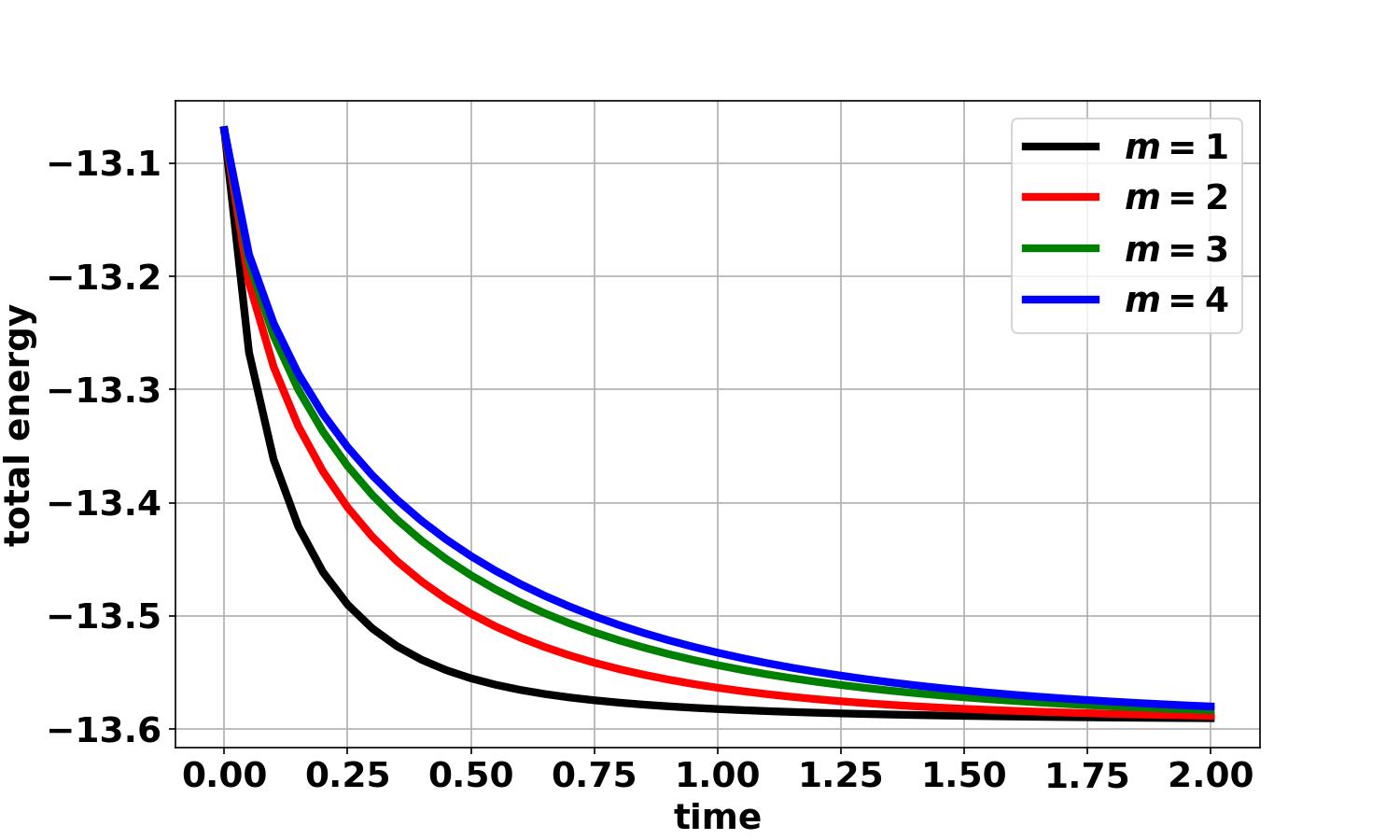}
\includegraphics[width=0.48\textwidth]{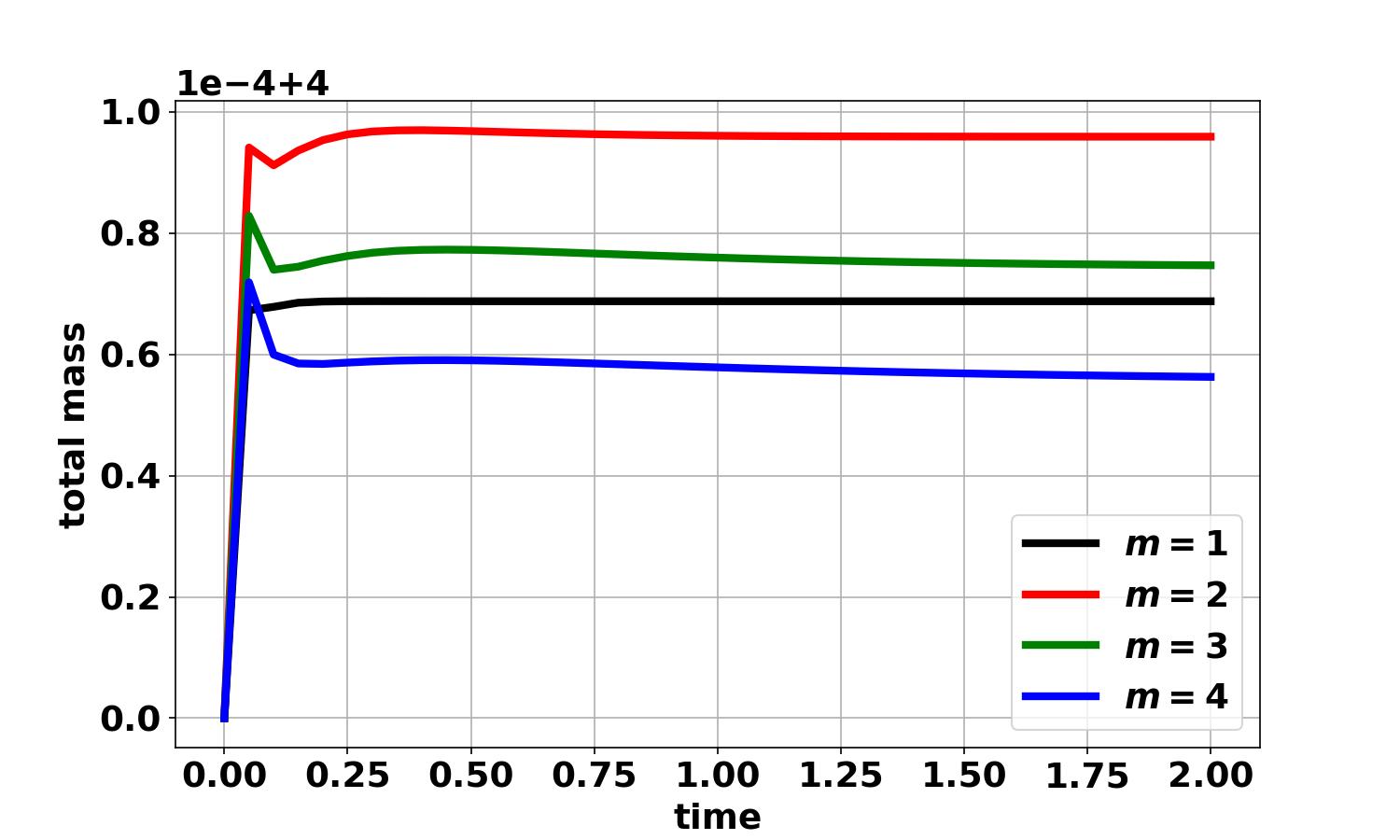}
\caption{Example \ref{ex5}. Evolution of 
total energy (left) and total mass (right) over time
with $V_{1,1}(\rho)=\gamma_1\rho^m$.}
\label{fig:5c}
\end{figure}

\subsection{Reversible Gray-Scott model}
\label{ex6}
In our last example, we simulate the 4-component 
reversible Gray-Scott model \eqref{rGS} using the Algorithm \ref{alg:3} for the fully discrete scheme 
\eqref{aug-mfp-hs} with variables/operators \eqref{rGS-full}.
The physical parameters are chosen to be the following:
\begin{align*}
\gamma_1 = 1, &\quad\gamma_2=0.01, \\
k_+^1= 1, &\quad k_-^1=10^{-3},\\ 
k_+^2= 8.4\times10^{-2},&\quad k_-^2=8.4\times10^{-5}, \\
k_+^3= 2.4\times10^{-2},&\quad k_-^3=2.4\times10^{-5}, 
\end{align*}
where the backward reaction rates are taken to be 1000 times smaller than the forward reaction rates. 
This provides a good approximation to the irreversible Gray-Scott model \eqref{irrGS}.
We consider both 1D and 2D simulations for this problem.
The initial conditions 
for the second component density $\rho_2$
is taken to be 
\begin{align*}
\rho_2(x, 0) = &\; 
\begin{cases}
0.15+\frac{1}{4}x^2(x+1)^2 &\quad \text{if } -1\le x\le 0,\\[.3ex]
0.15+\frac{1}{4}x^2(1-x)^2 &\quad \text{if } 0\le x\le 1,\\[.3ex]
0.15 & \quad \text{elsewhere},\\[.3ex]
\end{cases}
\end{align*}
in one dimension, and 
\begin{align*}
\rho_2(x, 0) = &\; 
\begin{cases}
0.15+4x^2(x+1)^2y^2(y+1)^2 &\quad \text{if } -1\le x\le 0
\text{ and }-1\le y\le 0,\\[.3ex]
0.15+4x^2(x+1)^2y^2(1-y)^2 &\quad \text{if } -1\le x\le 0
\text{ and }0\le y\le 1,\\[.3ex]
0.15+4x^2(1-x)^2y^2(y+1)^2 &\quad \text{if } 0\le x\le 1
\text{ and }-1\le y\le 0,\\[.3ex]
0.15+4x^2(1-x)^2y^2(1-y)^2 &\quad \text{if } 0\le x\le 1
\text{ and }0\le y\le 1,\\[.3ex]
0.15 & \quad \text{elsewhere},\\[.3ex]
\end{cases}
\end{align*}
in two dimensions.
The initial conditions for the other densities are taken to be 
\[
\rho_1(x,0) = 1-2\rho_2(x,0), \;\;
\rho_3(x,0) = 1, \;\;
\rho_4(x,0) = k_+^3/k_-^3=1000.
\]
For the 1D simulation, we take the computation domain to be $\Omega_{1D}=[-16, 16]$
and set the final time of simulation to be $T=1600$.
For the 2D simulation, we take 
a smaller  computational domain with $\Omega_{2D}=[-8, 8]\times[-8, 8]$ and set the final time of simulation to be $T=500$.

We apply the scheme \eqref{aug-mfp-hs}
with $k=4$ on a uniform mesh with mesh size $h=1$ (32 elements in 1D, and $16\times 16$ elements in 2D)
for both problems. 
Here we gradually increase the time step size from $\Delta t=0.01$ to $\Delta t = 0.1$ as initially taking $\Delta t = 0.1$  leads to numerical instability.
This may be caused by our splitting version of the ALG2 implementation in Algorithm \ref{alg:3}. 

We record the snapshots of the second-component density at various times in Figure~\ref{fig:X}.
For both cases, we observe pattern formations and the solution reaches a nontrivial steady state at large time.
Finally, we plot the evolution of total energy for both cases in Figure~\ref{fig:Y}, where we observe the expected monotone energy decay.

\begin{figure}[tb]
\centering
\subfigure[1D results. Left to right time: $t=200,400,800,1600$.]{
\label{fig:51x}
\includegraphics[width=0.245\textwidth]{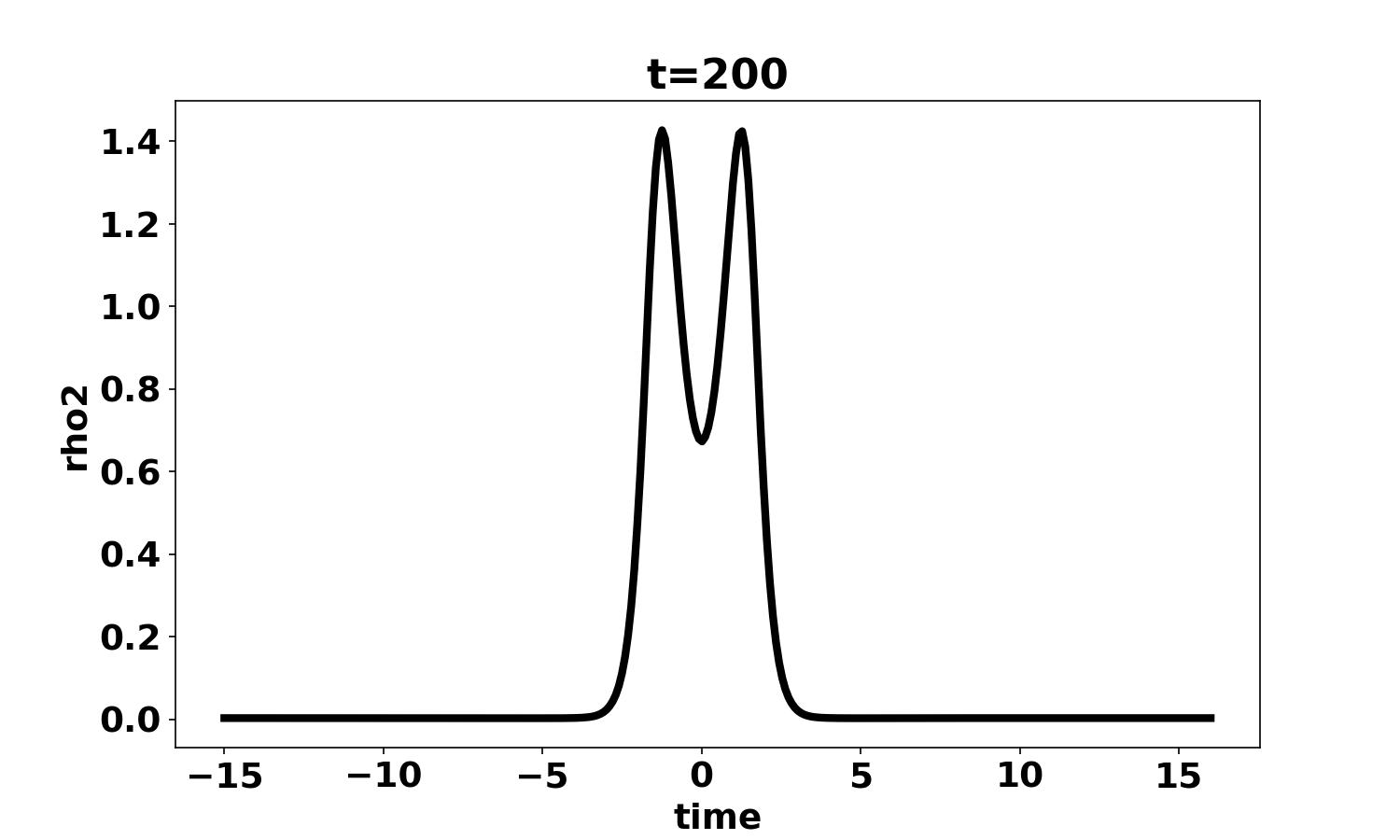}
\includegraphics[width=0.245\textwidth]{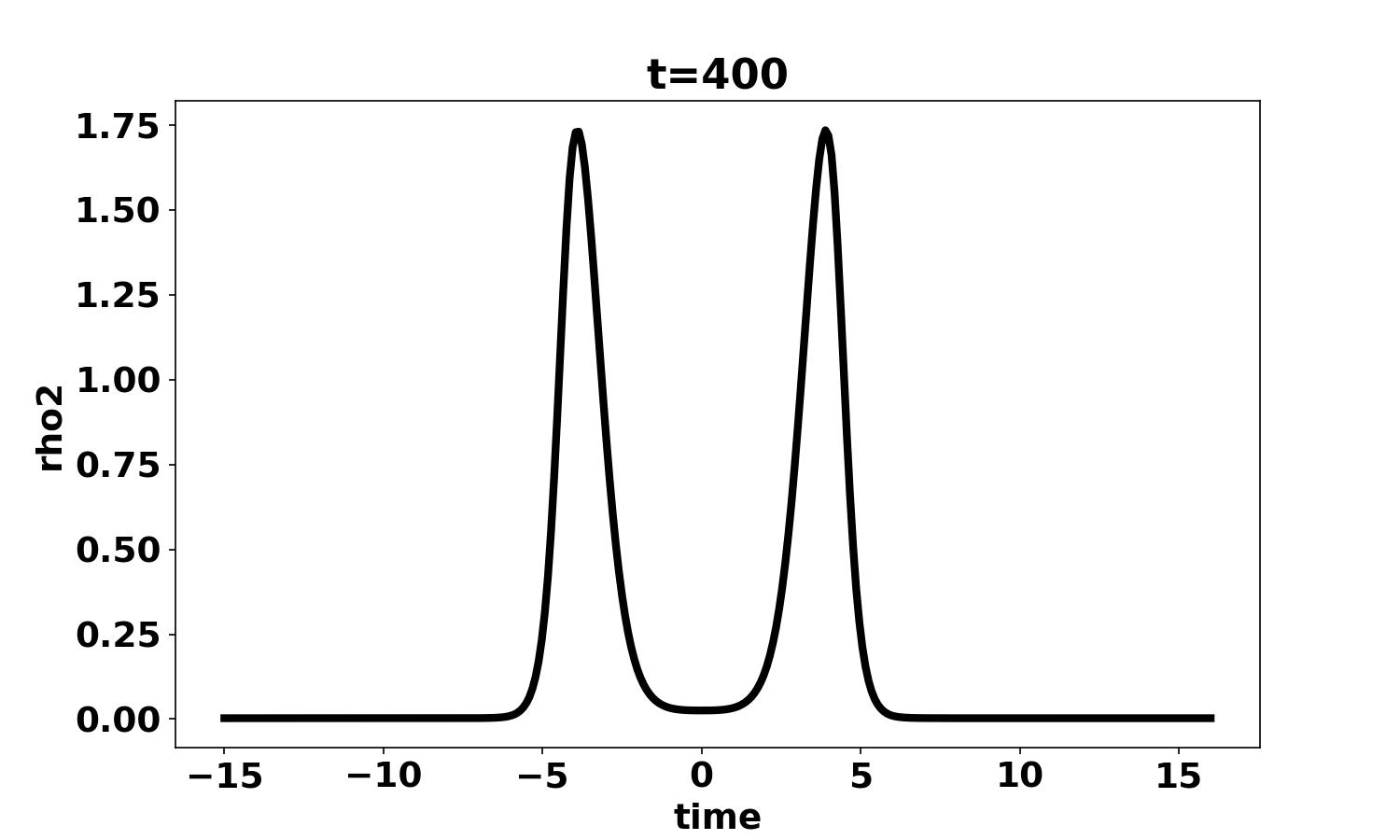}
\includegraphics[width=0.245\textwidth]{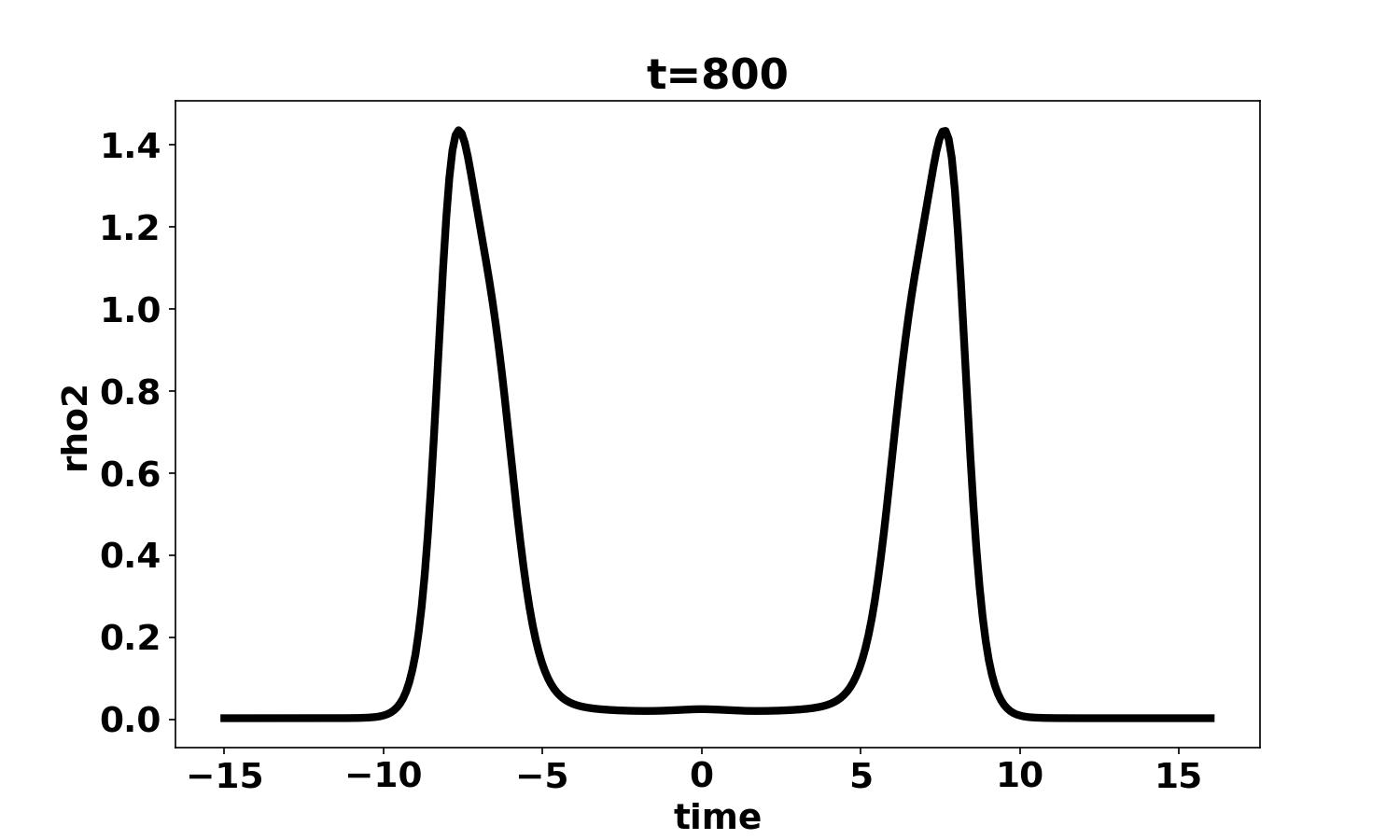}
\includegraphics[width=0.245\textwidth]{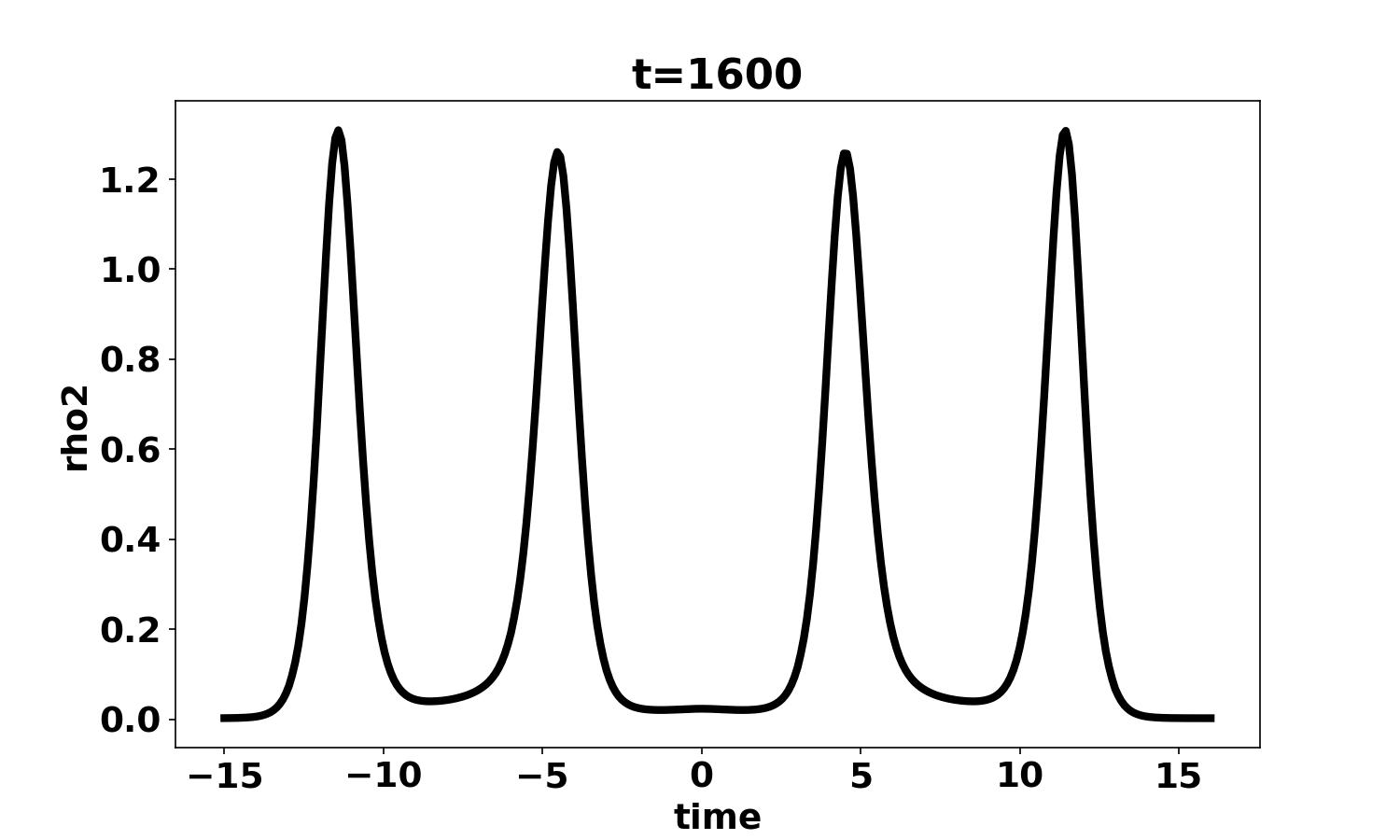}
}
\subfigure[2D results. Left to right time: $t=100,200,300,400,500$]{
\label{fig:52x}
\includegraphics[width=0.192\textwidth]{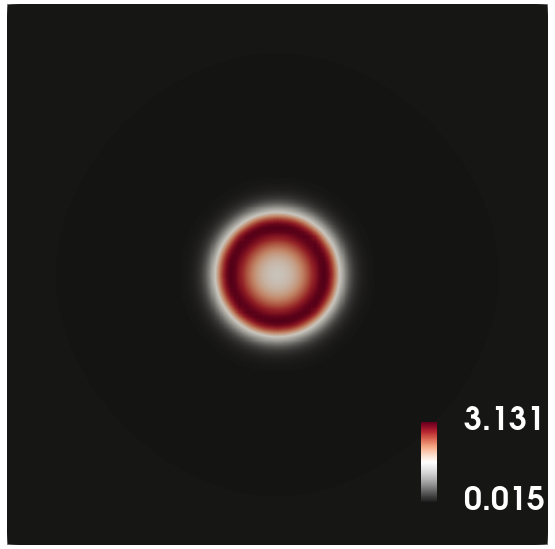}
\includegraphics[width=0.192\textwidth]{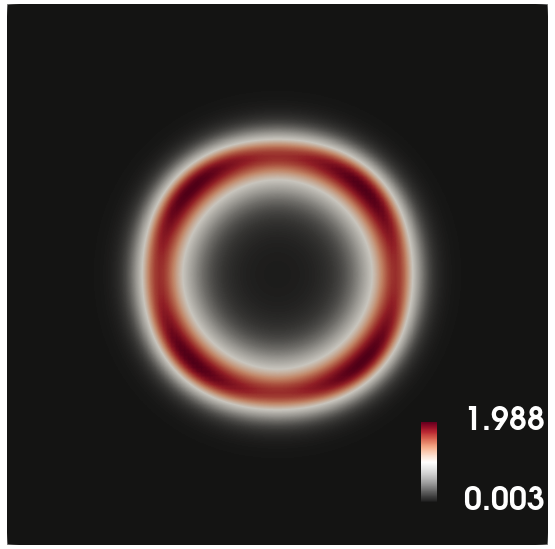}
\includegraphics[width=0.192\textwidth]{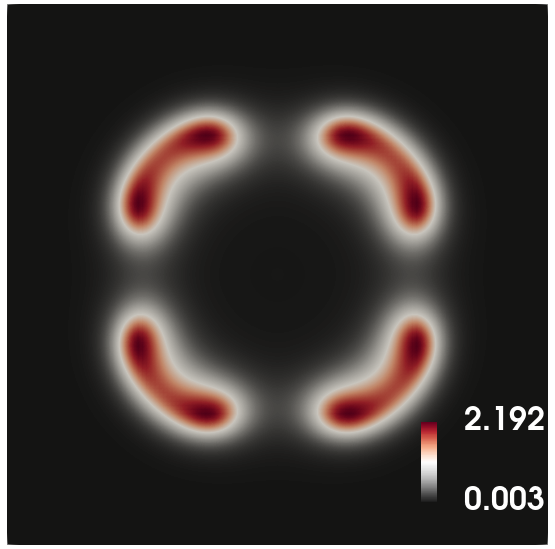}
\includegraphics[width=0.192\textwidth]{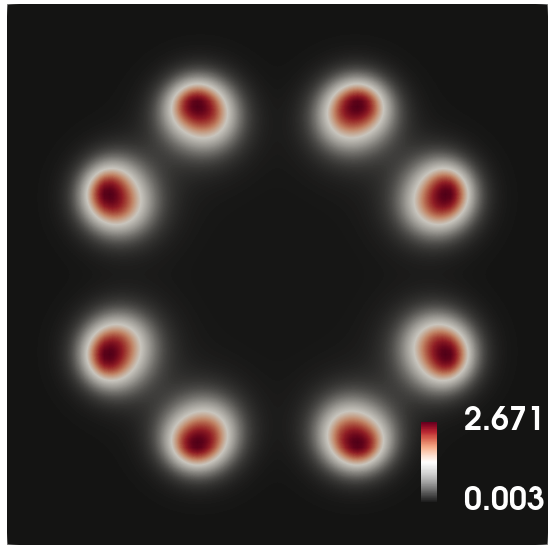}
\includegraphics[width=0.192\textwidth]{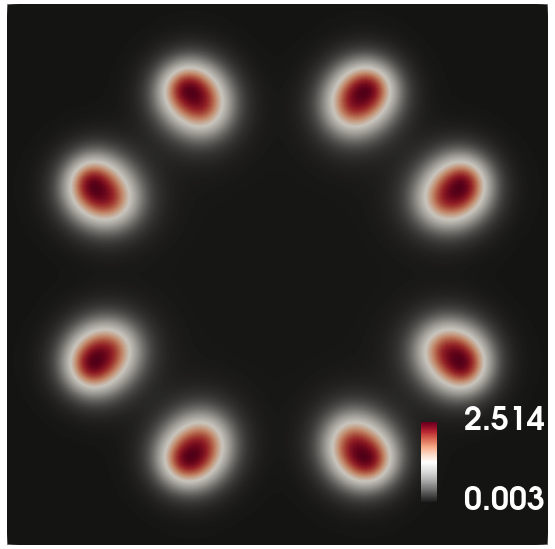}
}
\caption{Example \ref{ex6}. Snapshots of second-component density contours at different times for 1D (top) and 2D (bottom) simulations.}
\label{fig:X}
\end{figure}

\begin{figure}[tb]
\centering
\includegraphics[width=0.48\textwidth]{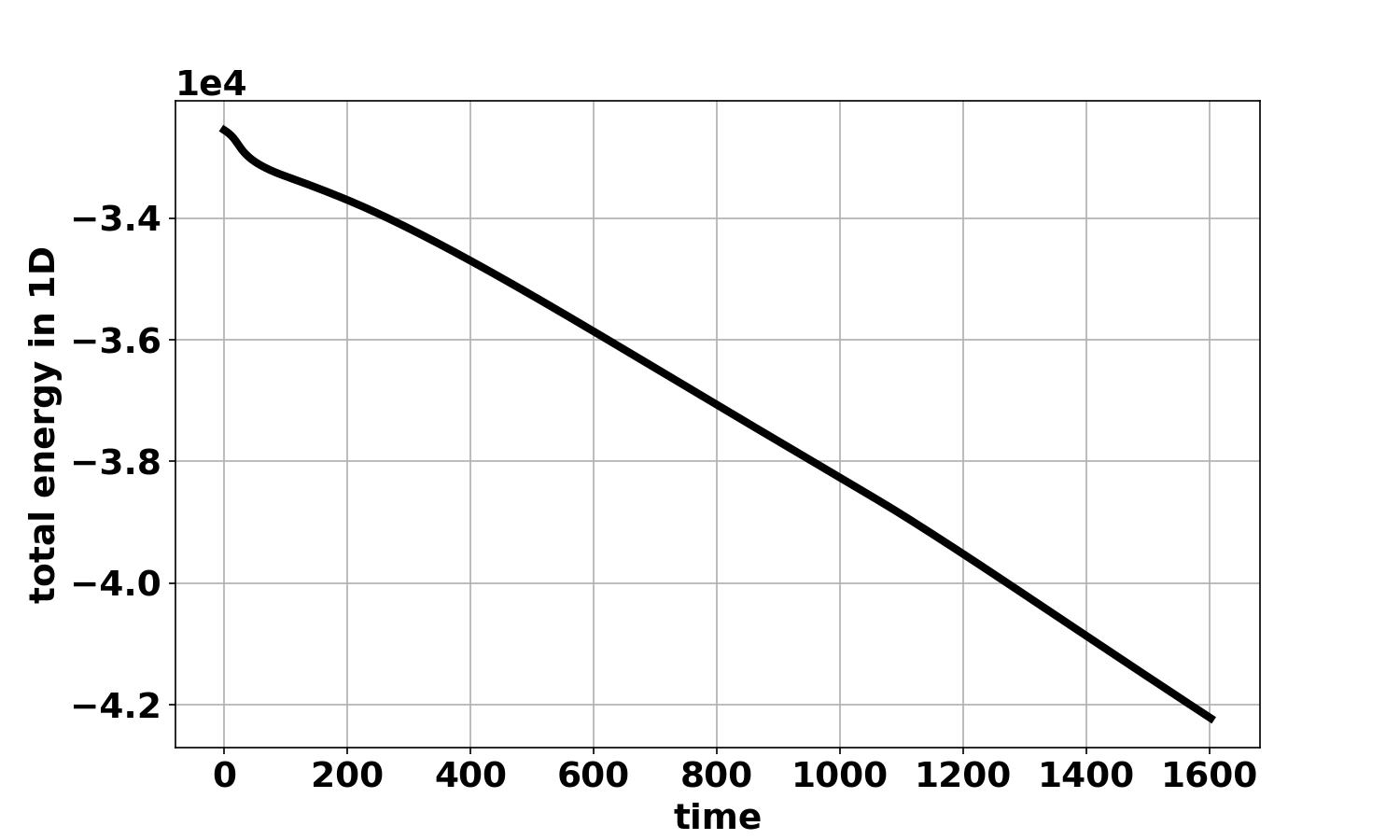}
\includegraphics[width=0.48\textwidth]{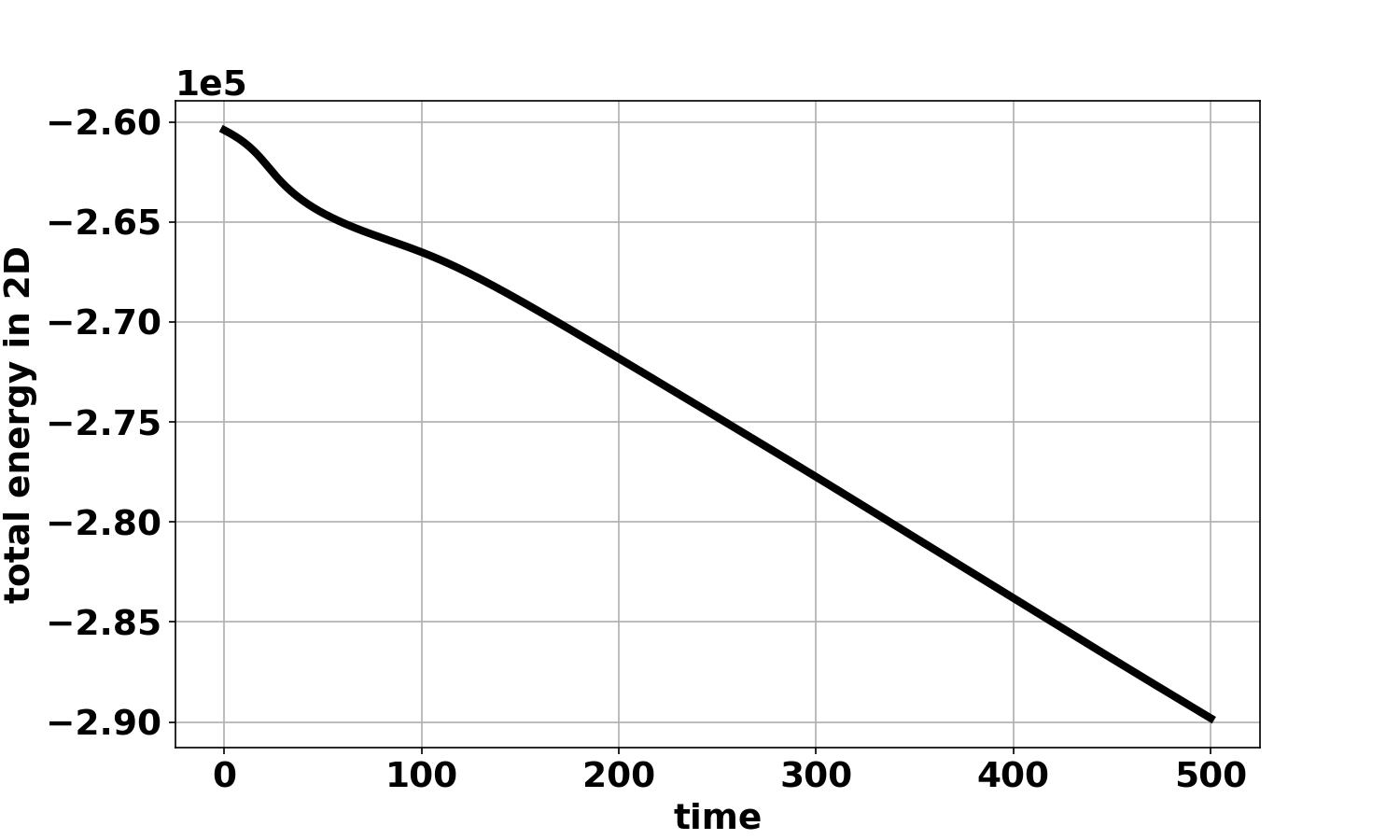}
\caption{Example \ref{ex6}. Evolution of 
total energy in 1D (left) and 2D (right).}
\label{fig:Y}
\end{figure}

\section{Conclusion}
\label{sec:summary}
This paper applies high-order accurate finite element methods in space to compute first-order accuracy implicit-in-time  gradient flows. Our formulation applies a one-step time discretization of the generalized JKO scheme and then uses the ALG2 to calculate optimization problems in each generalized JKO time step. The method is unconditionally stable when the optimization problem is convex. Numerical experiments in two-dimensional gradient flow dynamics, such as Wasserstein gradient flows, Fisher--KPP dynamics, and reversible reaction-diffusion systems, demonstrate the effectiveness of the proposed method with high-order spatial accuracy.

We note that for dissipative dynamics, such as strongly reversible reaction-diffusion systems, different entropies $\mathcal{E}$, and optimal transport-type metrics $V_1$, $V_2$, could produce the same evolutionary equation. In simulations, we suggest selecting a suitable class of entropies and metrics to develop simple and efficient optimization procedures. Some limitations exist for computing implicit-in-time gradient flows in generalized optimal transport metric spaces. The constructed functions $V_1$ and $V_2$ should be nonnegative for entropy dissipation schemes. Our generalized JKO scheme is unstable for many reaction-diffusion equations, e.g., the Allen-Cahn-type equations \cite{shen10X}. We also remark that the current computations are limited to the first--order time accuracy variational-implicit schemes of gradient flows. In future work, we shall design and compute generalized optimal transport and mean field control problems for implicit-in-time fluid dynamics with general conservative-dissipative formulations. Typical examples include regularized conservation laws \cite{li2021controlling, li2022controlling}.

\end{document}